\documentclass{amsart}

\usepackage{amsfonts,amssymb,latexsym,xspace} 

\usepackage{wick}

\def\cal{\mathcal}
\def\frak{\mathfrak}
\def\C{\mathbb{C}}
\def\ds{\displaystyle}
\def\nor{\operatorname{norm}}
\def\redu{~\longrightarrow~}
\def\reda{\longrightarrow}
\def\ul{\underline}
\def\into{\longrightarrow}

\def\Z{\mathbb{Z}}
\def\C{\mathbb{C}}
\def\K{\mathbb{K}}
\def\Q{\mathbb{Q}}

\def\De{\Delta}
\def\seq{\operatorname{\hbox{\sc Seq}}}
\def\set{\operatorname{\hbox{\sc Set}}}
\def\cyc{\operatorname{\hbox{\sc Cyc}}}

\usepackage{color}

\definecolor{light-gray}{gray}{0.5}

\def\implies{\Longrightarrow}

\newcommand{\minip}[2]{\begin{minipage}{#1cm}{#2}\end{minipage}}
\newcommand{\mc}[1]{\multicolumn{1}{c}{#1}}

\newtheorem{theorem}{Theorem}
\newtheorem{definition}{Definition}

\usepackage{graphicx}
\newcommand{\Img}[2]{\begin{tabular}{c}\includegraphics[width=#1truecm]{#2}\end{tabular}}
\newcommand{\img}[2]{\includegraphics[width=#1truecm]{#2}}

\newcounter{noteno}
\newenvironment{note}{%
        \begin{small}\refstepcounter{noteno}\smallbreak\noindent{\bf Note~\thenoteno.}}%
        {\dotfill\hbox{~$\blacksquare$}\par\medbreak\end{small}\smallbreak}

\def\hb{\hbar}

\newtheorem{proposition}{Proposition}

\def\z{\cal Z}

\def\nor{{\operatorname{{\mathfrak N}}}}
\newcommand{\stirc}[2]{\genfrac{[}{]}{0pt}{}{#1}{#2}}
\newcommand{\stirp}[2]{\genfrac{\{}{\}}{0pt}{}{#1}{#2}}
\def\Pr{\mathbb{P}}
\def\Ex{\mathbb{E}}
\newcommand{\parag}[1]{\par\smallskip\noindent{\bf\em #1}}

\def\hat{\widehat}

\usepackage{bbold}
\def\ct{\operatorname{\hbox{\bf C.T.}}}
\def\one{\mathbb{1}}
\def\tilde{\widetilde}
\def\lp{\operatorname{\cal L}}
\def\sm{\operatorname{sm}}

\def\Y{{\mathbb{Y}}}
\def\cal{\mathcal}
\def\zz{{\mathcal Z}}
\newcommand{\OEIS}[1]{\emph{OEIS~{\bf #1}}}

\begin{document}


\title[Combinatorial Models of Creation--Annihilation]{Combinatorial Models of Creation--Annihilation}

\author[P. Blasiak]{Pawel Blasiak}

\address{\newline H. Niewodnicza\'nski Institute of Nuclear Physics
\newline Polish Academy of Sciences
\newline ul.\ Radzikowskiego 152, PL-31342 Krak\'ow, Poland
\vspace{0.2cm}}

\email{Pawel.Blasiak@ifj.edu.pl}
\urladdr{www.ifj.edu.pl/~blasiak}

\author[P. Flajolet]{Philippe Flajolet}

\address{\newline Algorithms Project
\newline INRIA-Rocquencourt
\newline F-78153 Le Chesnay, France
\vspace{0.2cm}}

\email{Philippe.Flajolet@inria.fr}
\urladdr{algo.inria.fr/flajolet}


\date{}

\begin{abstract}
Quantum physics has revealed many interesting formal properties associated with
the algebra of two operators, $A$ and $B$, satisfying the partial commutation 
relation~$AB-BA=1$. This study surveys the relationships between 
classical combinatorial structures and the reduction to
normal form of operator polynomials in such an algebra. The connection
is  achieved 
through suitable labelled graphs, or ``\emph{diagrams}'', that are composed of elementary ``\emph{gates}''.
In this way, many normal form evaluations can be systematically obtained,
thanks to models that involve set partitions, permutations, increasing trees, 
as well as weighted 
lattice paths. Extensions to
$q$-analogues, multivariate frameworks, and urn models are also briefly discussed.
\end{abstract}

\maketitle

\thispagestyle{myheadings}
\font\rms=cmr8
\font\its=cmti8
\font\bfs=cmbx8

\markright{\its S\'eminaire Lotharingien de
Combinatoire \bfs 65 \rms (2011), Article~B65c\hfill}
\def\thepage{}

\begin{small}
\setcounter{tocdepth}{2}
\tableofcontents
\end{small}

%

\bigskip

\section{\bf Introduction}

The theme of our study is extremely simple. 
It consists in 
investigating some of the formal consequences of the  partial commutation relation
between two operators~$A$ and~$B$ (belonging to some unspecified algebra of operators):
\begin{equation}\label{cran}
[A,B]=1.
\end{equation}
Here $[U,V]:=UV-VU$ is the Lie bracket and~$1$ represents the identity operator.
The relation~\eqref{cran} will henceforth be called the 
\emph{creation--annihilation axiom}. Indeed, in quantum
physics\footnote{
The present paper deals with one or a finite number of pairs of such
operators (known as
``modes''), a situation that is directly relevant to the ``second quantization'' 
of quantum theory, in particular, in quantum optics; see, for instance,
Louisell's book~\cite{Louisell90}. The case of a continuum of modes,
which lies at the basis of quantum field
theory and is associated with Feynman diagrams~\cite{Feynman62,Feynman88,Mattuck92,Weinberg95a},
is not touched upon here. The developments in our study require a modicum
of combinatorial theory but no \emph{a priori} knowledge of
quantum physics.}, 
it is satisfied by the annihilation and creation operators, $a$ and $a^\dagger$,
which are adjoint to each other and serve to decrease or increase the number of photons by~1:
in this context, one should take $A=a$ and $B=a^\dagger$.

From an algebraic standpoint, 
one thus considers the polynomial ring~$\C\langle A,B\rangle$ in non-commuting
indeterminates~$A,B$, and makes  systematic use of the reduction
of  $AB-BA-1$ to~0.
(Technically, this corresponds to taking the quotient 
$
\C\langle A,B\rangle / \cal I$,
of the ring $C\langle A,B\rangle$ of polynomials in
non-commuting indeterminates by the two-sided ideal $\cal I$ generated by $AB-BA-1$.)
It is then possible to regard~\eqref{cran} as a directed rewrite rule
\begin{equation}\label{red}
AB\redu 1+BA,
\end{equation}
by which any non-commutative polynomial $\frak{h}$ in indeterminates~$A,B$ 
is completely reduced to a unique \emph{normal form}, $\nor(\frak{h})$, where, 
in each monomial, all the occurrences of~$B$
precede all the occurrences of~$A$. 

\pagenumbering{arabic}
\addtocounter{page}{1}
\markboth{\SMALL P. BLASIAK AND P. FLAJOLET}{\SMALL COMBINATORIAL MODELS
OF CREATION--ANNIHILATION}

For instance, the chain
\begin{equation}\label{exa0}
\renewcommand{\arraycolsep}{3truept}
\begin{array}{lll}
AB\underline{AB}A &\reda& AB(1+BA)A \equiv ABBAA+\ul{AB}A \\
&\reda&  \ul{AB}BAA+ A+BAA \redu B\ul{AB}AA+A+2BAA \\
&\reda& BBAAA+3BAA+A
\end{array}
\end{equation}
proves that $\nor(ABABA)=B^2A^3+3BA^2+A$.
(At each step, the pair~$\ul{AB}$ involved in a reduction has been underlined.)
It is precisely this non-trivial rearrangement process, known in quantum physics as  
\emph{normal ordering}, which we propose to examine.

A particular realization of the commutation relation~\eqref{cran}
is obtained by choosing some sufficiently general space~$\{f(x)\}$ of smooth functions
(typically, the class of $\cal C^\infty(0,1)$ 
of infinitely differentiable functions over the interval $(0,1)$,
or the space $\C[x]$ of polynomials in indeterminate~$x$),
on which two operators, $X$ and~$D$ are defined:
\[
(Xf)(x)=xf(x); \qquad
(D f)(x)=\frac{d}{dx} f(x).\]
Then the creation--annihilation axiom is obviously satisfied by~$A=D$ and $B=X$: 
one recovers in this way the familiar Weyl relations $[D,X]=1$ of abstract differential algebra~\cite[Ch.~1]{Ritt50}.
The interest of such a differential model of creation--annihilation is that it is ``\emph{faithful}'', meaning that
any identity that holds in it (without \emph{any} additional
assumption
regarding the space of functions) is true 
in all generality under~\eqref{cran}.
This differential view will prove central to our combinatorial developments.

We henceforth adopt the more suggestive differential terminology,
$A\mapsto D$ and $B\mapsto X$. 
We shall concern ourselves here with the reduction to normal form 
of a variety of expressions such as
\begin{equation}\label{exa}
(XD)^n, \quad (X^2D)^n, \quad
(D+X)^n, \quad
(X^2D^2)^n, \quad (D^2+X^2)^n,
\end{equation}
and many more. Observe that, by taking an \emph{exponential generating function}, 
the collection of all the reductions associated with a family $(\frak H^n)_{n\ge0}$ is summarized
by the reduction of $e^{z\frak H}$.

Our starting point, to be developed in Section~\ref{maindef-sec},
is a combinatorial representation of
any normal-form monomial $\frak{m}=X^r D^s$ by a basic graph, called a ``\emph{gate}''. 
The reduction of a polynomial in operators $D,X$ (for instance, any of~\eqref{exa}) can then be regarded
combinatorially as the process of building a whole collection of \emph{labelled diagrams}, which are those 
special \emph{graphs} 
assembled from a fixed collection of elementary gates (the ones arising from the monomials
in~\eqref{exa}). 
From here, as we shall see repeatedly,
\emph{obtaining the coefficients in the normal form of 
an operator expression is equivalent to enumerating complex graphs built from
a fixed collection of gates}. In other words, we are going to
explore the following \emph{Equivalence Principle} (see
Theorem~\ref{eqp-thm} below
for a precise formulation):

\[
\hbox{\fbox{\minip{3.5}{\begin{center}{\bf Normal ordering}\\(operator algebra)\end{center}}}} \quad\Longleftrightarrow\quad 
\hbox{\fbox{\minip{4.1}{\begin{center}{\bf Diagram  enumeration}\\(combinatorics)\end{center}}}}\,
\]

\smallskip

\noindent
In so doing, we build on classical works in combinatorial analysis
relative to the combinatorics of differential operators; see, for instance,
Joyal's theory of species~\cite{Joyal81}, its exposition in the book by Bergeron, Labelle
and Leroux~\cite{BeLaLe98}, or the recent book by Flajolet and Sedgewick~\cite{FlSe09},
whose notations we adopt. 

The idea of using graphs to model creation--annihilation is not new.
The most celebrated originator of the representation of  physical processes by labelled graphs 
is R.P.~Feynman, who used it as a convenient book-keeping tool 
for perturbative expansions of quantum
electrodynamics~\cite{Feynman62,Feynman88} and
the whole idea is is neatly articulated
by Baez and Dolan in~\cite[pp.~46--49]{BaDo01}. 
Seen under  the interesting perspective of~\cite{BaDo01}, what we
propose to do is ``decategorize'' the abstract creation--annihilation theory
by providing concrete combinatorial models of this theory. Our standpoint is however
different in that we place a strong emphasis on 
connections with classical combinatorial structures and aim at
developing exactly solvable models, hence explicit formulae.
With this in mind, we tread in the steps of earlier works of Blasiak, Duchamp, Horzela,
M\'endez, Penson, and Solomon~\cite{Blasiak05,Blasiak10b,BlDuSoHoPe10,BlHoPeSo06,BlHoPeSoDu07,BlPeSo03b,MeBlPe05}
for which the present article can serve \emph{inter alia}
as a synthetic and systematic review.

Our purpose is  to shed  light on calculations developed within quantum 
physics and  do so by means of adapted combinatorial models.
For instance, set partitions (with their associated Bell and Stirling numbers)
naturally  arise from diagrams corresponding to powers of $XD$,
a fact largely known in combinatorics and finite difference calculus,
which is examined, under a quantum field angle, 
by  Bender, Brody, and Meister in~\cite{BeBrMe99},
and is further treated by M\'endez, Blasiak, and Penson in~\cite{MeBlPe05}.
Also, by providing a complete permutation model,
we explain and extend calculations done by Mehta~\cite{Mehta77}
that are relative to the normal ordering of
\[
e^{zQ(X,D)}, \qquad\hbox{where}\quad Q(X,D)=\alpha D^2+\beta X^2 +\gamma XD
\]
Finally, the combinatorial approach we advocate provides, in a number of cases, 
a transparent alternative to the Lie algebra approach exemplified by Wilcox~\cite{Wilcox67}.
Figure~\ref{recap-fig} lists a representative
set of operators discussed in this paper, 
together with the underlying combinatorial structures and the 
corresponding generating function types, as summarized
in each case by a prototypical instance. 

\begin{figure}\small
\begin{center}\def\mc{}\renewcommand{\arraystretch}{1.3}
\begin{tabular}{llll}
\hline\hline
\mc{$\frak{H}^n$} & \mc{\emph{Structure}} & \mc{\emph{Generating function type}} & \\
\hline\\&&&\vspace{-0.8cm}\\
$(X+D)^n$ & involutions &  $\ds e^{z^2/2}$ & \S\ref{lin-sec}  \\
$(XD)^n$ & set partitions & $\ds e^{e^z-1}$ & \S\ref{bell-sec} \\
$(X^2+D^2)^n$ & alternating cycles& $\ds \frac{1}{\sqrt{\cos(2z)}}$ & \S\ref{zig-subsec}\vspace{0.1cm}\\
$(X^rD)^n$ & increasing trees & $\ds \exp\left(\frac{1}{(1-(r-1)z)^{1/(r-1)}}-1\right)$ & \S\ref{semilin-sec}
\vspace{0.1cm}\\
$(X^2D)^n$ & permutations & $\ds \exp\left(\frac{z}{1-z}\right)$ & \S\ref{incr-subsec}\vspace{0.1cm}\\
\hline\hline
\end{tabular}
\end{center}

\caption{\label{recap-fig}\small Some normal ordering problems relative to~$\frak{H}^n$,
the corresponding combinatorial structures, a representative 
counting generating function, and the
relevant section in this article.}
\end{figure}

\parag{Plan of the paper.} 
Section~\ref{maindef-sec}  presents the  basic combinatorial model  of
gates and diagrams, by which normal ordering  is reduced to counting
problems.  It also contains a  brief reminder of the symbolic approach
to     combinatorial     enumeration     via   generating    functions
(\S\ref{combana-subsec}).  We   then  consider  the  normal   ordering
problem  relative to $\frak{h}^n$  in increasing order of complexity of
the   polynomial    $\frak{h}$.   Section~\ref{lin-sec}  
concerns the simplest case of
 linear  forms in~$X,D$, for which   the natural model  is seen to be
that  of  (coloured) \emph{involutions}.   In Section~\ref{bell-sec},
we  proceed with the  special
quadratic form~$(XD)$, which is tightly coupled with \emph{set partitions}.
The general quadratic form in~$X, D$ treated in Section~\ref{quad-sec}
leads  to
\emph{``zigzags''}, which are structures related to 
\emph{alternating permutations}, and to local order patterns
in (unconstrained) \emph{permutations}. Section~\ref{semilin-sec} 
shows that semilinear operators of the form $(\phi(X)D)$ 
are in general modeled by simple varieties of \emph{trees}
and \emph{forests}. Section~\ref{cf-sec} momentarily departs from our main thread
and introduces \emph{lattice path} models:
it presents a direct treatment of binomial forms, including the
interesting ``Fermat case'' $(X^r+D^r)$, which is conducive to
continued fraction representations.
Section~\ref{frameworks-sec}  discusses two frameworks that are
closely related to diagrams: 
one is the rook placement model, for which a general scanning algorithm schema
 provides an alternative to the methods of Section~\ref{cf-sec};
the other is the planar embedding  of diagrams, which leads to a systematic
approach to many $q$--analogues of combinatorial theory.
We conclude in Section~\ref{mult-sec} with a brief discussion
of multivariate extensions of the gates-and-diagram model.

The present paper contains few new results. Rather, it is an attempt at
a synthetic presentation of the combinatorial approach to normal
ordering problems, which has been the subject of an extensive literature
in recent decades. 
What we found striking,
when preparing this paper,  
is the fact that \emph{most (perhaps all?) explicitly known expansions 
in this orbit of questions can be put in correspondence with 
classical or semi-classical combinatorial models}.
As we hope to demonstrate, one of the interests of the combinatorial approach 
in this range of problems  is to bring clarity into 
what should be explicit and what is not likely to be.

\section{\bf Diagrams, normal ordering, and enumeration} \label{maindef-sec}

In   this section, we  associate with  each  polynomial in the operators
$X,D$  a finite collection   of elementary  ``\emph{gates}'' that  are
 graphs   with just one inner vertex, as in~\eqref{simpf} below.   We  then    define
``\emph{diagrams}''  that are   complex  graphs  built by   assembling
elementary gates  according   to certain rules.  (Our   terminology is
inspired by that of  digital circuits composed of elementary  boolean
gates.)  A fundamental and easy theorem (Theorem~\ref{eqp-thm})
then relates the normal ordering of the
powers  $\frak{h}^n$ of a  polynomial~$\frak{h}$  to the enumeration of labelled diagrams
of size~$n$ that can be built  out of the gates associated with~$\frak{h}$.

\subsection{Gates, diagrams and labelling.} \label{diag-subsec}
We make use of the standard definitions of graph theory, as found, for
instance,   in~\cite{Diestel00}.  We  however   need   to   extend  it
somewhat. Technically, our diagrams 
are directed multigraphs---edges are directed, multiple edges 
between two vertices are allowed---that
are further enriched by placing  an ordering on the edges  stemming 
from   a  vertex 
and similarly for edges leading   into a vertex.  
(Put otherwise, we are considering graphs given together with an
\emph{embedding} into the plane.) 
We recall that the \emph{indegree} and \emph{outdegree} of a vertex
are, respectively, the number of incoming and outgoing edges of that vertex. In   what
follows, we freely make use of graphical representations to illustrate
our basic notions, while avoiding long formal developments; see Figure~\ref{diags0-fig}.

\begin{definition} 
A  \emph{gate}  of  type $(r,s)$   consists of  a   vertex, called the
\emph{inner node}, to  which is attached  an ordered collection of  $r$ outgoing edges
and an ordered collection   of $s$ ingoing edges.  
The vertices, not the inner node, that have indegree~$0$
are called the \emph{inputs};
 the vertices, not the inner node, that have 
outdegree~$0$ are called the \emph{outputs}.
\end{definition}

A gate thus comprises $r+s+1$ nodes in total.
The inner node is conventionally coloured in black ($\bullet$), whilst the input and output nodes are coloured in grey (${{\color{light-gray}\blacktriangle}\!\!\!\!\!\!\vartriangle}$) and white ($\vartriangle$) respectively. For instance:\vspace{0.1cm}
\begin{equation}\label{simpf}
\hbox{gate of type $(r,s)$} : \quad \hbox{\Img{2.1}{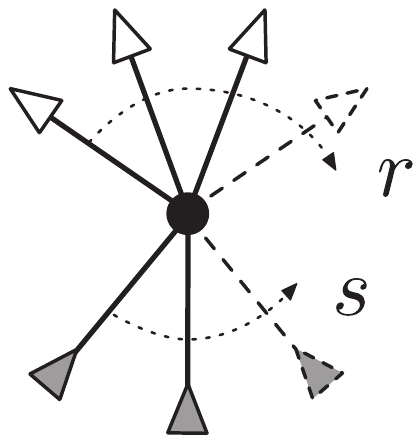} }\qquad \begin{array}{r}
\hbox{$r$ outputs}\\ \ \\ \ \\ \ \hbox{$s$ inputs.}\end{array}
\end{equation}
In the graphical representation,
edges are systematically (and implicitly) oriented from bottom to top;
they are naturally ordered among themselves, conventionally,  from left to right.

\begin{figure}
\begin{center}\renewcommand{\tabcolsep}{2truept}
\begin{tabular}{cc}
\begin{tabular}{c} $\cal H : {} $ \\[0.2truecm] \fbox{\Img{3.1}{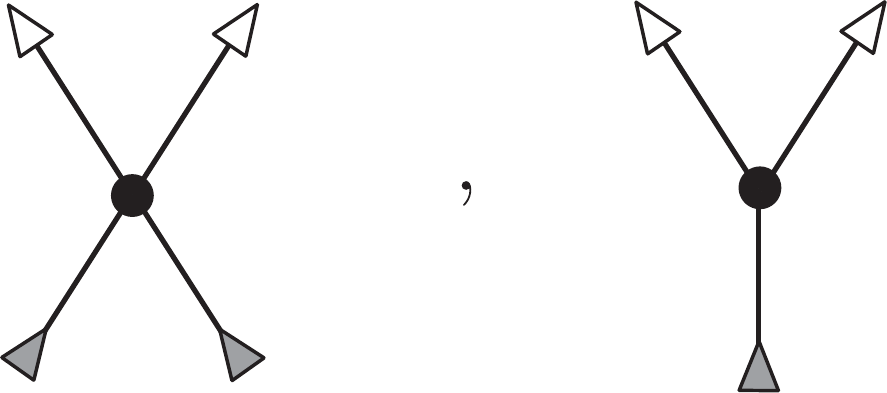} }\end{tabular} &
\begin{tabular}{c} $\delta :{}$ \\[0.2truecm] \fbox{\,\Img{7.9}{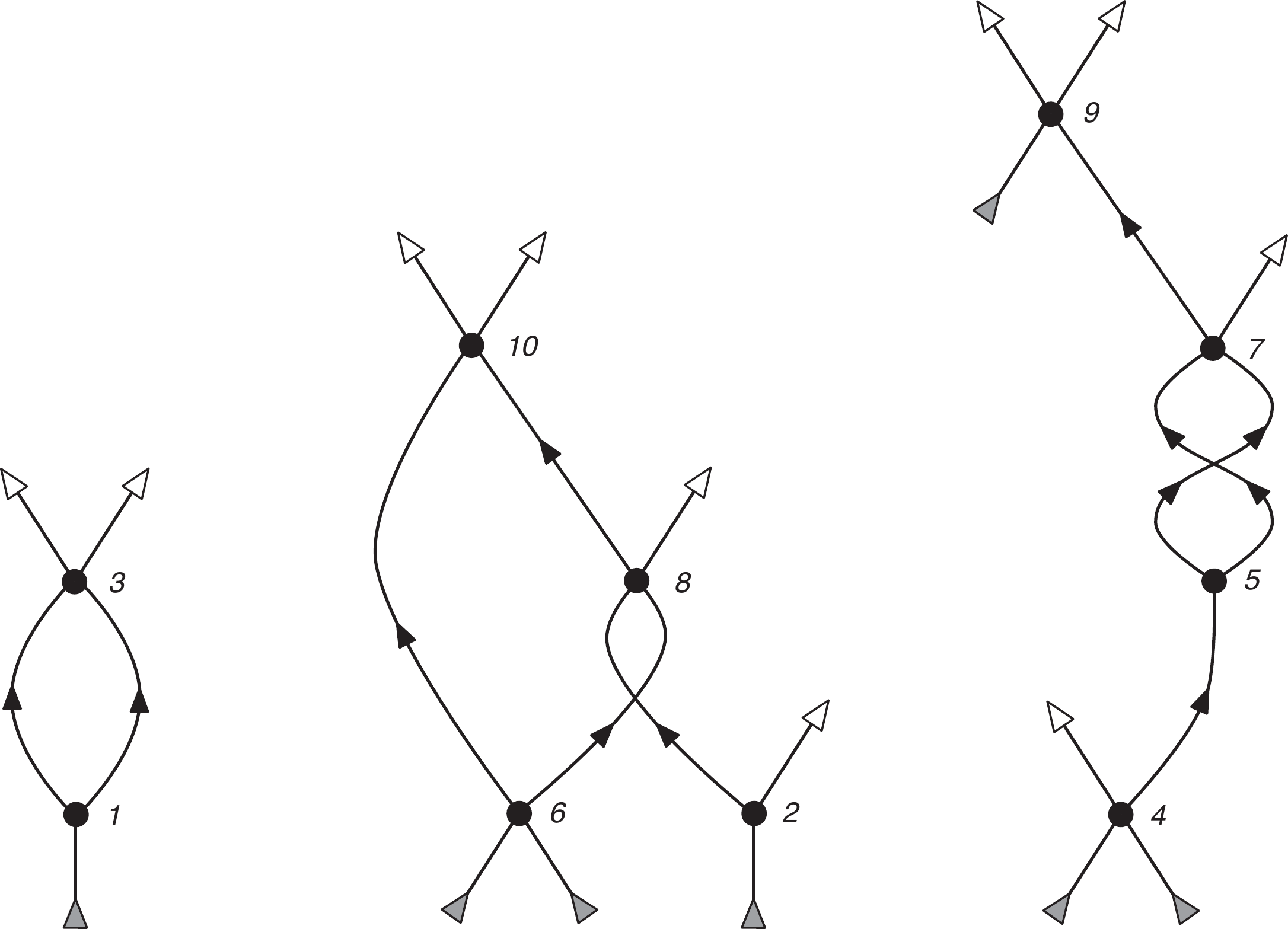}\, }\end{tabular}
\end{tabular}
\end{center}

\vspace*{-2truemm}
\caption{\label{diags0-fig}\small
Left: a collection $\cal H$ of two gates of type $(2,2)$ and $(2,1)$. 
Right: a labelled diagram~$\delta$ on the basis~$\cal H$ that comprises three components.}
\end{figure}

\begin{definition} \label{unlabdiag}
A  \emph{diagram}  is a directed multigraph, which  is \emph{acyclic}
(i.e., has no directed cycle) and is such that, for each vertex, 
an ordering has been fixed on its incoming
and on its outgoing edges. In addition, there are  a designated set of
\emph{inputs} (vertices with indegree~$0$ and outdegree~$1$) and  a designated set of 
\emph{outputs}
(vertices with indegree~$1$ and outdegree~$0$).
The vertices that are neither inputs nor outputs are called
\emph{inner nodes}.
The \emph{size} of a diagram is the number of inner nodes\footnote{
In what follows, we shall discount from size outer nodes ${{\color{light-gray}\blacktriangle}\!\!\!\!\!\!\vartriangle}$ and $\vartriangle$ (which have
size~0), and then often use the term ``vertex'' and ``inner node'' (or simply ``node'') interchangeably.}
 it comprises.
\end{definition}

The colouring convention of gates, as in~\eqref{simpf}, 
is extended to the representation of diagrams;
see Figure~\ref{diags0-fig} for an instance.
Note that a diagram is \emph{not} necessarily connected: it may be comprised
of several (weakly connected) \emph{components}. 
Clearly, a diagram specifies the collection of gates it is composed of.


\begin{definition}
Given a set $\cal H$ of gates, assimilated to a subset of $\Z_{\ge 0}\times\Z_{\ge0}$,
a diagram~$\delta$ is said to have~$\cal H$ as a \emph{basis} if, for each 
outdegree--indegree pair~$(r,s)$ of an inner node of~$\delta$, one has $(r,s)\in\cal H$.
\end{definition}

Finally, we need to introduce a crucial notion of labelling,
which is 
consistent with the standard one of 
combinatorial analysis~\cite{BeLaLe98,FlSe09,GoJa83,Stanley99,Wilf94},
to which an additional \emph{monotonicity} constraint is adjoined.
\begin{definition} \label{labdiag}
A \emph{labelled diagram} is comprised of
an unlabelled  diagram in the sense of Definition~\ref{unlabdiag} together with an assignment of integer labels 
to inner nodes in such a way that: 
$(i)$~all labels are distinct and they form an initial segment of $\Z_{\ge1}$;
$(ii)$~labels  \emph{increase} along any directed path.
\end{definition}
\begin{figure}
\begin{equation*}
\begin{tabular}{c}\hbox{\setlength{\unitlength}{1truecm}
\begin{picture}(13,4.5)
\put(-0.2,0){\img{12.5}{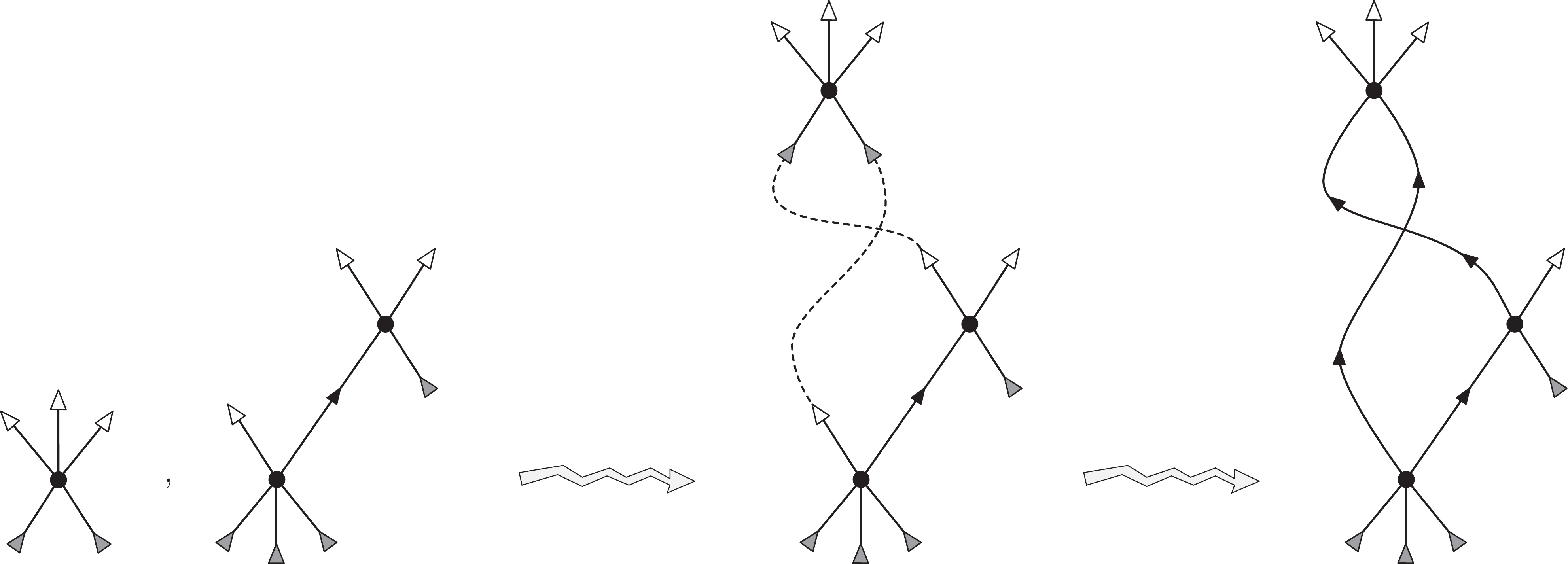}}
\put(-0.1,1.8){$\delta :$}
\put(1.5,1.8){$\gamma : $}
\end{picture}}\end{tabular}
\end{equation*}

\vspace*{-2truemm}
\caption{\label{connect-fig}\small
 A particular grafting of a gate~$\delta=X^rD^s$ (here, $r=3,s=2$)
 on a diagram $\gamma$
corresponding to a term
$X^aD^b$ (with $a=3,b=4$), in the case where two edges are matched ($t=2$).}
\end{figure}

Note that  a  labelled diagram can be  interpreted  as representing
the complete \emph{history} of a particular
incremental construction  of an unlabelled  diagram by  successively  ``grafting''
gates  one after the other (Figure~\ref{connect-fig}), following   the  order given by  the
labels---this          viewpoint     will        prove     useful   in
Subsection~\ref{proof-subsec}.    The enumeration of labelled diagrams
corresponding to a  fixed basis $\cal H$  of gates is a central aspect
of our approach to the creation--annihilation algebra\footnote{
As pointed out to us by Ch. Brouder and G. Duchamp (October 2010), an early
ancestor of the diagrams presented here is to be found in the 1965
book of Friedrichs~\cite{Friedrichs65}  relative to perturbation of
spectra in {Hilbert Space};
see, especially, pages 55--56. We note also that the diagrams can be cast into the framework of~\cite{MeBlPe05} by a bijective correspondence $``gates``\leftrightarrow``bugs``$ and $``diagrams``\leftrightarrow``colonies``$.}.

\subsection{The equivalence principle.} \label{eqp-subsec}
We  propose to 
state a general form of the \emph{Equivalence Principle} (Theorem~\ref{eqp-thm} below),
in which gates   and diagrams can be   \emph{weighted}. Fix a field~$\K$,
called the  \emph{domain of weigths}---it  may  be the  field~$\C$ of complex
numbers, in  the case of  numeric weights, or a field~$\C(x,y,\ldots)$
of  rational functions  in  formal variables   $x,y,\ldots$,  in the case   of
symbolic weights.  A   \emph{weighting}  of a  collection $\cal  H$  of
gates is  an assignment  of a  weight   $w_{r,s}\in\K$ to each  gate
of~$\cal H$; equivalently, a function $w:\cal H \into \K$. Such a weighting
on gates induces a weighting on diagrams: \emph{the weight $w(\delta)$ of a diagram~$\delta$
is defined to be the product of
the weights of the individual gates that the diagram is comprised of.} 
Finally, the \emph{total weight}\footnote{
	Thus ``total weight'' extends the notion of ``cardinality'' or ``number of elements''
	and determining total weights can be regarded as a (generalized) enumeration problem.
} of a class $\cal D$ of diagrams
is $\sum_{\delta\in\cal D}w(\delta)$.
With this convention, we state:

\begin{theorem}[Equivalence Principle] \label{eqp-thm} 
Consider a polynomial $\frak{h}$ with normal form
\begin{equation}\label{preeqp}
\frak{h}:=\sum_{(r,s)\in \cal H} w_{r,s} X^r D^s.
\end{equation}
Then the normal ordering of the power~$\frak{h}^n$,
\begin{equation}\label{eqp}
\nor(\frak{h}^n)=\sum_{n,a,b} c_{n,a,b} X^a D^b,
\end{equation}
is such that the coefficient $c_{n,a,b}$ coincides with the \emph{total weight}
of (monotonically) labelled diagrams that admit~$\cal H$ as a basis  weighted by~$w$,
have size~$n$, and are comprised of ~\underline{$a$} outputs and \underline{$b$} inputs.
\end{theorem}

The proof   is deferred until the next subsection.
The interest of Theorem~\ref{eqp-thm} lies in the fact that it
transforms an algebraic  normal-ordering 
problem into a  combinatorial  enumeration problem, one that may be
studied by the methods of combinatorial analysis. It is precisely our aim to 
explore aspects of this equivalence and the derivation of explicit
normal forms, based on combinatorial theory.

We note that Theorem~\ref{eqp-thm} can be rephrased in terms
of generating functions.  
First, by~\eqref{eqp}, the
quantity $\nor(\frak{h}^n)$ is the generating polynomial
of labelled diagrams with $X$ marking  outputs, and~$D$ marking
inputs: as in the case of gates,  a diagram with~$a$ outputs and~$b$ inputs
corresponds to a term~$X^aD^b$ in the normal form.
Introduce 
next the   operator exponential
\[
e^{z\frak{h}} := \sum_{n\ge0} \frak{h}^n \frac{z^n}{n!},
\]
which is a formal power series in~$z$ with coefficients in 
the polynomial ring $\C\langle X,D\rangle$ and the
corresponding normal form
\[
\nor(e^{z\frak{h}})=\sum_{n\ge0} \nor(\frak{h}^n) \frac{z^n}{n!}.
\]
Then, equation~\eqref{eqp} expresses the fundamental identity
\begin{equation}\label{maingf}
\nor(e^{z\frak{h}})=\left(\sum_{n,a,b} c_{n,a,b} u^av^b
  \frac{z^n}{n!}\right)_{\!\!\renewcommand{\arraystretch}{0.5}
\left\{\!\!\begin{array}{c}u\to X\\v\to D\end{array}\right.}.
\end{equation}
The  sum on the right-hand side
is nothing but the  \emph{exponential generating  function} (EGF) of the
 total weight  (number) of diagrams  built  on the weighted  basis $\cal  H$,
where~$z$  marks size, $u$ marks the number of outputs, and $v$ marks
the number of inputs. In writing such equalities, it is understood
that all occurrences of $X$ should be
systematically written\footnote{\label{foot::}
Thus, one calculates generating functions of diagrams in the usual
way,
then, at the end, one should group all occurrences of $u$ before
those of $v$, 
and finally perform the substitution $u\to X, v\to D$. If $G(z;u,v)$ is the
generating function
of diagrams, a notation often used by physicists to indicate this
process is
\[: G(z;X,D) :
\]
which plainly represents a commutative image, with
all $X$s conventionally preceding all $D$s, then interpreted back 
as an operator
in non-commuting $X,D$.
} to
the left of all occurrences of $D$.

Theorem~\ref{eqp-thm}  thus expresses a general
correspondence between the world of operators and the realm of
combinatorics, which is summarized in Figure~\ref{sum-fig}.

\begin{figure}\small
\begin{center}\renewcommand{\tabcolsep}{2truept}
\fbox{\begin{tabular}{lcl}
monomial $X^rD^s$ &$\ \ \ \leftrightsquigarrow\ \ \ $& gate of type $(r,s)$\\
polynomial $\frak{h}$ in $X,D$ &$\ \ \ \leftrightsquigarrow\ \ \ $& weighted basis $\cal H$ of gates\\
$\frak{h}^n$ &$\ \ \ \leftrightsquigarrow\ \ \ $& labelled diagrams of size~$n$ on~$\cal H$\\
$e^{z\frak{h}}$ &$\ \ \ \leftrightsquigarrow\ \ \ $& all diagrams (exp. generating function)\\
$(z^nX^aD^b)$ &$\ \ \ \leftrightsquigarrow\ \ \ $& (size${}=n$, \#outputs${}=a$, \#inputs${}=b$).
\end{tabular}}
\end{center}
\caption{\label{sum-fig}\small
The correspondences between normal forms of algebraic expressions in
$X,D,z$ (\emph{left}) and 
combinatorial diagrams (\emph{right}).}
\end{figure}

\subsection{Proof of the Equivalence Principle (Theorem~\ref{eqp-thm}).} \label{proof-subsec}
For the proof of Theorem~\ref{eqp-thm},
a basic observation is the identity
\begin{equation}\label{obs0}
(X^rD^s)(X^aD^b) =\sum_{t=0}^s \binom{s}{t}\binom{a}{t}t!X^{r+a-t}D^{s+b-t}.
\end{equation}
Without loss of generality, it suffices to consider
the case $r=b=0$.
Then, with the interpretation $Df(x)=\frac{d}{dz} f(x)$ and $Xf(x)=xf(x)$,
we find, by Leibniz's rule,
\[
\begin{array}{lll}
(D^sX^a)f(x)\ \equiv \  D^s(x^af(x)) 
&=& \ds \sum_{t=0}^s \binom{s}{t}(D^tx^a)(D^{s-t}f(x))\\
&=& \ds \sum_{t=0}^s \binom{s}{t}\binom{a}{t}t!x^{a-t} (D^{s-t}f(x)).
\end{array}
\]
The last line  implies~\eqref{obs0},  upon  multiplying on the    left
by~$X^r$ and on the right by~$D^b$.

Let now $X^aD^b$ be a monomial that figures explicitly in $\nor(\frak{h}^n)$
and let $X^rD^s$ be a monomial of~$\frak{h}$. The reduction to normal form of~$\frak{h}^{n+1}$
(viewed as $\frak{h}^{n+1}=\frak{h}\cdot\frak{h}^n$)
is obtained by applying the rule~\eqref{obs0}, then summing over all values $(r,s)\in\cal H$.
Assume  now,  by induction that the   identity~\eqref{eqp} 
of Theorem~\ref{eqp-thm} holds for a
certain value  of~$n$, so that  the coefficients of $\nor(\frak{h}^n)$
count (with suitable weights)  all diagrams of size~$n$. The collection
of   legal diagrams of  size~$n+1$  is  obtained  by  grafting in  all
possible ways a gate of the  basis~$\cal H$ on  a diagram of size~$n$.
If  this  diagram~$\delta$ has  $a$  outputs and~$b$  inputs  and  the
gate~$\gamma$ is of type $(r,s)$, then the number  of ways that such a
grafting can  be effected,  when  $t$ outputs of~$\delta$  are plugged
into~$t$ inputs of~$\gamma$ is exactly
\begin{equation}\label{obs1}
\binom{s}{t}\binom{a}{t}t!.
\end{equation}
(The  binomials count the choices of the~$t$ inputs of~$\gamma$
and the choices of the~$t$ outputs of~$\delta$ that are matched;
the factorial counts the possibilities of attachment.)
Figure~\ref{connect-fig}
displays a particular  attachment of a gate of type~$(3,2)$ to a connected diagram
with~3 outputs and 4~inputs.

It   is    then     observed,   from a     comparison  of~\eqref{obs0}
and~\eqref{obs1} that the multiplicities induced either by a reduction
of   a left  multiplication  by~$X^rD^s$  or    by adding  a  gate  of
type $(r,s)$ are    the      same.     Thus, the       property    that
$\nor(\frak{h}^{n+1})$ is   the generating  function  of  diagrams  of
size~$n+1$ with~$X$ marking outputs and~$D$ marking inputs is established.
The property trivially holds for $\nor(\frak{h})$, ensuring the basis of the induction. 
The proof of Theorem~\ref{eqp-thm} is now complete.

\begin{note} \label{defder-note}\emph{The combinatorics of derivatives
    and Wick's Theorem.} 
The usual rule of calculus $D x^n=nx^{n-1}$ has combinatorial content.
For instance, $D(x^4)=4x^3$ can be obtained as
\[
D(xxxx)=\not{x}xxx+x\!\!\!\not{x}xx+xx\!\!\!\not{x}x+xxx\!\!\!\not{x},
\]
corresponding to the formal rule: \emph{select in all possible ways an
occurrence of the variable  $(x)$ and replace  it by a neutral element
$(1)$.} This  lifts to   arbitrary classes  of  combinatorial
objects   as  a   general   ``pointing--erasing''   operation~\cite[p.~201]{FlSe09}, 
which applies to ``atoms'' (basic components of size~1,
such as letters of words or vertices of graphs)
that compose combinatorial structures.

Equipped with this combinatorial interpretation of~$D$, we can proceed 
to revisit the normal form problem. As before~$X$ is the operator of
multiplication
by the variable~$x$ (i.e., $(Xf)(x)=xf(x)$, for an arbitrary
function~$f(x)$). 
Combinatorially, if $f$ is
regarded as representing an arbitrary combinatorial class whose
elements are made of $x$-atoms, then $X f$ means adjoining---or
``creating''--- one new atom for each
element of~$f$.
Consider now the normal form of an expression such as $DX^2$.
It can be obtained by working out what  $DX^2f$ is:
first $X^2f$ adds two atoms; next, when~$D$ is applied, it must either pick up and erase
one of these added $x$-atoms or hit and destroy an atom of $f$. 
In summary, the external $D$ either ``annihilates'' one occurrence of a following~$X$ or it gets
directly applied to $f$ itself.
Pictorially,
\[
DX^2f=\wick{1}{<1{\not\!\!D}>1{\not \!\!X} Xf}+\wick{1}{<1{\not\!\!D} X >1{\not \!\!X} f}+ 
\wick{1}{<1{D} XX >1{f} }
= 2Xf+X^2Df,
\]
%
so that 
\[
\nor(DX^2)=2X+X^2 D.
\]
Similar developments apply to arbitrary monomials in $D$ and $X$,
giving rise to what is known in quantum physics as \emph{Wick's
  Theorem} and is usually expressed in terms of the annihilation~$a$
and creation $a^\dagger$ operators. (Here, we can interpret~$a$ as $D$ and~$a^\dagger$ as $X$.)

\begin{proposition}[Wick's Theorem]
 The normal form of a monomial in~$X,D$ equals
the sum of all expressions obtained by removing in
all possible ways an arbitrary number
of pairs $D\ldots X$ of annihilation $(D)$ 
and creation ($X$) operators, where $D$ precedes $X$,
then  reorganizing the resulting monomials in such a way that 
all $X$s precede all $D$s.
\end{proposition}

\noindent
Such a removal is called a \emph{contraction}. An example of a normal form
computation
in this framework (with $: E :$ representing the commutative reorganization
where all $X$s precede all $D$s, as 
in footnote${}^{\hbox{\footnotesize(\ref{foot::})}}$)
 is as follows:
\[
\begin{array}{lll}
  DX DDX D
   &=& \ :\!\underbrace{DX DDX D}_{\text{no pair removed}}\!:\\
   &&\quad {} +:\underbrace{\wick{1}{<1{\not \!\!D}>1{\not \!\!X}\!\!\ DDX D}+\wick{1}{<1{\not 
\!\!D}X DD>1{\not \!\!X} D}
+\wick{1}{DX <1{\not \!\!D}D>1{\not \!\!X}D}+\wick{1}{DXD<1{\not \!\!D}>1{\not
\!\!X}D}}_{\text{1 pair removed}}:\\
    &&\quad{} +:\underbrace{\wick{11}{<1{\not \!\!D}>1{\not \!\!X} <2{\not \!\!D}D>2{\not \!\!X}
        D}+\wick{11}{<1{\not \!\!D}>1{\not \!\!X} D<2{\not
          \!\!D}>2{\not \!\!X} D}}_{\text{2 pairs removed}}:\\
   & =& X^{2} D^4+4\ X D^3+2\ D^2.
\end{array}
\]
(From here, the proof of Wick's Theorem is immediate: it suffices to 
proceed by induction on the length of the monomial to be reduced,
distinguishing the two cases $\nor(Xw)$ and $\nor(Dw)$, where $w$
in an arbitrary monomial in $X,D$.)

The procedure that underlies Wick's Theorem may involve a large number
of steps,  as it amounts  to enumerating  all  possible \emph{sets} of
contractions.  The computational complexity is thus exponential in the
worst case.  By contrast, the combinatorial approach based on diagrams
provides a graphical means to track  down patterns in the diversity of
Wick's contractions, which  can then be effectively  used to achieve a
reduction in computational  complexity and, in  several cases, lead to
closed-form solutions.




Next, we can interpret gates in the same vein: $D^s$ corresponds
to selecting a sequence of $s$ distinct atoms and replacing each of
them by the neutral element, whereas $X^r$ means adding $r$ new atoms.
A derivative $(D)$ then ``hooks'' on atoms ($X$) 
or it ``jumps'' to the right.
A particular labelled diagram 
composed of gates $(\gamma_1,\ldots,\gamma_n)$, with $\gamma_j$
associated with $X^{r_j}D^{s_j}$,
 then corresponds to a particular expansion of
\[
(X^{r_n}D^{s_n})\cdots (X^{r_2}D^{s_2}) (X^{r_1}D^{s_1}).
\]
Identities such as~\eqref{obs0} then receive a natural interpretation
and the construction of diagrams can be 
entirely developed in this way from first  principles---this approach
will be revisited in Section~\ref{frameworks-sec}, when we discuss the 
$q$-difference operator~$\Delta$.
\end{note}

\def\ph{\varphi}
\begin{note} \emph{Duality.} \label{dual-note}
There is an easy but important duality in reductions to normal forms.
For noncommutative expressions in the two variables~$X,D$, define an
\emph{antimorphism} $\ph$ by the rules
\begin{equation}\label{dual}
\ph(X)=D, \quad \ph(D)=X, \quad \ph(U\cdot V)=\ph(V)\cdot \ph(U),
\end{equation}
together with an extension to polynomials by linearity.
Thus, for a monomial $\frak{m}$, its image $\ph(\frak{m})$ is obtained
  by
exchanging the r\^oles of $X$ and $D$ as well as reading letters backwards.
It is then observed that the basic quantity $DX-XD-1$ is invariant
under~$\ph$.
As a consequence, any identity $\frak{U}=\frak{V}$ (modulo $DX-XD=1$) 
over terms $\frak U, \frak V$ in $X,D$ immediately implies a 
\emph{dual identity} $\ph(\frak{U})=\ph(\frak{V})$.

%
\begin{figure}[t]\bigskip\small
\begin{center}
\resizebox{0.8\columnwidth}{!}{\includegraphics{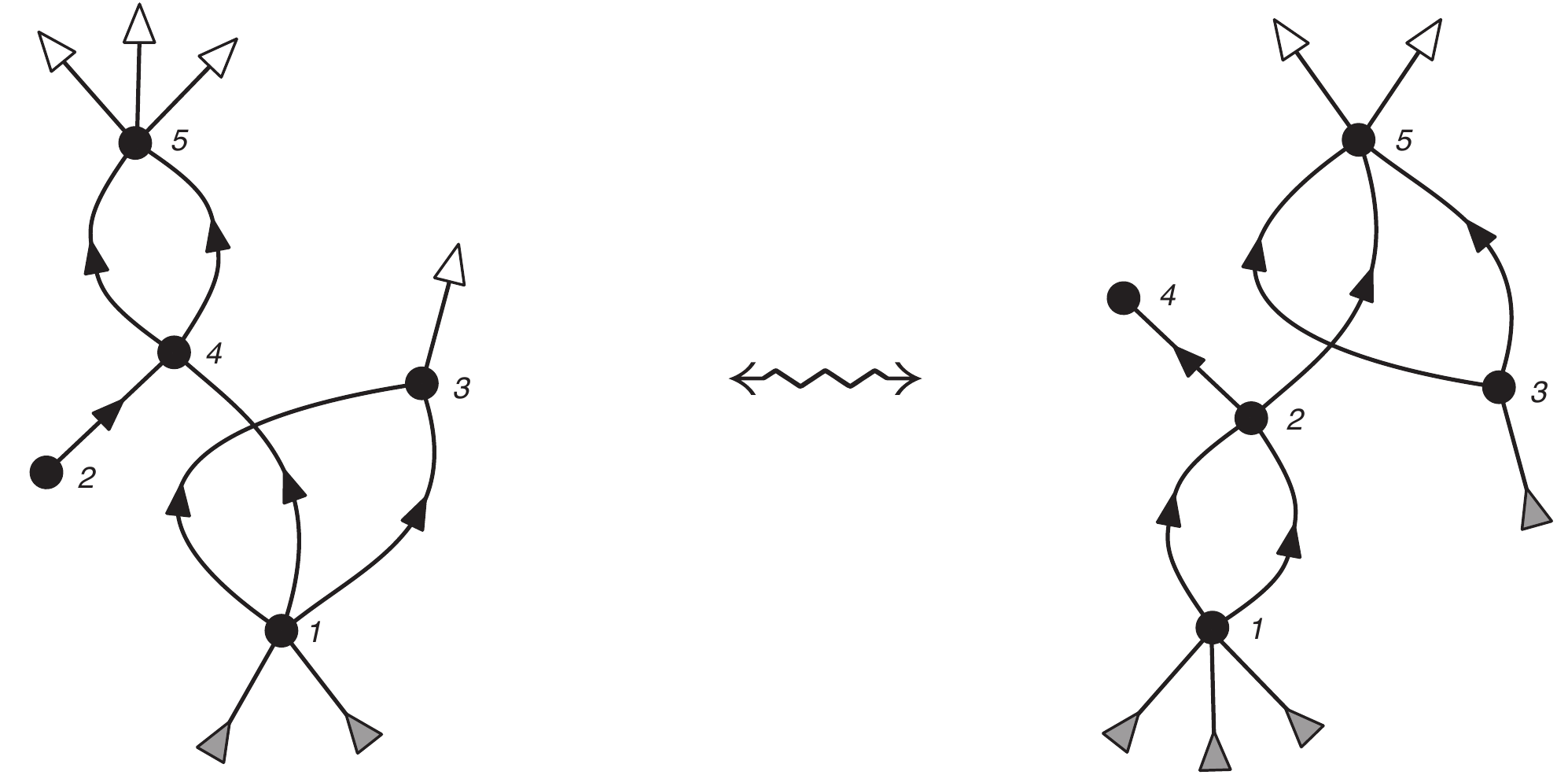}}
\caption{\label{Diagram-dual} \small
A diagram (left) and its dual (right).}
\end{center}
\end{figure}

This duality has a natural graph interpretation.
Define the \emph{dual graph} $\tilde \varGamma$ by reverting the arrows
in $\varGamma$, relabeling the vertices $\bullet$ in  reverse order
(\emph{i.e.}, $1,2,...,n-1,n$ are changed into $n,n-1,...,2,1$
respectively).
The result is again a legitimate graph. See Figure~\ref{Diagram-dual} for illustration. 
Clearly, the dual graph is made of a
basis of gates that is dual to the
one of the original graph.

A consequence of the foregoing considerations is that the normal forms
of $\frak h^n$ and of $\ph(\frak h)^n$ involve the same coefficients,
with the coefficient of~$X^aD^b$ in $\nor(\frak h^n)$ being identical
to that of $X^bD^a$ in $\nor(\ph(\frak h)^n)$. In this way, for
instance, the normal forms we develop below for $(X^2D)^n$ 
immediately translate to the case of $(XD^2)^n$.
(In physical contexts, duality is often associated with hermitian
conjugacy; see Mikhailov's article~\cite[\S3]{Mikhailov85} for several  examples).
\end{note}

%

\begin{note}\label{eulermac-note}
\emph{The combinatorics of Taylor's formula.}
Here is a basic illustration of the combinatorics
of derivatives. It is well known, from several areas of analysis and the calculus of finite differences\footnote{
See, e.g., the treatises of Jordan~\cite[pp.~13--14]{Jordan65}
and Milne-Thomson~\cite[p.~33]{Milne81}.
The interest of this symbolic view  is \emph{inter alia} to lead to an immediate proof 
of (a form of) the Euler--Maclaurin summation formula  by a computation
of $(e^{yD}-1)^{-1}$; cf~\cite[p.~260]{Milne81}.
}, that
the exponential of a derivative plays the r\^ole of a translation operator---also
known as ``shift''.
Specifically, with the notations of Note~\ref{defder-note}, 
we consider the operator
\begin{equation}\label{expd}
 T_y:=e^{yD}.
\end{equation}
Symbolically, we have, corresponding to Taylor's formula:
\begin{equation}\label{shift}
T_y\cdot f(x)=\sum_{n\ge0} \frac{y^n}{n!}D^n f(x)=
\sum_{n\ge0}\frac{y^n}{n!} f^{(n)}(x)=f(x+y).
\end{equation}

It is piquant to note that the formula admits a transparent 
combinatorial interpretation, in view of Note~\ref{defder-note}.
Think of $f(x)$ as being the generating function of a class $\cal F$ of 
combinatorial objects, themselves composed of atoms represented by~$x$:
\[
f(x)=\sum_{\phi\in\cal F} x^{|\phi|}.
\]
The application of the operator $\frac{1}{n!}D^n$ means: ``select in  all possible ways
an unordered collection of~$n$ atoms in each element of~$\cal F$ and replace these by neutral atoms''.
The application of $\frac{1}{n!}y^nD^n$ then translates as follows:
``select in  all possible ways
a collection of~$n$ atoms of type~$x$ in each elements of~$\cal F$ 
and replace these by $y$-atoms''. The exponential $e^{yD}$ then corresponds to 
choosing an \emph{arbitrary} number of~$x$'s and replacing them by~$y$'s,
the outcome being exactly the
bivariate generating function $f(x+y)$. Thus, 
seen from combinatorics, Taylor's formula
\[
f(x+y)=\sum_{n\ge0}\frac{y^n}{n!} f^{(n)}(x)
\]
simply expresses \emph{the decomposition of a bicolouring process
according to the number of atoms whose colour is changed}.
Figuratively:
\[
\hbox{Bicolour${}_{x,y}\,[\cal F]$}
\quad\equiv\quad \bigcup_{n=0}^\infty \hbox{``change $n$ occurrences of~$x$ into~$y$ in elements of~$\cal F(x)$''}.
\]
(This exercise appears, for instance,  explicitly as Note~III.31 in~\cite[p.~201]{FlSe09}.)
\end{note}

\begin{note}\label{pdes-note}
\emph{Normal forms and PDEs.}
The reduction of powers of differential operators to normal form is of
interest  in  the  analysis  of    certain   types of    \emph{partial
differential equations} (PDEs). Consider the \emph{initial value problem}, 
also known as \emph{``Cauchy problem''},
\begin{equation}\label{evol}
\left\{
\begin{array}{ll}
\ds \frac{\partial}{\partial t} F(x,t)=\Gamma\, F(x,t), &
\qquad\hbox{with}\quad \Gamma\in \C[x,\partial], 
\quad \partial\equiv\partial_x=\frac{\partial}{\partial x}\\[2mm]
F(x,0)=f(x),
\end{array}\right.
\end{equation}
where $\Gamma$ is a differential operator, which is a polynomial in $x$
and $\partial_x$.
The solution is determined by the initial data $f(x)$ at time~$t=0$.
An equation of this sort  is sometimes referred to as an \emph{evolution equation},
with the specific choice of the linear differential operator~$\Gamma$ 
serving to model various physical contexts. 
Classical examples, in the one-dimensional case, are the 
heat equation,
\begin{equation}\label{heat}
\frac{\partial F}{\partial t}= \frac{\partial^2F}{\partial x^2},
\end{equation}
and the Schr\"odinger equation with time-independent 
potential $V\equiv V(x)$:
\begin{equation}\label{schroed}
i\hb \frac{\partial F}{\partial t}=
-\frac{\hb^2}{2m}\frac{\partial ^2 F}{\partial x^2}
+V(x) F .
\end{equation}

With the notations of~\eqref{evol}, introduce the \emph{Ansatz}
\begin{equation}\label{ansatz}
F=e^{t\Gamma}\, f,
\end{equation}
where the operator exponential is classically defined as
\begin{equation}\label{expdef}
e^{t\Gamma}:=\sum_{n=0}^\infty \frac{t^n}{n!} \Gamma^n.
\end{equation}
Proceeding formally, we verify that
\begin{equation}\label{sola}
\frac{\partial F}{\partial t}
=\left(\frac{\partial}{\partial t}
e^{t\Gamma}\right)\, f =
\Gamma e^{t\Gamma}\, f =\Gamma\,  F,
\end{equation}
while,  by construction, $F$ reduces to~$f$   at $t=0$.  (Analytically,
sound uses of  such operator  exponentials can be based on the theory   of
operator semigroups;     see, for instance,  the    book  by Engel and
Nagel~\cite{EnNa00}, and especially its Chapter  6, for an  accessible
discussion.) 

If now the normal form of the powers $\Gamma^n$
can somehow be regarded as known,
\[
e^{t\Gamma}=\sum_{n=0}^\infty \frac{t_n}{n!} A_n,
\quad\hbox{with}\quad
A_n=\sum_{\alpha,\beta} a^{(n)}_{\alpha,\beta} x^\alpha \partial^\beta,
\]
then, a formal solution of the evolution equation~\eqref{evol} is
\begin{equation}\label{soluf}
F(x,t)=\sum_{n=0}^\infty \frac{t^n}{n!}
\sum_{\alpha,\beta} a_{\alpha,\beta}^{(n)} x^\alpha \partial^\beta f.
\end{equation}
It is naturally a nontrivial matter to make analytic sense of~\eqref{soluf},
which often proves to be a \emph{divergent} expansion, but 
the formal solution~\eqref{soluf} should at least provide valuable clues as to the kind
of special function involved in the \emph{analytic} solution of
the evolution equation~\eqref{evol}.
\end{note}

\begin{note}\label{heat-note}
\emph{The heat equation.} 
To illustrate the usefulness
of operator calculus, we briefly show how to rederive\footnote{
What follows is extremely  classical material (see, e.g.,~\cite[pp.~107--109]{DyMK72}),
only slightly rearranged to suit our needs.}
 formally the solution of
the heat equation~\eqref{heat} in  the present perspective.
In accordance with the Ansatz~\eqref{ansatz}, one should
consider the operator exponential $e^{t\partial^2}$.

We  choose\footnote{
This choice actually 
corresponds to an eigenfunction expansion in disguise. 
For instance, if $\Gamma$ has
known eigenvalues $\{\lambda\}$ with eigenfunctions~$\{v(x)\}$, then 
$e^{t\Gamma}\, v(x)=e^{t\lambda}v(x)$, 
from which solutions can be composed by linearity.}
 to examine  the  effect  of the exponential-of-the-Laplacian
$e^{t\partial^2}$  on the  collection  of base   functions $\{e^{i\omega
x}\}$, so that
\[
e^{t\partial^2}\cdot e^{i\omega x}=\sum_{n\ge0} \frac{t^n}{n!}
\left(( i\omega)^{2n} e^{i\omega x}\right) = e^{-t\omega^2}e^{i\omega x}.
\]
Then, if $f(x)$ is a Fourier integral,
\[
f(x)=\int_{-\infty}^{+\infty} e^{-i\omega x} \phi(x)\, d\omega,
\]
Equation~\eqref{soluf} yields by linearity the  \emph{formal} solution
\[
F(x,t)=\int_{-\infty}^{+\infty}
 e^{-i\omega x} \left(e^{-t\omega^2}\phi(\omega)\right)\, d\omega.
\]

The function~$F({}\cdot{},t)$ thus appears as the Fourier transform of 
the product of the two functions $e^{-t\omega^2}$ and $\phi(\omega)$.
This, by well known properties of the Fourier transform,
which exchanges convolutions and ordinary product, leads
(formally still) to the celebrated ``heat kernel'' solution
\begin{equation}\label{heatsol}
F(x,t)=\frac{1}{\sqrt{4\pi t}}\int_{-\infty}^{+\infty} 
\exp\left(-\frac{(x-y)^2}{4t}\right) f(y)\, dy.
\end{equation}

In this particular case of the heat equation, 
the normal form problem relative to $(\partial^2)^n$ is trivial.
Nonetheless, the derivation above demonstrates the type
of usage of operator exponentials, once these
can be made sufficiently ``explicit''. It is one of our goals
to develop a combinatorial toolbox for the simplification 
of such exponentials, which could then serve as a preamble
to  the analysis of more
complicated types of PDEs.
\end{note}

%
%
%

\subsection{Combinatorial enumeration.} \label{combana-subsec}
Throughout this paper, we appeal to general methods of combinatorial analysis
relative to the enumeration of labelled objects and
extensively base our discussion on the standard conventions of the book
\emph{Analytic Combinatorics}~\cite[Ch.~2]{FlSe09}.
If~$\cal C$ is a combinatorial class formed of labelled objects (typically, diagrams), 
we systematically let $\cal C_n$ be the subclass
of objects of size~$n$, with $C_n$ the corresponding cardinality. The
\emph{exponential generating function} (EGF) of the class is
\[
C(z):=\sum_{n\ge0} C_n \frac{z^n}{n!}=\sum_{c\in\cal C} \frac{z^{|c|}}{|c|!}.
\]

What we want to do is construct complex combinatorial
classes from simpler ones. The initial classes include the 
\emph{atomic class}~$\cal Z$, which comprises a single element of
size~1
and has EGF~$z$, as well as
 the \emph{neutral class}~$\cal E$, which consists of a single element of
size~0
and has EGF~$1$

Disjoint unions, henceforth written plainly as `$+$', clearly correspond to sums of EGFs:
\begin{equation}\label{dic00}
\cal C=\cal A+ \cal B \qquad\implies\qquad C(z)=A(z)+B(z).
\end{equation}
The \emph{labelled product}  $\cal  C=\cal A\star  \cal  B$ of two   labelled
classes is obtained  by taking all  the ordered pairs $(\alpha,\beta)$,
with~$\alpha\in\cal A$ and $\beta\in\cal B$, then  relabelling them in  all
order-consistent ways    so     as   to  obtain  a    well-labelled     pair
$(\alpha',\beta')$. We then have the correspondence
\begin{equation}\label{dic0}
\cal C=\cal A\star \cal B \qquad\implies\qquad C(z)=A(z)\cdot B(z).
\end{equation}
It is then possible to form the class of
all (labelled) \emph{sequences} (`$\seq$'), \emph{sets} (`$\set$'), and \emph{cycles} (`$\cyc$')
with components in~$\cal A$, the corresponding EGFs being given by the following dictionary:
\begin{equation}\label{dic}
\left\{\begin{array}{lll}
\cal C=\seq(\cal A) &\implies& \ds C(z)=\frac{1}{1-C(z)}\\
\cal C=\set(\cal A) &\implies& \ds C(z)=\exp(C(z))\\
\cal C=\cyc(\cal A) &\implies& \ds C(z)=\log\frac{1}{1-C(z)}\,.
\end{array}\right.
\end{equation}
These basic constructions suffice to transcribe the normal form problem 
in simpler cases such as $(X+D)$ in Section~\ref{lin-sec} and $(XD)$ in Section~\ref{bell-sec}.

Finally, we shall need the modified \emph{``boxed product''} construction,
$
\cal C=\left(\cal A^{\Box}\star \cal B\right)
$,
which corresponds to the subset of ordered pairs $(\alpha,\beta)\in(\cal A\star\cal B)$, such that
the \emph{smallest} label is constrained to belong to the~$\alpha$ component;
see Greene's thesis~\cite{Greene83},
as well as the accounts in~\cite[Ch.~7]{BeLaLe98} and~\cite[\S II.6.3]{FlSe09}.
The translation rule from constructions  to EGFs is
\begin{equation}\label{boxp}
\cal C=\left(\cal A^{\Box}\star \cal B\right),
\quad\implies\quad C(z)=\int_0^z \left(\frac{d}{dt}  A(t)\right)\cdot B(t)\, dt.
\end{equation}
This  covers
in  particular the \emph{min-rooting}   operation,  which attaches  an
external atom  to a  $\cal B$--structure and  assigns to  it the smallest
label: the  construction is simply  $(Z^{\Box}\star  \cal B)$ and it
corresponds to an integration operator $\int_0^z B(t)\, dt$.  The same
translation  applies  to  the  dual    boxed   product $\cal   C=(\cal
A^\blacksquare\star\cal B)$, where it  is now the \emph{largest} label
that is constrained to belong to the $\cal A$-component. We shall make
much  use of these  constructions, which correspond to a  decomposition
according to the first or  last gate of  a diagram: they are used
typically in Section~\ref{quad-sec}, relative to $(X^2+D^2)$, and
Section~\ref{semilin-sec},
relative to $(X^2D)$.

All the rules above extend to \emph{multivariate} generating
functions:
these  contain extra parameters, 
which can keep track of various additive characteristics of structures~\cite[Ch.~III]{FlSe09}.

In the next sections,   we are precisely   going to  make use   of the
dictionary  formed by~\eqref{dic00},  \eqref{dic0},~\eqref{dic}, and~\eqref{boxp} in
order to enumerate     various classes of diagrams     associated (via
Theorem~\ref{eqp-thm}) with the normal ordering of
terms in the operators~$X,D$.

\section{\bf Linear forms ($X+D$), involutions, and generalizations} \label{lin-sec}

This section is dedicated to the normal ordering of expressions of the form
$(a(X)+D)^n$ and $(a(D)+X)^n$, with~$a({}\cdot{})$ a polynomial, starting
with the easier case $(X+D)^n$. Thus, the schema is that of a base form in~$X$
and~$D$, which is linear in at least one of the operators $X,D$.
The combinatorial models turn out to be \emph{special coloured 
permutations} (involutions and generalizations),
for which the introduction of diagrams 
easily leads to fully explicit forms.

\subsection{The basic linear case $(X+D)$.} \label{baslin-subsec}
Linear forms in the operators $X$ and $D$, namely, operators of the form
\[
\Gamma=\alpha X+ \beta D,
\]
can serve to illustrate the usefulness of diagrams with minimal apparatus. In accordance with
the developments of Section~\ref{maindef-sec}, we are talking here of diagrams based on two types of gates:
a gate $X$ has no input and one output; a gate $D$ has one input and
no output: 
\[
\hbox{\Img{6.5}{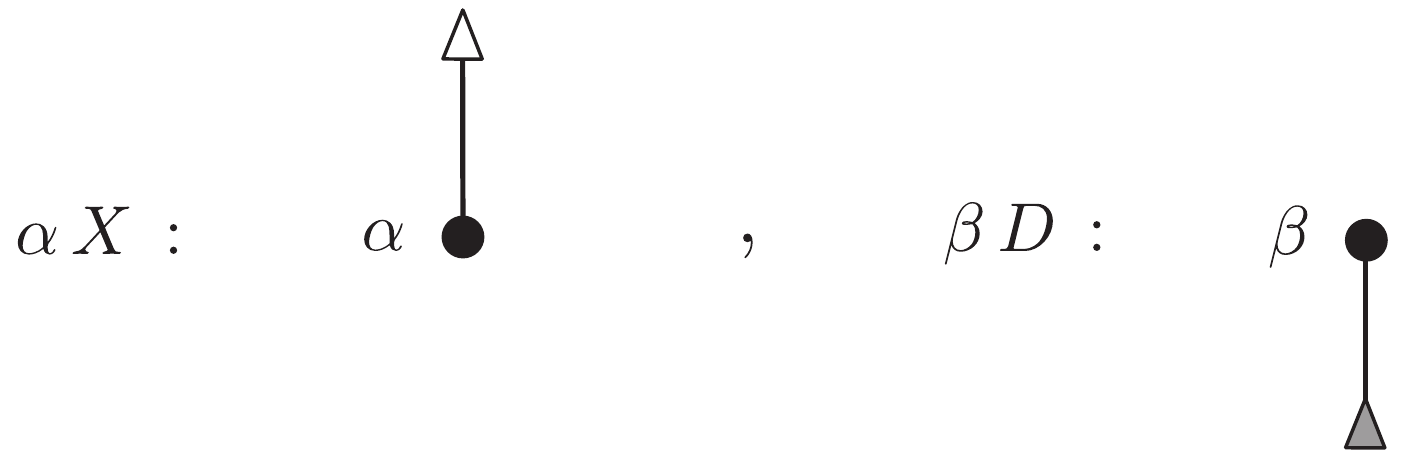}}.
\]
By inspection of all the possible ways of assembling gates, 
it is immediately seen that a
 diagram must  be comprised of three types of components:
isolated~$X$--gates, isolated $D$--gates, 
and pairs $DX$, where a $D$ is hooked to an earlier arrived $X$.
In the labelled case, the isolated~$X$ and $D$ may receive arbitrary labels, while the ``saturated'' $DX$ pairs are
only constrained by the fact that the label of the $X$ is  smaller than the label of the~$D$; 
see Figure~\ref{XplusD-fig}.

\begin{figure}\small
\begin{center}
\Img{9}{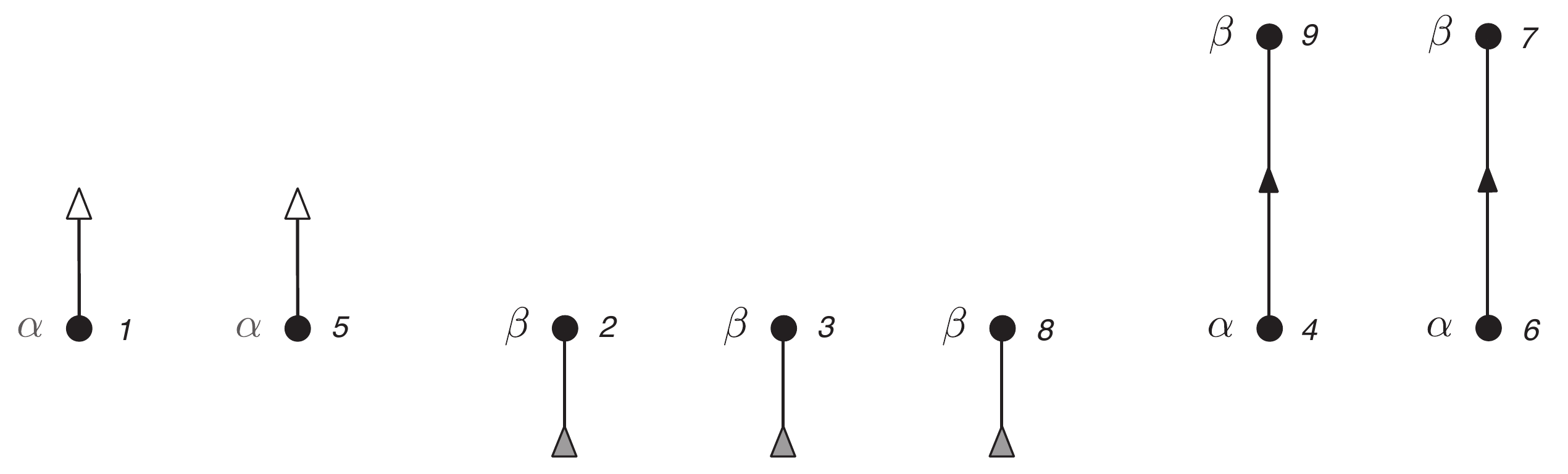}
\end{center}
\caption{\label{XplusD-fig}\small
A particular $(X+D)$--diagram.}
\end{figure}

Combinatorially, the class of all diagrams is thus specified by
\[
{\cal G}=\set(\alpha\cal Z+\beta\cal Z+\alpha\beta\set_2(\cal Z)),
\]
when the multiplicative weights $\alpha,\beta$ are taken into account. 
Here, the $\set_2$ construction\footnote{
In this particular case, a $\cyc_2$ construction can equivalently be used,
given the obvious combinatorial isomorphism $\set_2(\cal Z)\cong\cyc_2(\cal Z)$.}
describes an unordered pair
of two  labels whose arrangement is immaterial: the
smaller label is necessarily attached to an~$X$-gate, the larger one
to a~$D$-gate;
see Figure~\ref{XplusD-fig}.
The generating function of all weighted diagrams is accordingly
\begin{equation}\label{inv1}
G(z)=e^{(\alpha+\beta)z+\alpha\beta z^2/2}.
\end{equation}
An equivalent way of phrasing this result is as an equality,
\begin{equation}\label{inv2}
\nor\left(e^{z(\alpha X+\beta D)}\right)=e^{\alpha\beta z^2/2}\cdot e^{\alpha zX}\cdot e^{\beta zD},
\end{equation}
where the right-hand side is to be interpreted as a standard expansion
in the
commuting indeterminates $X,D$.
A straight expansion then yields an explicit form of~\eqref{inv2},
\[
\sum_{k,\ell,m\ge0} \frac{z^{2k}}{2^k \, k!}(\alpha\beta)^{k} \frac{z^\ell}{\ell!}\alpha^\ell X^\ell
\frac{z^m}{m!}\beta^m D^m,
\]
which, after collection of the coefficient of~$z^n$, can be summarized 
as follows\footnote{
Proposition~\ref{invol-prop} is a classical result: it is for instance
derived by means of operator calculus methods in Wilcox's paper~\cite[Eq.~(10.43)]{Wilcox67},
published in 1967.}.

\begin{proposition} \label{invol-prop}
The normal form of $(\alpha X+\beta D)^n$ satisfies
\[
\nor\left((\alpha X+\beta D)^n\right)=\sum_{\ell,m} I^{(n)}_{\ell,m} X^\ell D^m,
\]
where, for $n-\ell-m$ odd, the coefficients $I^{(n)}_{\ell,m}$ are~$0$,
while, for $n-\ell-m$ even, they satisfy
\[
I^{(n)}_{\ell,m}=\frac{n!}{2^{(n-\ell-m)/2}((n-\ell-m)/2)!\, \ell!\, m!}
\alpha^{(n+\ell-m)/2}\beta^{(n+\ell+m)/2}.\]
\end{proposition}

In combinatorics, a permutation~$\sigma$ containing only cycles of sizes~1 and~2 is known as an \emph{involution}
(it satisfies~$\sigma^2=\mathbf{1}$). The specification and 
exponential generating functions are accordingly
\[
\cal I=\set(\cal Z+\cyc_2(\cal Z)),
\qquad
I(z)=e^{z+z^2/2};
\]
see, e.g.,~\cite[p.~122]{FlSe09}.
The numbers $I^{(n)}_{\ell,m}$ therefore
enumerate coloured involutions, where singleton cycles can receive any one of 
two colours (corresponding either to an $X$ or a $D$). In other words: \emph{the natural combinatorial 
model for the normal ordering of $(X+D)^n$ is that of (bicoloured) involutions.}

The classical result expressed by Proposition~\ref{invol-prop}
is usually derived in the context of Lie groups by means 
of the Baker--Campbell--Hausdorff formula~\cite{Gilmore74,Hall03},
since, in this case, nested Lie brackets of higher order vanish. 
Indeed, it is known in this theory that, if the commutator $C=[A,B]$ commutes with both~$A$ and~$B$, one has the general identity\footnote{
Formulae of this sort are sometimes known as ``disentanglement formulae''~\cite{DasGupta96}.}
(see~\cite[p.~463]{Gilmore74} and \cite[p.~64]{Hall03}):
\[
e^{z(A+B)}=e^{zA}\cdot e^{[B,A]z^2/2}\cdot e^{zB}.
\]

\begin{note} \label{pdeinv-note}
\emph{Solution of a special PDE by normal ordering of $(X+D)^n$.}
The partial differential equation of interest is
\begin{equation}\label{pdeinv}
\frac{\partial F}{\partial t}
=\frac{\partial F}{\partial x}+xF,
\end{equation}
with initial value condition $F(x,0)=f(x)$. This is of the form considered when discussing 
operator exponentials in Note~\ref{pdes-note}, p.~\pageref{pdes-note}, 
here  with $\Gamma=(X+D)$. The operator solution $e^{t\Gamma}\, f$,
under the normal ordering provided by 
Equation~\eqref{inv2} above, becomes
\[
e^{t\Gamma}\, f =e^{t^2/2}\cdot e^{tx}\cdot e^{tD}\, f.
\]
We know (from the earlier Note~\ref{eulermac-note}, 
p.~\pageref{eulermac-note}) that the exponential of a derivative is a shift,
$e^{tD}\, f(x)=f(x+t)$, so that the solution to~\eqref{pdeinv} is the fully explicit
\begin{equation}\label{pdeinv2}
F(x,t)=e^{t^2/2+xt}f(x+t),
\end{equation}
whose validity is easily checked directly.

No claim is made that the PDE~\eqref{pdeinv} is hard to solve, and, indeed, it succumbs easily to the method of characteristics,
which is generally applicable to linear and quasilinear PDEs~\cite[\S1.15]{Taylor96}. 
It is nonetheless instructive to observe the way the reduction
to normal ordering (with the right factor $e^{tD}$) could automatically put us on the tracks of a general solution.
\end{note}

\subsection{Generalizations to $(X+D^r)$ and $(X^r+D)$.}
\def\qu{\Box}
We next examine the case of the operator $(X+D^2)$. According to 
the main theorem, one should consider diagrams built out of two types of gates:
$D^2$--gates have two inputs and no output; $X$--gates have, as usual, one output and 
no input:
\[
\hbox{\Img{6.5}{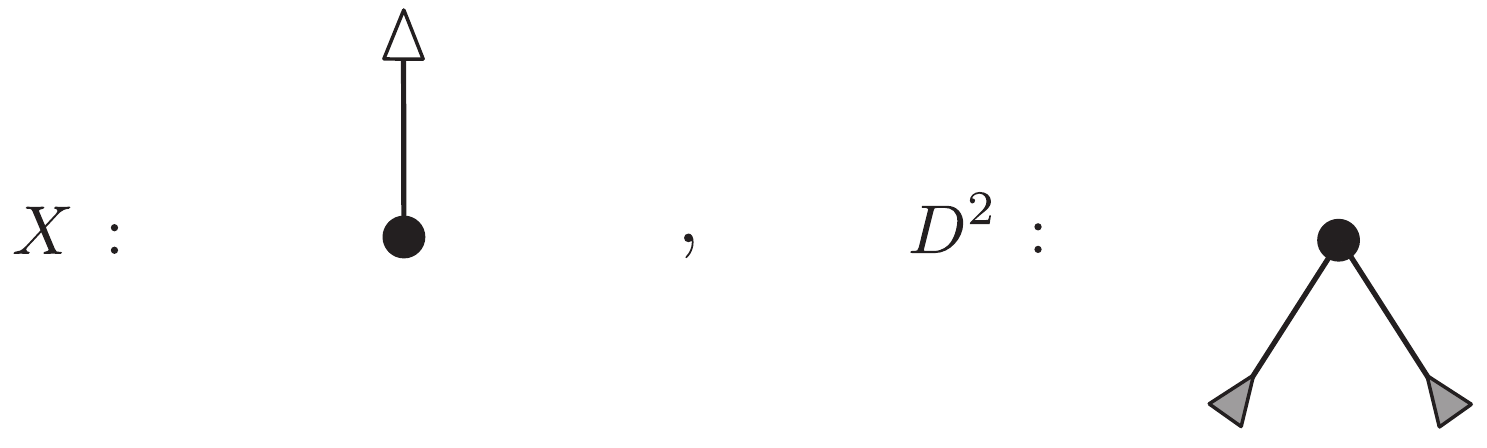}}.
\] 
It is then easily realized that the only possibilities 
for connected diagrams are given by the following list:
\begin{equation}\label{list0}
\hbox{\Img{10}{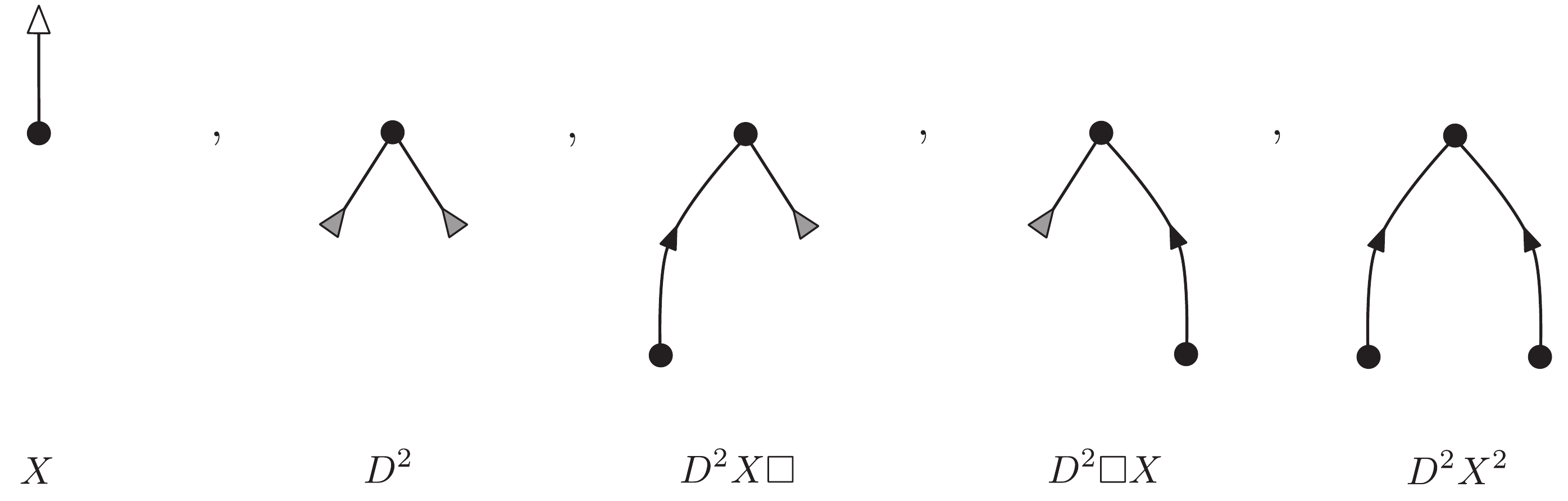}}.
\end{equation}
For instance, $D^2X^2$ represents a connected component that is ``saturated'',
in the sense that the two inputs of a $D^2$--gate are connected to 
the (earlier arrived)  outputs of two $X$--gates; similarly,
$D^2 X \qu $ depicts a component, for which the 
output of an $X$--gate is plugged into the first  input of a $D^2$--gate,
while the second input remains free; and so on.

We can then look at the balance of inputs and outputs corresponding 
to each of the gates of the list~\eqref{list0}. For instance, a 
$D^2 X \qu $ component has one input gate (a dangling~$D$) and no output,
so that it is equivalent to a $D$, in terms of the balance between number
of inputs and number of outputs. Proceeding systematically, we determine the following
table for all the possible types of  connected components: 
\begin{equation}\label{list1}
\begin{array}{lccccc}
\hline\hline 
\hbox{type} : & D^2 &  X &  D^2\qu X &  D^2 X\qu &   D^2X^2\\
\hline
\hbox{size} : & 1 & 1 & 2 & 2 & 3 \\
\hbox{balance} : & D^2 & X & D & D & 1. \\
\hline\hline
\end{array}
\end{equation}

Next, we should examine the number of ways of placing labels on connected components.
In the case of a  $D^2\qu X$ or $ D^2 X\qu$, the size equals~2, but the $D^2$--gate
is necessarily associated with the largest label; thus, any such component is 
a $\set_2(\cal Z)$, itself combinatorially equivalent to a $\cyc_2(\cal Z)$. 
In the case of a saturated component $D^2X^2$, 
two possible orderings are to be considered,
since the first input of $D^2$ is associated with either the
earlier arrived~$X$ or with the later~$X$; in this case, the component 
turns out to be equivalent
to a $\cyc_3(\cal Z)$. Thus, using now,
as standard \emph{commuting} formal variables, 
$u$ to mark  components with a dangling output ($X$)
and $v$ to mark  components with a dangling
input ($D$), we have,
for the various components, the specifications:
\begin{equation}\label{list2}
\begin{array}{lccccc}
\hline\hline
\hbox{type} : & D^2 &  X &  D^2\qu X &  D^2 X\qu &   D^2X^2\\
\hline
\hbox{specification} : 
 & v^2\z &u\z & v\cyc_2(\z) & v\cyc_2(\z) & \cyc_3(\z). \\
\hline\hline
\end{array}
\end{equation}
There results that the class of all diagrams is described by
\[
\cal G =  \set\left(v^2\z +u\z +v\cyc_2(\z) + v\cyc_2(\z) + \cyc_3(\z)\right),
\]
with corresponding EGF
\begin{equation}\label{geninv0}
G(z;u,v)=\exp(v^2 z+uz+vz^2/2+vz^2/2+z^3/3).
\end{equation}
In other words, the model is now that of \emph{(multicoloured) permutations, 
all of whose cycles are of length at most three, where, in addition, singletons
and doubletons can be of any of two colours}.

The very same reasoning applies to the normal ordering of $(X^2+D)^n$.
The diagrams are isomorphic to the earlier ones, with inputs and outputs being
exchanged and time being reversed. That is, the r\^oles of~$X$ and~$D$ 
in~\eqref{geninv0} are
simply to be exchanged---this is a special case of 
the general \emph{duality} discussed in Note~\ref{dual-note}, p.~\pageref{dual-note}. 
In summary:

\begin{proposition}
The normal orderings corresponding to the base operators $(D^2+X)$
and $(X^2+D)$ satisfy
\begin{equation}\label{geninv}
\left\{\begin{array}{lll}
\ds \nor\left(e^{z(D^2+X)}\right)
&=& \ds e^{z^3/3+zX}\cdot e^{z^2D+zD^2}\\
\ds \nor\left(e^{z(X^2+D)}\right)
&=& \ds e^{z^3/3+z^2X+zX^2}\cdot  e^{zD}\,.
\end{array}\right.
\end{equation}
\end{proposition}

Finally, the explicit 
normal forms associated with $(D+a(X))$ and $(X+a(D))$,
for $a({}\cdot{})$ a polynomial, are accessible 
via a combinatorial calculus that suitably extends~\eqref{list1} 
and~\eqref{list2}: 
the combinatorial model 
now involves permutations with cycles of length at most~$r$.
We leave it as an exercise to the reader to work out details
and state\footnote{
Mikhailov~\cite{Mikhailov83} has a simple proof based on operator algebra in~\cite{Mikhailov83},
and he refers to earlier works of  Witschel (1975)
and Yamazaki (1952).
}:

\begin{proposition} The normal ordering corresponding to the base operators $(a(D)+X)$
and $(a(X)+D)$ are given by
\begin{equation}\label{geninv2}
\left\{\begin{array}{lll}
\ds \nor\left(e^{z(a(D)+X)}\right)
&=&\ds \left. e^{zX}\cdot \exp\left(\int_0^z a(v+w)\, dw\right)\right|_{v\mapsto D}\vspace{0.3cm}\\
\ds \nor\left(e^{z(a(X)+D)}\right)
&=&\ds  \left.\exp\left(\int_0^z a(X+w)\, dw\right)\right|_{u\mapsto X}\cdot e^{zD}.
\end{array}\right.
\end{equation}
\end{proposition}
\noindent
In the pure binomial case of $(D^r+X)$ and $(X^r+D)$, the normal forms further simplify\footnote{
These formulae have also been found independently by Karol Penson (unpublished, 2004).}
as
\[
\left\{\begin{array}{lll}
\ds  \nor\left(e^{z(D^r+X)}\right) &=& \ds e^{zX}\cdot \exp\left(
{\textstyle \frac{1}{r+1}}\left[(D+z)^{r+1}-D^{r+1}\right]\right)
\\
\ds  \nor\left(e^{z(X^r+D)}\right) &=& \ds \exp\left(
{\textstyle \frac{1}{r+1}}\left[(X+z)^{r+1}-X^{r+1}\right]\right)\cdot e^{zD},
\end{array}\right.
\]
which gives back~\eqref{inv2} and~\eqref{geninv}, when $m=1,2$.

\smallskip

\begin{note} \emph{Another PDE.}
The usual consequences for PDEs also hold, with the further simplification that
$e^{tD}$ is a shift and that,
for the quadratic case at least, the exponential of the Laplacian can be made 
explicit in the Fourier basis of complex exponentials. 
For instance, the general PDE schema
\[
\frac{\partial F}{\partial t}=\frac{\partial F}{\partial x}+a(x)F,
\]
admits the solution
\[
F(x,t)=e^{Q(x,t)} f(x+t), \qquad Q(x,t):=\int_0^t a(x+w)\, dw.
\]
In particular, $a(x)=x^2$ corresponds to $Q(x,t)=t^3/3+tx^2+t^2x$.
\end{note}

%
%
%
%

\section{\bf The special quadratic form $(XD)$, set partitions, and product forms}\label{bell-sec}

This section is dedicated to the normal ordering of 
the powers $(XD)^n$, a problem which has been recognized for  a long
time to be tightly coupled to set partitions and Stirling numbers of
the second kind. In Appendix~\ref{scherk-ap}, we summarize the main results 
contained in the rather remarkable thesis  memoir of Heinrich {\sc Scherk},
defended in~1823: the reduction of $(XD)^n$ figures there explicitly!
We also consider, in this section, the related quadratic form 
$XD+g(X+D)$. Finally, we show that the combinatorial approach advocated here
extends rather easily to powers of operators such as $X^2D^2$, $X^3D^3$,
and so on, which relates to several combinatorial enumeration
problems of independent interest.

\subsection{The form $(XD)$ and set partitions.}\label{xd-subsec}

By the general isomorphism theorem, 
the operator $(XD)$ corresponds to a gate with one input and one
output:
\[
\hbox{\Img{2}{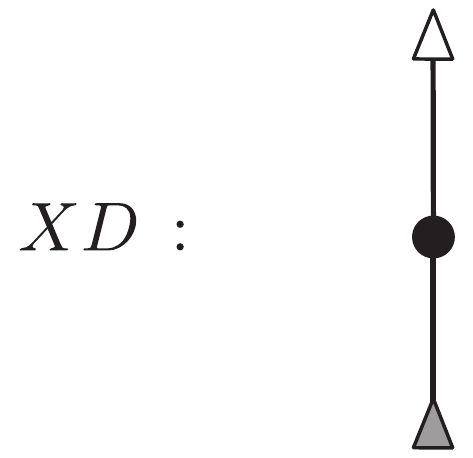}}
\]
With these, we can form chains where the input of an $XD$--gate 
is connected to the output of a previously arrived $XD$--gate.
When time stamps are put on the vertices of the gates, a connected component
becomes an increasing linear graph 
(see, e.g., \cite[p.~99]{FlSe09}),
in the sense that vertices are linearly arranged and labelled in increasing order,
with the additional condition that size needs to be at least one: see Figure~\ref{xd-fig}. Note that edges between adjacent elements in a connected component are redundant and hence each chain can be represented by the nonempty set of labels. Here is a particular representation:
\[
\hbox{\Img{5.5}{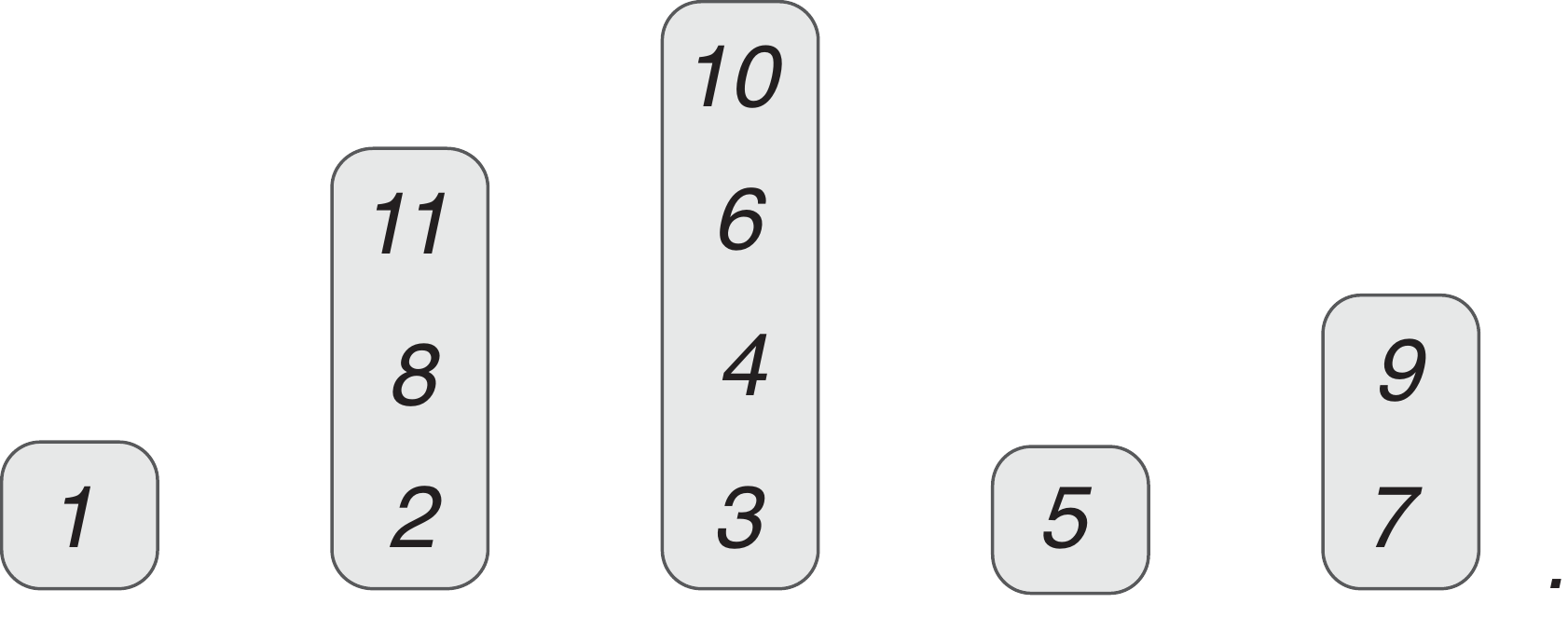}}
\]
Under this form one recognizes the classical structure of set partitions,
where a partition of a set is a subdivision of the elements of the set
into  indistinguishable non-empty \emph{classes}, 
also known as ``blocks''~\cite{Comtet74,FlSe09,GrKnPa89}.
(Here the components have been presented in  increasing order of their leading elements.)

\begin{figure}\small
\begin{center}
\Img{10}{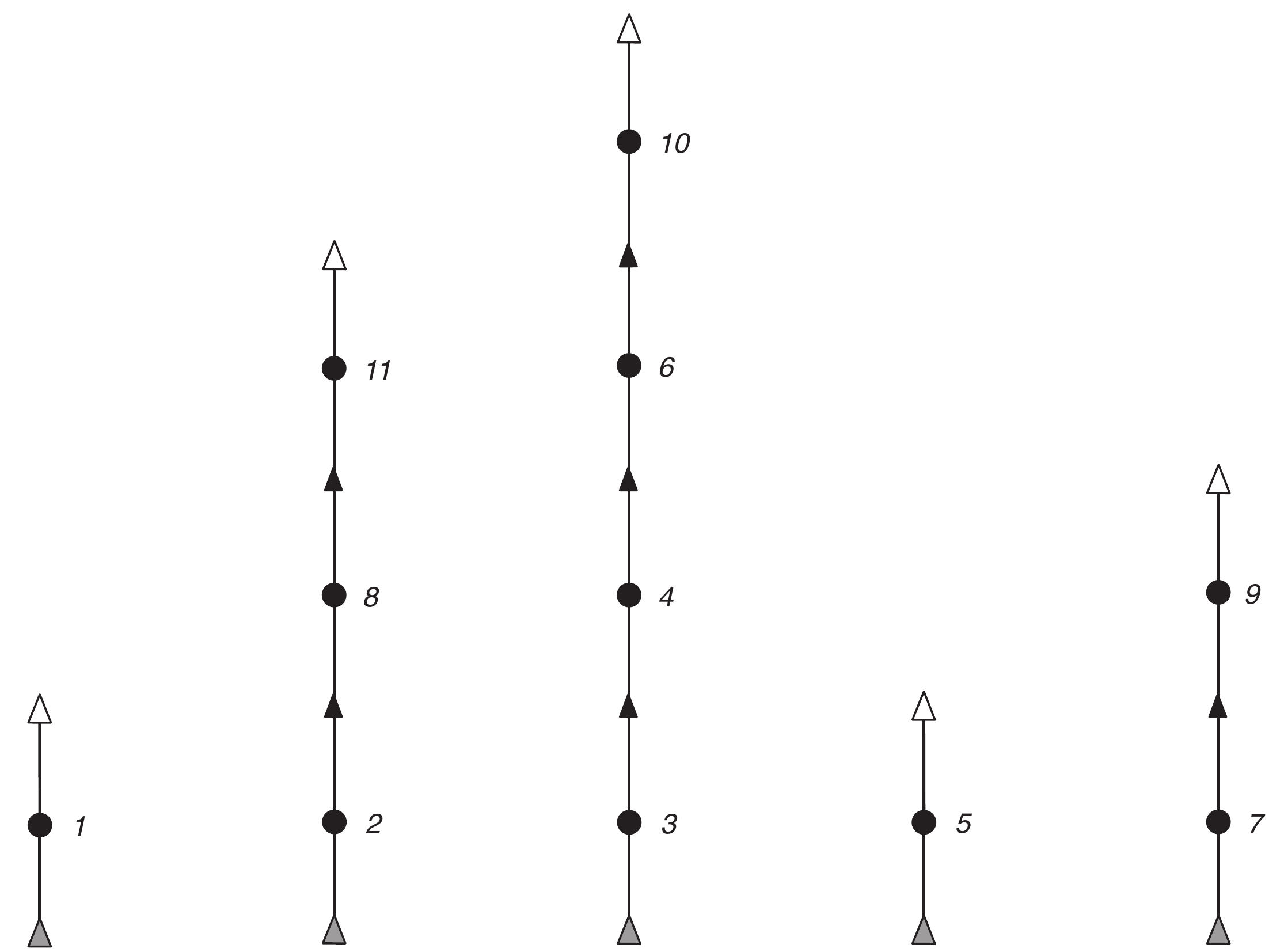}
\end{center}
\caption{\label{xd-fig}\small
A particular diagram associated with $(XD)$.}
\end{figure}

Components can be equivalently regarded as non-empty 
unordered collections 
of labels (non-empty ``urns'' in the terminology of~\cite[p.~99]{FlSe09});
that is, they are specified by $\cal K=\set_{\ge1}(\z)$.
An arbitrary graph built out of $XD$--gates is then an unordered collection
of such components, so that the corresponding class is $\set(\cal K)$.
With $z$ marking the size of a graph 
(its number of gates, equivalently, vertices) and~$u$ marking 
the number of  connected components
(here equal to the number of inputs \emph{and} to the
the number of outputs), we thus have for diagrams the specification
\[
\cal G=\set\left(u \set_{\ge1}(\z)\right),
\]
hence the corresponding generating function
\begin{equation}\label{bellgf}
G(z,u)=e^{u(e^z-1)}.
\end{equation}

The generating functions solve the corresponding
enumeration problems. By
an expansion of~\eqref{bellgf} taken at~$u=1$,
the total number~$\varpi_n$ of partitions of a set of cardinality~$n$ 
satisfies  the ``Dobi\'nski relation''
\begin{equation}\label{dobo}
\varpi_n = n![z^n]e^{e^z-1}=e^{-1}\sum_{\ell\ge0} \frac{\ell^n}{\ell!},
\end{equation}
a quantity known in combinatorics as a \emph{Bell number}~\cite{Comtet74}.
The number of partitions of~$n$ elements into $k$ classes
is the \emph{Stirling number of the second kind}, 
nowadays usually denoted by~$\stirp{n}{k}$,
whose value is~\cite{Comtet74,FlSe09,GrKnPa89}
\begin{equation}\label{stirdef}
\stirp{n}{k} = n![z^nu^k]e^{u(e^z-1)}= \frac{n!}{k!}[z^n](e^z-1)^k = \frac{1}{k!} 
\sum_{j=0}^n \binom{k}{j}(-1)^{k-j} j^n,
\end{equation}
as can be immediately verified from a series expansion of~\eqref{dobo}.

Back to the normal ordering problem, each connected component of a diagram, i.e., each chain 
of $XD$--gates, has one input and one output, so that the balance of
a single component is of the form~$XD$, the balance of~$k$ components being 
accordingly $X^kD^k$.
We then have:

\begin{proposition} 
The normal ordering associated with $e^{z(XD)}$ involves the Stirling partition numbers:
\begin{equation}\label{stirr}
\begin{array}{lll}
\ds \nor\left(e^{z(XD)}\right)& =& \ds \left. e^{u(e^z-1)}\right|_{\langle u^k\mapsto X^kD^k \rangle}\\
&=& \ds \sum_{n\ge0} \frac {z^n}{n!}\left[ \sum_{k} \stirp{n}{k}X^k D^k\right]
\quad=\quad \sum_{k\ge0} \frac{1}{k!}\left(e^z-1\right)^k X^kD^k.
\end{array}
\end{equation}
\end{proposition}

\begin{note} \emph{Yet another PDE.}
Proceeding as before, we obtain from~\eqref{stirr}, that the PDE
\[
\frac{\partial F}{\partial t}=x\frac{\partial F}{\partial x}
\]
with initial condition~$F(x,0)=f(x)$ possesses the solution
\[
F(t,x)=\sum_{k\ge0} (e^t-1)^k X^k \frac{D^k}{k!} f(x)=
f(x(e^t-1)+x)=f(xe^t),\]
whose  \emph{a posteriori} verification is immediate.
\end{note}

\begin{note} \emph{History.}
The origin of the Stirlng partition numbers $\stirp{n}{k}$ and their cognates, the Sirling cycle numbers
$\stirc{n}{p}$,
lies in eighteenth century calculus. The classical way of defining them is as
the coefficients expressing,
in the polynomial algebra~$\C[x]$,
 the change of basis  between the canonical basis~$x^n$
and the factorial basis $x^{\underline n}\equiv x(x-1)\cdots(x-n+1)$
or its trivial variant $x^{\overline n}=x(x+1)\cdots (x+n-1)$. Indeed, one has~\cite[pp.~248--249]{GrKnPa89}:
\begin{equation}\label{bothdef}
x^n = \sum_k \stirp{n}{k} x^{\underline k},\qquad
x^{\overline n}=\sum_k \stirc{n}{k} x^k\,.
\end{equation}
(In a different circle of ideas, Stirling numbers are known to probabilists as a way of
relating factorial moments and standard power moments~\cite[p.~47]{DaBa62}.)
Upon examining the effect on coefficients of power series,
it is then  seen that Stirling coefficients serve 
to express the connection between powers 
$(x\partial_x)^n$ and standard derivatives, via
 $x^k\partial_x^k$, which is exactly what we derived
by elementary combinatorics in~\eqref{stirr}.
An equivalent operator formulation appears as an exercise\footnote{
This interesting exercise also contains the reduction of $(X^{a+1}D)^n$ and its connection with
$\{(1-at)^{1/a}-1\}$  [with the minor typo of a missing exponent of~$n$], 
which we shall encounter later in Section~\ref{semilin-sec};
see also Riordan~\cite[\S6.6]{Riordan68} for related operational calculus derivations.}
in Comtet's book~\cite[Ex.~2, p.~220]{Comtet74}. (Some of these 
properties were already familiar to Scherk in 1823; 
cf Appendix~\ref{scherk-ap}, p.~\pageref{scherk-ap}.)

In the context of quantum physics, the combinatorial connections between 
the normal ordering of $(a^\dagger a)^n$ and Stirling numbers 
have been recognized early: see, for instance, the papers
of Wilcox~\cite[p.~978]{Wilcox67}, for an algebraic perspective,
 and especially Katriel~\cite{Katriel74}, for the combinatorial connection.
(See also references therein to earlier works by Schwinger and others.)
This thread has given rise to a large body of subsequent literature. 
In particular, Bender, Brody, and Meister explicitly discuss Bell 
numbers in the context of  diagrams in~\cite{BeBrMe99},
although their approach differs substantially from ours, as they focus on
interpretations of the schema $e^{\phi(D)}e^{\xi(X)}$.
The relations between set partitions, diagrams and Stirling
numbers are used as a lead example by Baez and Dolan~\cite{BaDo01}
 to illustrate the process of ``decategorifying''~(!)
the creation--annihilation theory. 
We examine below, in Subsection~\ref{x2d2-subsec}, further interesting extensions 
of the combinatorial approach that are due to Blasiak
\emph{et al}.
\end{note}

\begin{figure}\small
\img{12.3}{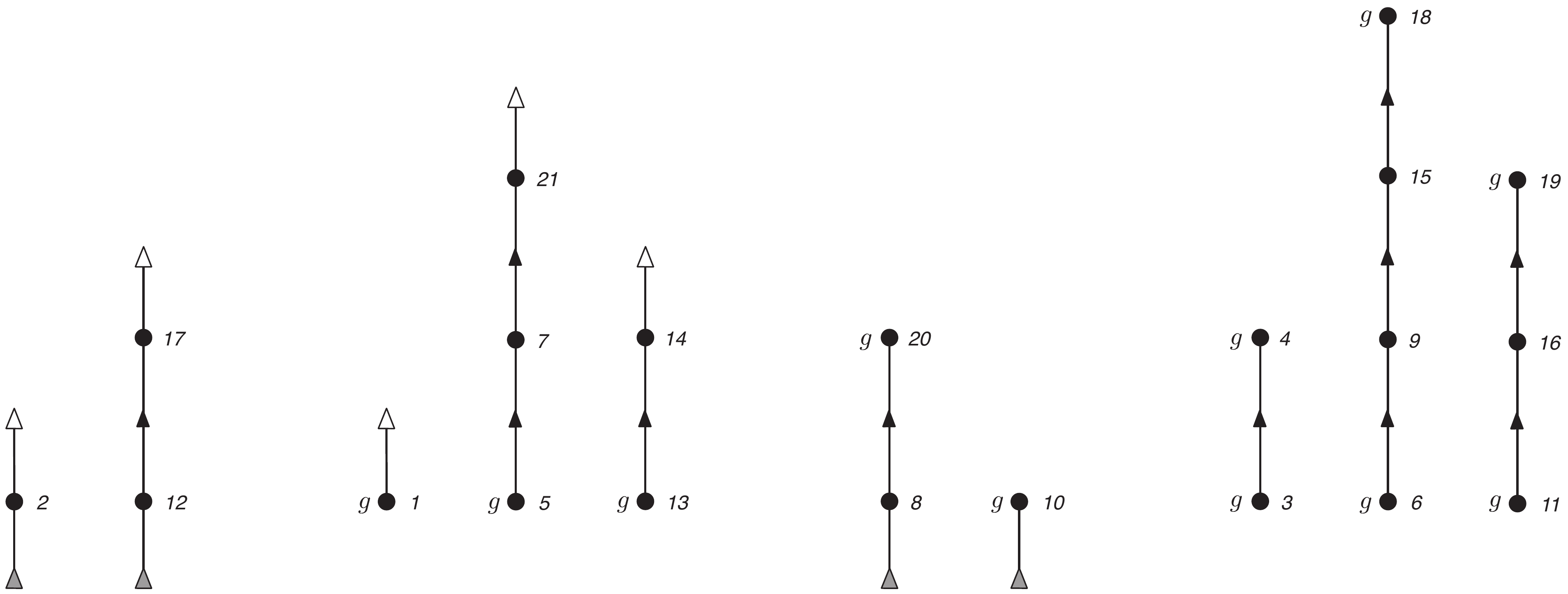}
\caption{\label{xdxd-fig}\small
A sample diagram associated with $(XD+g(X+D))$.}

\end{figure}

\begin{note} \emph{Normal forms relative to $(XD+g(X+D))$}.
We briefly discuss this case, which is treated by 
Wilcox~\cite{Wilcox67}  and Louisell~\cite{Louisell90}
by means of operator calculus,
as it
illustrates the versatility of the combinatorial method.
The gates are now of the three types $X$, $D$ and $XD$:
\[
\hbox{\Img{8.5}{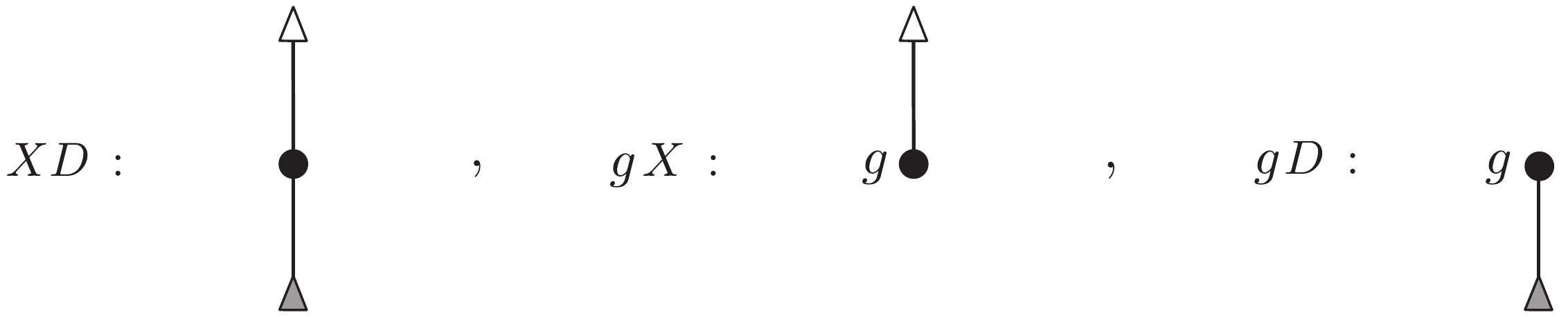}}
\]
The connected components are similar 
to those relative to $(XD)$, with, in addition, the possibility that a 
line graph can be ``capped'' with a~$D$, or ``cupped'' with an~$X$, 
or both capped and cupped; they are consequently of \emph{four}
possible types. 
Hence, one can write directly the specification of graphs as
\[
\cal G=\set\left((uv+ug+vg)\set_{\ge1}(\cal Z)+g^2\set_{\ge2}(\cal Z)\right),
\]
where~$u$ and $v$ mark inputs and outputs, respectively.
The corresponding multivariate EGF is then
\[
G(z;u,v)=e^{(x+g)(y+g)(e^z-1)}e^{-g^2z}.
\]
(This case prefigures the planted-tree construction of
Subsection~\ref{plant-subsec},
p.~\pageref{plant-subsec}.)
\end{note}


\subsection{The product form $(X^2D^2)$.} \label{x2d2-subsec}
This subsection and the next one 
serve to revisit the combinatorial works 
of Blasiak, Horzela, Penson, Solomon,
and coauthors~\cite{Blasiak05,BlHoPeSo06,BlPeSo03,BlPeSo03b,MeBlPe05}.
See also Schork's synthetic study~\cite{Schork03}, which 
furthermore includes some $q$-analogues.
To avoid cumbersome notations, we start with a discussion 
of the operator $X^2D^2$. What is at stake is understanding the structure of reductions
such as
\begin{equation}\label{x2d2}
\left\{\begin{array}{lll}
(X^2D^2) &=& X^2D^2 \\
(X^2D^2)^2 &=& 2X^2D^2 + 4X^3D^3 + X^4D^2\\
(X^2D^2)^3 &=& 4X^2D^2+32X^3D^3+38X^4D^4 + 12X^5D^5 +
X^6D^6 ,
\end{array}\right.
\end{equation}
where the sums of the coefficients 
form a sequence $\varpi_n^{2,2}$, which starts as 
\begin{equation}\label{bellnum}
(\varpi_n^{2,2})=1,7,87,1657,43833,1515903, 65766991, 3473600465,\ldots\,.
\end{equation}
These numbers appear as 
Sequence~\OEIS{A020556} in Sloane's \emph{Online Encyclopedia
of Integer Sequences}~\cite{Sloane08}, henceforth abbreviated as ``\emph{OEIS}''.
The triangle of coefficients in the expansion of powers of~$(X^2D^2)$ 
as in~\eqref{x2d2} is Sequence~\OEIS{A078739}.

By the combinatorial isomorphism, we now have one kind of gate, namely, $X^2D^2$;
that is, a gate has two input edges and two output edges. 
Such a gate thus picks up two, one or none of the previously existing outgoing edges ``prolonging'' the hooked edge and generating respectively none, one or two new links; the local balance of 
the number of links (edges) being clearly null.
In other words, we can view a gate as simply ``propagating'' links.
It is then of advantage to align edges vertically and
represent an $X^2D^2$--gate by a horizontal vector:
the orientation conventionally serves to distinguish the first input from the second input;
these inputs can be conveniently tagged by a $\oplus$ and a $\ominus$
sign, respectively,
as in the following diagram:
\begin{equation}\label{scaffold-eqn}
\hbox{\Img{8}{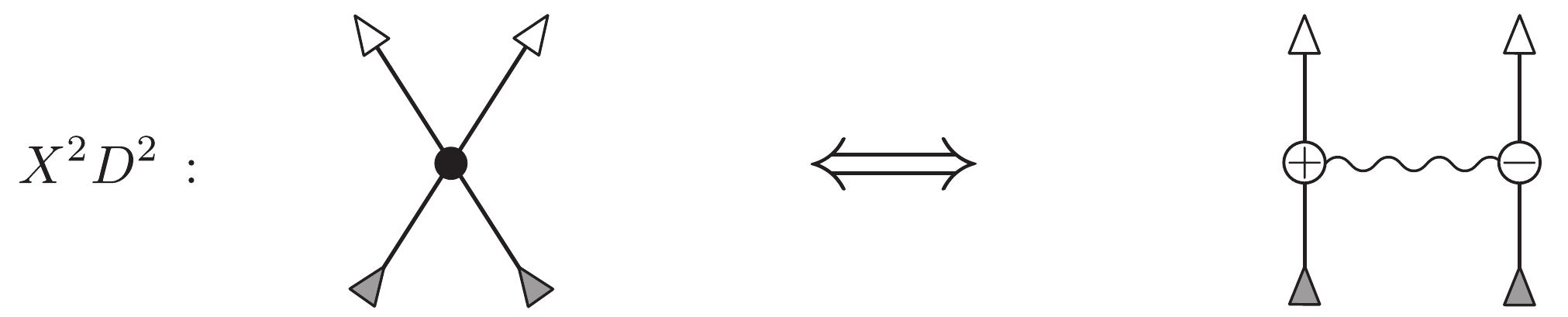}}
\end{equation}
\begin{figure}\small
\begin{center}
\Img{12}{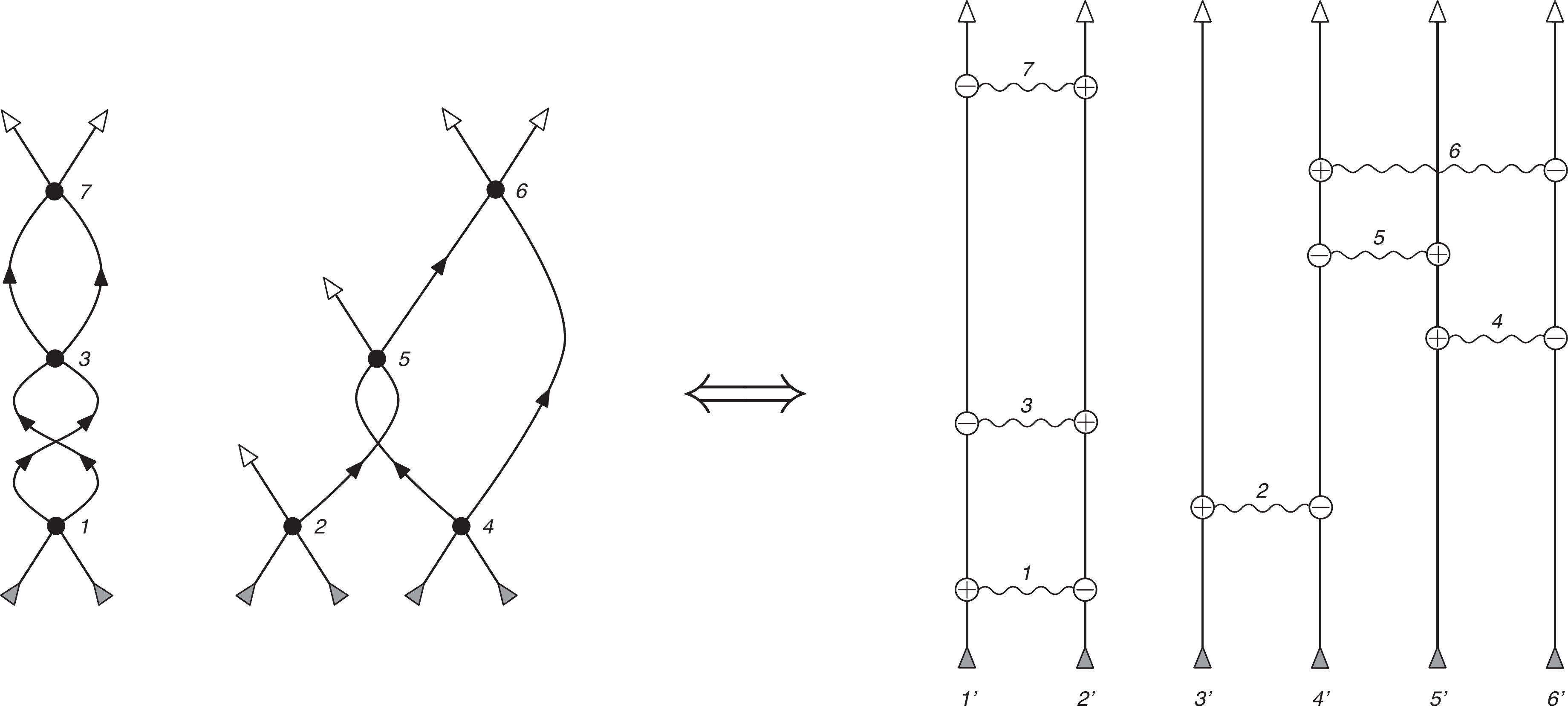}
\end{center}

\vspace*{-3truemm}
\caption{\label{scaffold-fig} \small
The scaffolding representation of an~$X^2D^2$ graph (and the induced canonical labelling of lines).
}
\end{figure}
When such gates are stacked on top of one another, they  then give rise to 
piles of vectors that can be viewed as (weird!) ``scaffoldings''; see Figure~\ref{scaffold-fig}.
(This structure is
loosely evocative of Viennot's theory of ``heaps of pieces''~\cite{Viennot86}.)

\def\x{\zeta}
\parag{Bilabelled structures.}
Let $\cal S_{n,k}$ be the collection of all diagrams comprised of~$n$ gates of type $X^2D^2$
that have~$k$ inputs and let
$\stirp{n}{k}_{\!2,2}\equiv S_{n,k}$ be the corresponding cardinality.
We consider the class $\cal S_k=\bigcup_n \cal S_{n,k}$,
which it is now our goal to enumerate. This can be achieved by considering 
\emph{bilabelled} objects, that is, objects carrying two kinds of labels.
In the case at hand, we consider the extended class~$\hat S_k$, where \emph{inputs}
bear independent labels, $1',\ldots,k'$, 
which are indicated by a prime (``primed''),
in order to distinguish them from the usual 
gate/vertex labels. The labelled product extends trivially---distribute independently
both types of labels.
 It translates straightforwardly as a product of 
\emph{biexponential generating functions}, of the form 
\begin{equation}\label{biexp}
f(z,\zeta)=\sum_{n,\nu}f_{n,\nu} \frac{z^n}{n!} \frac{\zeta^{\nu}}{\nu!},
\end{equation}
where the powers of~$z$ and~$\zeta$ record, respectively, the 
number of standard and primed labels.

Let $\hat{\cal S}_{n,k}$ be   the class of diagrams  of size~$n$
with~$k$  labelled inputs.  In terms    of cardinalities,  we have   $\hat
S_{n,k}=k!  S_{n,k}$.  Indeed,  in a standard, singly
labelled   diagram   of~$\cal    S_{n,k}$,   all  inputs   are
\emph{distinguishable}, say, according  to the time they  are first used, and
by considering the first input, tagged ``$\oplus$'',
 before the second one, tagged ``$\ominus$'' (see Figure~\ref{scaffold-fig});
there are thus exactly $k!$ ways to superimpose an input labelling on a diagram.
We then claim the identity
\begin{equation}\label{main22}
\cal Q
=\set(\cal Z')\star \left(\bigcup_k \hat{\cal S}_k\right).
\end{equation}
Here, $\cal Z'$ represents an  atom 
that carries a primed label, but no standard (``unprimed'')  label;
the class~$\cal Q$ is a relaxed version of~$\bigcup_k \hat{\cal S}_k$
in which extra (primed--labelled) inputs are allowed; 
the labelled  product `$\star$' is, in accordance with our previous
discussion, taken to distribute both kinds
of labels. 

\begin{figure}\small

\begin{center}
\Img{12.3}{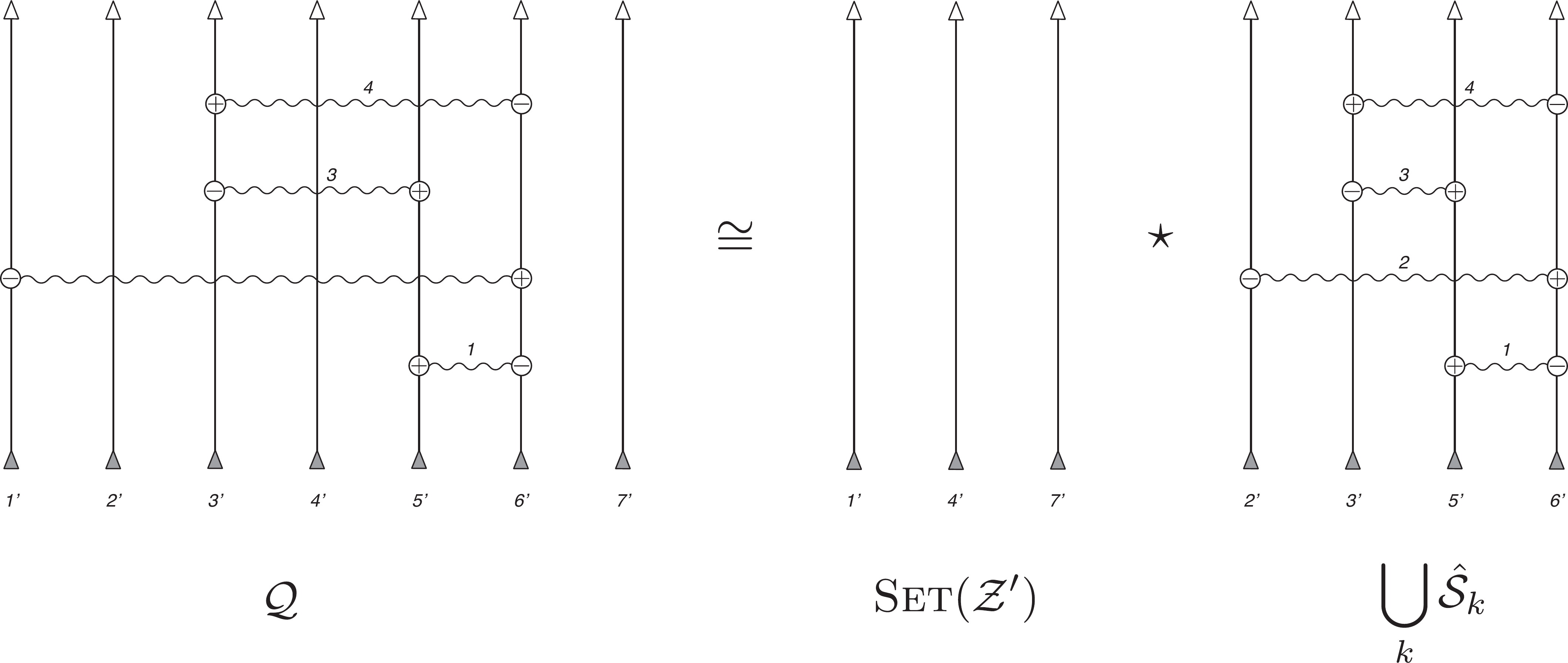}
\end{center}

\vspace*{-0.3truecm}

\caption{\label{scaffoldiso-fig} \small
A graphical rendering of the isomorphism expressed by~\eqref{main22}:
on the left appear ``relaxed'' scaffoldings; on the right the product form
that isolates unused inputs and a proper scaffolding.
}
\end{figure}
In order to justify~\eqref{main22}, observe that the left hand side
describes  all possible  sequences  of choices   of  pairs of distinct
prime-tagged   elements amongst a   set of  cardinality~$k$, with  the
possibility of some of the elements to be unused; the factor $\set(\cal Z')$
in the right hand side simply gathers all these unused inputs, with the sum there
corresponding to legal (bilabelled) diagrams; see Figure~\ref{scaffoldiso-fig}.

We next effect the translation into bivariate generating functions,
here taken to be of the form~\eqref{biexp}. Thus, $z$ records gates 
and~$\zeta$ records inputs. (Note the subtle difference between~$\zeta$, as
a carrier of primed labels, and~$u$, as a plain unlabelled marker 
in earlier developments relative to $(XD)$.)
By the definition of $\cal Q$, we have
\[
Q(z,{\x})=\sum_{n,k=0}^\infty \left[k\cdot(k-1)\right]^n \frac{{\x}^k}{k!}\frac{z^n}{n!}.
\]
Equation~\eqref{main22} then translates as the biexponential 
generating function identity
\[
e^{\x} \sum_{n,k=0}^\infty k!S_{n,k}\frac{{\x}^k}{k!}\frac{z^n}{n!}
=\sum_{n,k=0}^\infty \left[k\cdot(k-1)\right]^n \frac{{\x}^k}{k!}\frac{z^n}{n!},
\]
which solves to give explicitly (with $\stirp{n}{k}_{\!2,2}\equiv S_{n,k}$)
\begin{equation}\label{main23}
\sum_{n,k=0}^\infty \stirp{n}{k}_{\!2,2} {\x}^k\frac{z^n}{n!}
=e^{-{\x}}\sum_{n,k=0}^\infty \left[k\cdot(k-1)\right]^n \frac{{\x}^k}{k!}\frac{z^n}{n!}.
\end{equation}
Collecting the coefficient of~$z^n$ and expanding the result
as a binomial convolution leads to
the following statement~\cite{BlPeSo03b}.

\begin{proposition} The normal ordering associated 
to $e^{z(X^2D^2)}$ satisfies
\begin{equation}\label{nor22}
\nor\left(e^{z(X^2D^2)}\right)
=\sum_{n=0}^\infty \frac{z^n}{n!} \sum_{k} \stirp{n}{k}_{\!2,2} X^kD^k,
\end{equation}
where  the generalized Stirling numbers $\stirp{n}{k}_{\!2,2}$ are
\begin{equation}\label{stir22}
\stirp{n}{k}_{\!2,2}=\frac{1}{k!}\sum_{j=2}^k (-1)^{k-j}\binom{k}{j}[j\cdot (j-1)]^n.
\end{equation}
\end{proposition}
\noindent
This calculation also implies a Dobi{\'n}ski-type formula for the generalized Bell
numbers of~\eqref{bellnum}, 
\begin{equation}\label{dob22}
\varpi_n^{2,2}=e^{-1}\sum_{\ell\ge2} \frac{[\ell(\ell-1)]^n}{\ell!},
\end{equation}
an identity given in~\cite{BlPeSo03b} which 
is further generalized in Proposition~\ref{xrdr-prop} below.
The  formulae~\eqref{stir22} and~\eqref{dob22}  are seen neatly  to  extend their
standard counterparts\footnote{ 
A similar, simplified, treatment  can also  be
inflicted  to the   usual Stirling  numbers.  Let  $\cal  W$ be    the
bilabelled class of  all   finite words  over  the  labelled  alphabet
$\{1',2',\dots\}$.    The  biexponential   generating   function  is
$W(z,{\x})=\sum_{n,k} k^n ({\x}^k/k!)(z^n/n!)=\exp({\x}\cdot  e^z)$. With~$\cal
U$ representing a stock of ``unused'' letters,
we have the identity  $\cal W=\cal U\star \cal R$, where
$\cal R$ is the  class of ``gapless'' words,  in the sense that  their
letters  form  an initial  segment of   $\{1',2',\ldots\}$. (The
class~$\cal   R$  is   isomorphic  to  the  class   of ``surjections''~\cite[\S II.3]{FlSe09},  
also  known as   ordered set  partitions or
preferential arrangements.)  We  then have  $W(z,{\x})=e^{\x} R(z,{\x})$, which
implies $R=e^{-{\x}}W$ giving back $R(z,{\x})=\exp({\x}(e^z-1))$, which is both
a biexponential     generating function  of surjections    and  the  usual bivariate
exponential  generating  function  of set  partitions; cf Equation~\eqref{bellgf}, with the substitution~$\zeta\mapsto u$.}
in~\eqref{stirdef} and~\eqref{dobo}. Finally, expanding by the binomial
theorem the quantities~$[j(j-1)]^n$ in~\eqref{stir22}
or $[\ell(\ell-1)]^n$ in~\eqref{dob22}
serves to relate the generalized and basic families of numbers:
\[
\stirp{n}{k}_{\!2,2}=\sum_{r=0}^n (-1)^{n-r}\binom{n}{r}\stirp{n+r}{k},
\qquad
\varpi_n^{2,2}=\sum_{r=0}^n (-1)^{n-r}\binom{n}{r}\varpi_{n+r}.
\]

\begin{note} \emph{Matrix enumeration and diagrams.}
\def\M{\cal M}
Let $\M_{n,k}^{\pm}$ be the collection of matrices with~$n$ rows and~$k$~columns,
having all their entries in~$\{0,\pm1\}$, such that the following conditions are met:
$(i)$~there is exactly one $+1$--entry and one $-1$--entry in each row;
$(ii)$~no column consists solely of $0$s. The corresponding cardinality
satisfies
\[
M^{\pm}_{n,k}=k!\stirp{n}{k}_{2,2},
\]
as the set of matrices is clearly isomorphic to diagrams relative to~$X^2D^2$ 
with \emph{labelled} inputs (the column numbers)---these 
constitute the class $\hat S_k$ above.

Similarly, let $\M_{n,k}^+$ be the collection of matrices with~$n$ rows and~$k$~columns,
having all their entries in~$\{0,1\}$, such that the following conditions are met:
$(i)$~there are exactly two $1$--entries in each row;
$(ii)$~no column consists solely of $0$s. The corresponding cardinality
satisfies
\[
M^{+}_{n,k}= 2^{-n}k!\stirp{n}{k}_{2,2},
\]
since this is equivalent to killing the orientation of gates.

The counts of matrices without restrictions on the number of rows are then,
respectively,
\[
M_n^{\pm}=\sum_{k=2}^{2n} k!\stirp{n}{k}_{2,2}, \qquad M^{+}_{n,k}= \frac{1}{2^n}\sum_{k=2}^{2n} \
k!\stirp{n}{k}_{2,2},
\]
The sequence of values $(M^+_n)_{n\ge1}$ starts as
\[
1,13,409,23917,2244361,308682013,58514835289.
\]
It  has number \OEIS{A055203}, with 
a description, ``The number of different relations between $n$ intervals on a line'',
which is easily related to our previous discussion.
Also, perhaps more picturesquely: 
``Imagine you have $n$ events of non-zero duration, in how many different 
ways could those events overlap in time?''. (The \emph{OEIS} refers to some 
unpublished work of S.~Schwer relative to temporal logics and formal linguistics.)
\end{note}

\begin{note} \emph{Coupon collection with group drawings.} \label{coudra-note}
In probability theory, the coupon collector problem asks for the distribution of
the number of samplings with replacement from a finite set~$\cal E$, which are needed till a 
complete collection of the elements of~$\cal E$ is obtained. If~$T$ is this random time
and~$m$ is the cardinality of~$\cal E$, one has
\begin{equation}\label{cca}
\Pr_m[T\le n]=\frac{m!}{m^n}\stirp{n}{m}=\sum_{k=0}^m
 \binom{m}{k}(-1)^k \left(\frac{k}{m}\right)^n,
\end{equation}
since the event $\{T\le n\}$ corresponds to a sequence of~$n$ choices
(each with~$m$ possibilities),
such that each of the~$m$ possibilities is chosen at least once. 
(The $m!$ factor corresponds to the 
fact that \emph{ordered} set partitions are to be considered,
since the~$m$ elements of~$\cal E$ are distinguishable.)
See Feller's book~\cite{Feller68} for a classical 
probabilistic derivation
and~\cite[pp.~116--117]{FlSe09}
for a treatment cast within the framework of analytic combinatorics.
It is also well known that the expected time for a complete collection satisfies
\[
\Ex_m[T]=1+\frac12+\cdots+\frac1m=\operatorname{H}_m=\log m+\gamma+o(1).
\]

The coupon collector problem \emph{with group drawings} is the variant problem,
where one now draws from~$\cal E$ in \emph{groups} of~$r$ \emph{distinct} elements. Here~$r=2$.
This problem is, once more, a variant of those previously considered, and one finds
\[
\Pr_m[T\le n]=\frac{m!}{(m(m-1))^n}\stirp{n}{m}_{\!2,2}=
\sum_k \binom{m}{k}(-1)^k \left(\frac{k(k-1)}{m(m-1)}\right)^n,
\]
which exhibits a pleasant similarity to the basic case~\eqref{cca}.
This formula was obtained by Stadje~\cite{Stadje90}
by means of a subtle use of the inclusion--exclusion
principle combined with suitable combinatorial identities---the derivation above 
seems to us much more transparent.
Analytic methods then make it possible to derive the expected value~\cite{DuFlRoTa07}
\[
\Ex_m[T] = 
\frac{m(m-1)}{2m-1}\left(H_m+\frac{1}{2m-1}-\frac{(-1)^m}{(m+1)\binom{2m-1}{m+1}}\right)
=\frac{1}{2}\log m+ \frac{\gamma}{2}+o(1).
\]
This analysis is itself intimately related to that of the number of isolated vertices under
the Erd\H os--R\'enyi random graph model~\cite[\S7.1]{Bollobas85}.
\end{note}

\begin{note} \emph{Set partitions and contiguities.}
Here is yet another interpretation of generalized Stirling numbers.
Given a partition $\pi$ of $\{1,\ldots,n\}$, a pair
$(j,j+1)$ is called a \emph{contiguity} of~$\pi$ if $j$ and $j+1$ belong to the same block.
By extension, we also say that the number~$j$ is a contiguity.

In a diagram, we can view the gate bearing label $j$ as being associated to 
two  inputs labelled $2j-1$ and $2j$. In this way, the  set of labels becomes~$\{1,\ldots,2n\}$.
We now consider partitions of $\{1,\ldots,2n\}$ and, given a gate,
group in a single block all inputs that correspond to the same thread
(i.e., are ``vertically aligned''). The rules corresponding to the formation of diagrams then imply
the  following: \emph{diagrams formed from $X^2D^2$ gates that are of size~$n$ are
in bijective correspondence with set partitions of size~$2n$, where
contiguities of odd values 1,3,5,\ldots, are forbidden.} Thus the total
number of such set partitions is the generalized Bell number~$\varpi_n^{2,2}$ of~\eqref{dob22}.

This raises the question of enumerating set partitions  
according to the number of contiguities. The number $\widetilde B_n$ of those
\emph{without} any contiguity turns out to be the shifted
(standard) Bell number $\varpi_{n-1}$, where $\varpi_n$ is defined in~\eqref{dobo}.
This can be shown from the relation
\[
\widetilde {\cal B}^{\Box}\star\cal U =\cal B,\qquad \cal U=\set(\cal Z),
\]
which expresses the fact that an arbitrary partition (the class~$\cal B$
of \emph{all} set partitions) can be obtained
from a contiguity-free partition (the class~$\widetilde{\cal B}$)
by gluing extra atoms (from~$\cal U$) \emph{after} 
their immediate predecessors (hence the appearance of 
the box operator in~$\widetilde {\cal B}^{\Box}$).
From this specification and the general translation rules seen earlier, 
a simple computation provides the EGF: by~\eqref{boxp}, we have
(with here $B_n\equiv\varpi_n$, so that $B(z)=e^{e^z-1}$)
\[
\int \widetilde B'(z) e^z=B(z),\qquad\hbox{implying}\quad
\widetilde B(z)= \int B(z)=\sum_{n\ge1} B_{n-1} \frac{z^n}{n!}.
\]
This solves the enumeration problem, as $\widetilde B_{n}=B_{n-1}$.
It is  seen  from here  that  
\[
\int \widetilde  B'(z)e^{uz}=\int e^{e^z+uz-1}
\]
enumerates 
set partitions according the    number of contiguities,
marked by~$u$. 
(An analogous process 
serves to relate  derangements to  the  class of all
permutations.)
\end{note}

In summary, the orbit of equivalences mentioned above (matrices, coupon
collector, set partitions) justifies considering
the generalized Stirling numbers $\stirp{n}{k}_{2,2}$ as basic quantities
of combinatorial analysis.

%
%
%
%
%

\subsection{Higher order forms~$(X^rD^r)$.}

The discussion relative to $X^2D^2$ extends to $X^rD^r$ and more generally to
\emph{balanced polynomials}, which are of the form 
\begin{equation}\label{bal0}
\frak{h}=\sum_{\ell=1}^r \eta_\ell X^\ell D^\ell,
\end{equation}
where $\eta=(\eta_j)$ is a sequence of arbitrary coefficients. We state:

\begin{proposition}[\cite{Blasiak05,BlPeSo03,BlPeSo03b}]\label{xrdr-prop}
Let $\frak{h}$ be a balanced polynomial in the sense of~\eqref{bal0}.
The $\frak{h}$-Stirling numbers  defined as the coefficients in the expansion
\[
\frak{h}^n =\sum_k {n \brace k}_{\frak h} X^kD^k
\]
admit the explicit form
\begin{equation}\label{bal1}
{n \brace k}_{\frak{h}}
=\frac{1}{k!}\sum_{j=1}^k (-1)^{k-j}\binom{k}{j} h(j)^n,
\qquad h(x):=\sum_{\ell\ge1} \eta_\ell \cdot x(x-1)\cdots(x-\ell+1).
\end{equation}
\end{proposition}
\begin{proof}
A combinatorial proof based on gates and graphs proceeds along the same line as what has been done
for $X^2D^2$. See the diagram of~\eqref{scaffold-eqn} and
Figure~\ref{scaffoldiso-fig}. What happens is that each
rung of a scaffolding is now
comprised of an $\ell$-arrangement (for some~$\ell$ satisfying $1\le \ell\le r$) of labelled inputs:
the number of possibilities for each rung to be connected to~$j$ inputs is then exactly $h(j)$.
Then, with~$\widehat{\cal S}_k$
the collection of scaffoldings with~$k$ labelled inputs and~$\cal Q$ the relaxed
scaffoldings that may have unused labelled inputs, the
relation~\eqref{main22} still holds
(see also Figure~\ref{scaffoldiso-fig}).
We then have the equality of double exponential generating functions
\begin{equation}\label{bal15}
e^x\cdot  \sum_{n,k} k! {n \brace k}_{\frak{h}} \frac{x^k}{k!}\frac{z^n}{n!}=
\sum_{n,k} h(j)^n \frac{x^j}{j!}\frac{z^n}{n!}.
\end{equation}
This last relation is equivalent to the statement upon extracting the coefficient of~$x^kz^n$.
\end{proof}

The note below provides
a typical alternative derivation taken from~\cite{Blasiak05,BlPeSo03,BlPeSo03b}
and based on simple algebra; see also~\cite[\S3]{Schork03}.

\begin{note} \emph{Algebraic reduction of $(X^rD^r)^n$.} \label{algXrDr-note}
The idea is to apply $(X^r D^r)^n$ to the exponential function~$e^x$.
We write ${n\brace k}_{r,r}:={n\brace k}_{\frak{h}}$, for $\frak{h}=X^rD^r$.
On the one hand, we have
\begin{equation}\label{bal20}
\left(X^rD^r\right)^n e^x =\sum_k {n \brace k}_{r,r} X^kD^k e^x
=\sum_k {n \brace k}_{r,r} x^k e^x,
\end{equation}
by the definition of the ${n\brace k}_{r,r}$ numbers and the fact that $D^k e^x=e^x$. On the other hand,
\begin{equation}\label{bal21}
\left(X^rD^r\right)^n e^x =\sum_j \left[j(j-1)\cdots (j-r+1)\right]^n \frac{x^j}{j !},
\end{equation}
by the Taylor expansion of~$e^x$. The comparison of~\eqref{bal20} and~\eqref{bal21}
yields
\[
 \sum_k {n \brace k}_{r,r} x^k =e^{-x} \sum_j \left[j(j-1)\cdots (j-r+1)\right]^n \frac{x^j}{j!},
\]
which is equivalent to the statement of Proposition~\ref{xrdr-prop} 
in the case $\frak{h}=X^rD^r$.
\end{note}

Schork reviews the numbers 
of Note~\ref{algXrDr-note}
in~\cite{Schork03} and discusses as well some $q$-analogues
(see also the early work of Katriel and Kibler~\cite{KaKi92},
the papers by M\'endez and Rodriguez ~\cite{MeRo08b,MeRo08}, 
and Section~\ref{qana-subsec} below). The normal form problem for $(X^rD^s)^n$ unites the
case of $(X^rD^r)^n$ discussed above and the case of~$(X^rD)^n$, which is
closely related to trees: we shall accordingly discuss it later, in 
Note~\ref{XrDs-note} of Subsection~\ref{genincr-subsec}, 
p.~\pageref{XrDs-note}.
Relative to these and other cases, we note 
that Mikhailov~\cite{Mikhailov83} obtained the normal orderings of 
\[
(X+D)^n,~~(D^r+X)^n,~~(D+N)^n,~~(D^2+N)^n,
\]
where~$N=XD$; his end-formulae are ``combinatorial'', but his derivations are essentially algebraic.
See also Katriel~\cite{Katriel83} for related material.

\section{\bf Quadratic forms $(X^2+XD+D^2)$, zigzags, and permutations} \label{quad-sec}

The normal forms associated with $(X^2+D^2)$ are 
related to 
evolution equations of the form 
\[
\frac{\partial}{\partial t} \psi(x,t) = \frac{\partial^2}{\partial x^2} \psi(x,t)+
x^2 \psi(x,t)
\]
(in the  case  of the  Schr\"odinger equation,  this  corresponds to  a
quantum harmonic oscillator, i.e., a  quadratic potential)
and are of great importance in quantum optics, where ``squeezing'' 
is introduced by means of quadratic forms 
in annihilation and creation operators.
It is then
not a surprise that the normal form problem  for (powers of) quadratic
forms should have  been studied early, and,  for instance, a  paper by
Mehta~\cite{Mehta77} published in  1977 contains results equivalent to
our  Propositions~\ref{quad0-prop} and~\ref{perm-prop} below; see also
Wilcox~\cite[\S10]{Wilcox67} (published in  1967).  The derivations in
the cited paper  are entirely operator-algebraic.
As  we show  here, such
normal  forms follow routinely from the  representation by diagrams in
conjunction   with standard methods   of combinatorial analysis.  This
approach, which, to the best of our knowledge is new,
furthermore reveals connections with alternating permutations
(Subsection~\ref{zig-subsec})  as well as  with  general  permutations
classified    according       to     local       order        patterns
(Subsection~\ref{perm-subsec}).

\subsection{The circle form $(X^2+D^2)$ and zigzags.} \label{zig-subsec} 

For the  operator $(X^2+D^2)$, the associated gates
have either two outgoing ($X^2$) or two ingoing lines ($D^2$); these are then shaped like a 
``cup'' or a ``cap'': 
\[
\hbox{\Img{6.5}{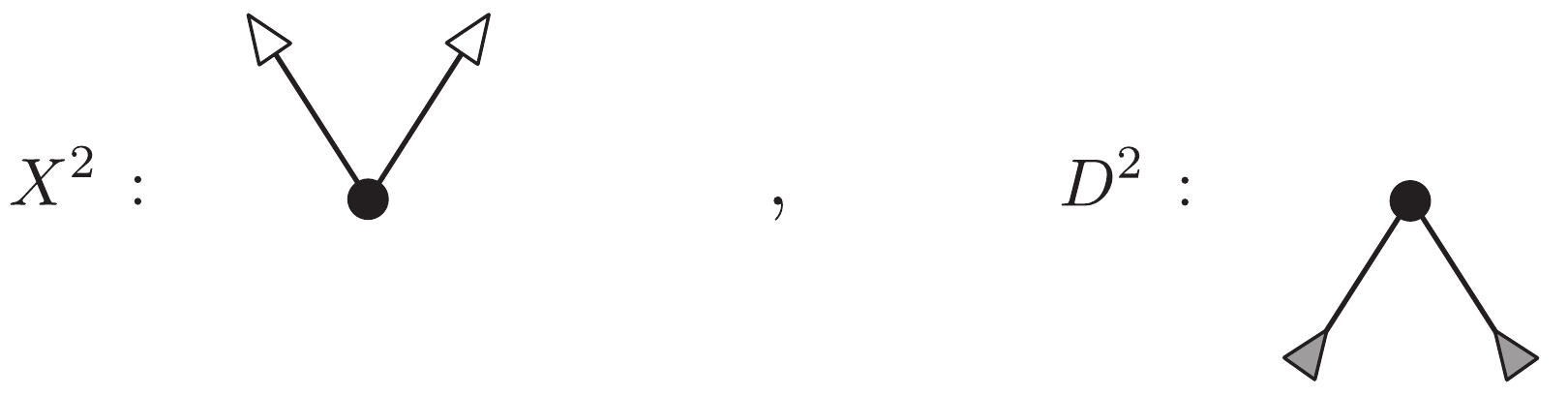}}
\]
A general graph relative to $(X^2+D^2)$
consists  of (weakly) connected components, and each  connected
component   involves    caps    and   cups   in   alternation: we shall 
refer to them as ``zigzags''; see
Figure~\ref{x2+d2-fig}. Note  that since the inputs  and outputs of gates
are distinguishable, edges may cross. It is then easily realized that zigzags
can be of one of four types, depending on the excess of the number of 
free outputs over the number of free inputs---values for this excess can be $-2,0,+2$---
with, in addition, for excess~$0$, the 
fact that diagrams may be either ``closed'' or ``open''. This gives rise to four
different types of connected components, which we denote by $\cal A, \cal B, \cal C, \cal D$. 
In summary,
with $\cal G$ being the class of all graphs relative to $(X^2+D^2)$, 
$u$ marking the number of outputs, and $v$ marking the number of inputs,
the decomposition into connected components is expressed by
\begin{equation}\label{GAlt}
\mathcal{G}=\set\left(u^2\mathcal{A}+v^2\mathcal{B}+uv\mathcal{C}
+\mathcal{D}\right),
\end{equation}

\begin{figure}
\begin{center}
\Img{12.4}{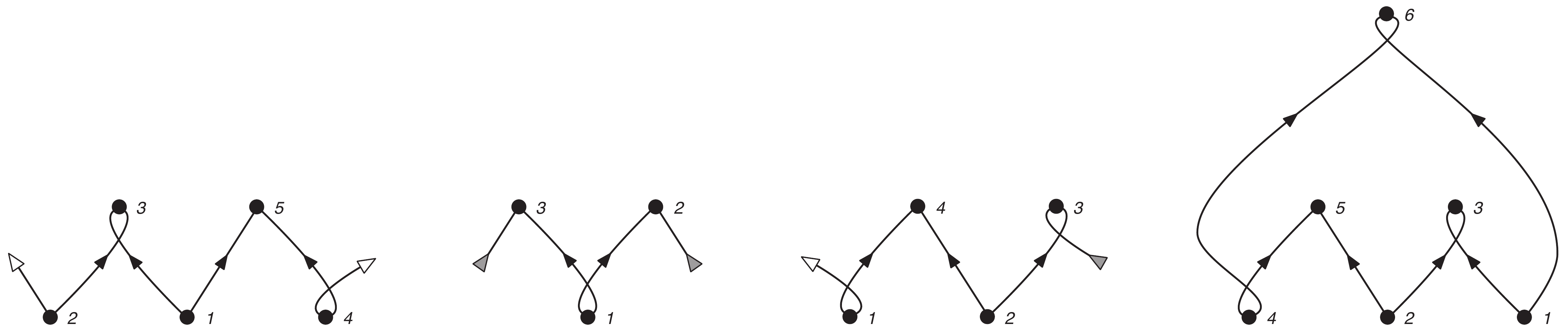}\\
$(\cal A)$ \hspace*{2.5truecm}$(\cal B)$ \hspace*{2truecm}$(\cal C)$ 
\hspace*{2.5truecm}$(\cal D)$ 
\end{center}
\caption{\label{x2+d2-fig}\small
The four classes of connected components (zigzag's) 
$\cal A,\cal B,\cal C,\cal D$ relative to $(X^2+D^2)$.}
\end{figure}

Removal of the vertex with the largest label (in the case of~$\cal A,\cal C, \cal D$)
or with the smallest one (in the case of $\cal B$) shows that a zigzag can be decomposed 
into one or two smaller zigzags.
This is expressed by  the formal specification 
\begin{equation}
\label{zig0} 
\left\{\begin{array}{lll}
\mathcal{A}&=&\mathcal{Z}+ 4\ (\mathcal{A}\star\mathcal{Z}^{\,\blacksquare}\star\mathcal{A}) \\
\mathcal{B}&=&\mathcal{Z}+ 4\ (\mathcal{B}\star\mathcal{Z}^{\,\square}\star\mathcal{B})\\
\mathcal{C}&=&4\ (\mathcal{A}\star\mathcal{Z}^{\,\blacksquare}\star(\mathcal{C}+\mathcal{E}))\\
\mathcal{D}&=&2\ (\mathcal{Z}^{\,\blacksquare}\star\mathcal{A}) \, .
\end{array}\right.
\end{equation}
The multiplicities (2 and 4, as the case may be) reflect the various possibilities of attachments: for
instance, in the case of type~$\cal D$, there are clearly 
two ways of attaching the largest label to a previously existing zigzag of type~$\cal A$.
(The class $\cal E$ is comprised of a unique ``neutral'' element of size~$0$;
the class $\cal Z$ consists of a single atom of size~1.)

The decomposition into connected components~\eqref{GAlt}
gives
\begin{equation}\label{zig05}
G(z;u,v)=\exp\left(u^2A(z)+v^2B(z)+uv\,C(z)+D(z)\right)
,
\end{equation}
while the combinatorial specification~\eqref{zig0} give rise to integral equations:
\begin{equation}\label{zig1} 
\left\{\begin{array}{lll}
\ds A(z)=z+ 4\int_0^z A(t)^2\,dt, &&
\ds B(z)=z+ 4\int_0^z B(t)^2\,dt, \\
\ds C(z)=4\int_0^z A(t)\cdot (C(t)+1)\,dt, &&
\ds D(z)=2\int_0^z A(t)\,dt\, .
\end{array}\right.
\end{equation}
The equivalent differential equations,
\[
\left\{\begin{array}{llll}
\partial_z\,A(z)=1+4\,A(z)^2\ , && \partial_z\,B(z)=1+4\,B(z)^2\ ,\\
\partial_z\,C(z)=4\,A(z)\cdot(C(z)+1)\ , && \partial_z\,D(z)=2\,A(z)\ ,
\end{array}\right.
\]
(with the initial conditions $A(0)=B(0)=C(0)=D(0)=0$) 
admit separation of variables,
hence they have closed-form solutions:
\begin{equation}\label{zig2}
\left\{\begin{array}{lll}
\ds A(z)=\tfrac{1}{2}\,\tan\,(2z), && \ds B(z)= \tfrac{1}{2}\,\tan\,(2z), \\ 
\ds C(z)=\cos\,(2z)^{-1}-1, 
&& \ds D(z)=-\tfrac{1}{2}\,\ln\,(\cos\,(2z)) \, . 
\end{array}\right.
\end{equation}
(Note that equality of generating functions $A(z)=B(z)$ mirrors the 
isomorphism between classes $\mathcal{A}$ and $\mathcal{B}$, which can
 be seen explicitly by a horizontal reflection of zigzags and 
a relabelling of the vertices according to $1,\ldots,m \,\mapsto\, m,\ldots,1$.)

Collecting the results of~\eqref{zig05} and~\eqref{zig2}, we obtain:

\begin{proposition} \label{quad0-prop}
The generating function of all graphs associated 
with $(X^2+D^2)$ admits the explicit form
\begin{equation}
G(z;u,v)=\frac{1}{\sqrt{\cos(2z)} } 
\cdot\exp\left( \tfrac{1}{2}(u^2+v^2)\tan(2z)
+ uv (\sec(2z)-1)\right).
\end{equation}
\end{proposition}

\begin{note} \emph{Zigzags and alternating permutations.} The enumeration of zigzags is closely related to one of the most
classical problems of combinatorial analysis; namely, the enumeration
of alternating permutations; see~\cite[pp.~258--259]{Comtet74}
and \cite[pp.~73--75]{Stanley99}. (We follow the presentation in~\cite[\S II.6.3]{FlSe09}.)
Let a permutation $\sigma=\sigma_1\cdots
\sigma_n$ be written as a word (so that $\sigma_i$ represents
$\sigma(i)$). The permutation~$\sigma$ is said to be
\emph{alternating}
if $\sigma_1<\sigma_2$, $\sigma_2>\sigma_3$, and
so on, with an alternating pattern of rises
($\sigma_{2j-1}<\sigma_{2j}$) and falls ($\sigma_{2j}<\sigma_{2j+1}$).
Let $\cal S$ and $\cal T$ be, respectively, the EGF of even-sized and
odd-sized permutations.
By a maximum based decomposition analogous to what we have seen
before, the specifications are
\[
\cal T = (\cal T\star \cal Z^\blacksquare\star \cal T);
\qquad
\cal S=\{\epsilon\}+ (\cal T\star \cal Z^\blacksquare\star \cal S).
\]
Hence, via integral and differential relations, the well-known
solutions\footnote{We note, following the referee's remark, that the link between calculation of vacuum expectation expectation values of powers of the quadratic form $X^2+D^2$ with alternating permutations and secant numbers has been also observed by C.V. Sukumar and A. Hodges in~\cite{HoSu07,SuHo07}.}:
\[
\begin{array}{lllllll}
S(z)&=&\sec(z)&=&\ds 1+1\,\frac{{z}^{2}}{2!}+5\,\frac{{z}^{4}}{4!}+61\,\frac{{z}^{6}}{6!}+\cdots
&&\hbox{(\OEIS{A000364})}\\[2truemm]
T(z)&=&\tan(z)&=&\ds z+2\,\frac{{z}^{3}}{3!}+16\,\frac{{z}^{5}}{5!}+272\,\frac{{z}^{7}}{7!}+\cdots
&&\hbox{(\OEIS{A000182})}.
\end{array}
\]

Also of interest are  \emph{alternating cycles} defined to be directed
cycles whose edges have alternatively increasing and decreasing 
end points. These are necessarily of even size with specification
$\cal U=(\cal Z^\blacksquare\star \cal T)$, so that $U(z)=\int T(z)$
and $
U(z)=\log \sec(z)$ whose coefficients are a shifted version of the
tangent numbers $n!\, [z^n]\tan(z)$. The undirected cycle version has
EGF $V(z)=\frac12\log\sec(z)+\frac14z^2$, so that (undirected) graphs whose 
components are (undirected) alternating cycles have EGF
\[
W(z)=\frac{e^{z^2/4}}{\sqrt{\cos(z)}}=
1+1\,
\frac{{x}^{2}}{2!}+4\,\frac{{x}^{4}}{4!}+38\,\frac{{x}^{6}}{6!}+710\,\frac{{x}^{8}}{8!}+
\cdots,
\]
which is not found in the \emph{OEIS}.
For reasons discussed below, in Section~\ref{cf-sec}, several of these numbers
have OGFs that admit explicit \emph{continued fraction}
representations.
\end{note}

\subsection{The general quadratic form  $(\alpha\, D^2+\beta\, X^2+\gamma\, XD)$.}\label{perm-subsec} 

We  treat here diagrams corresponding to the general quadratic form
\[
\alpha\, D^2+\beta\, X^2+\gamma\, XD,\] 
which introduce extra gates having one  ingoing and one outgoing line  (compare
with Subsection~\ref{zig-subsec}):
\[
\hbox{\Img{10.6}{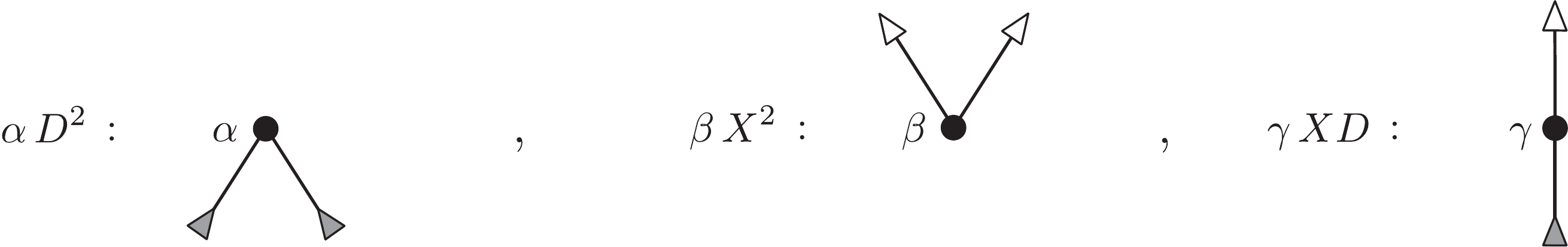}}
\]
Hence, the  graphs again have the structure  of
open  and  closed zig-zags,   but  now with ascents  and descents of
arbitrary length,   see Fig.~\ref{D2+X2+XD-fig}. Additionally,  parameters
$\alpha$,  $\beta$ and  $\gamma$   keep track, respectively, of   the
statistics   of   peaks, valleys     and    rises/falls  in     the
construction. The analysis of this case closely follows  the  scheme given in
Subsection~\ref{zig-subsec}. We state:

\begin{figure}
\begin{center}
\setlength{\unitlength}{1truecm}
\begin{picture}(11,9.7)
\put(0,5.5){\hbox{\Img{11}{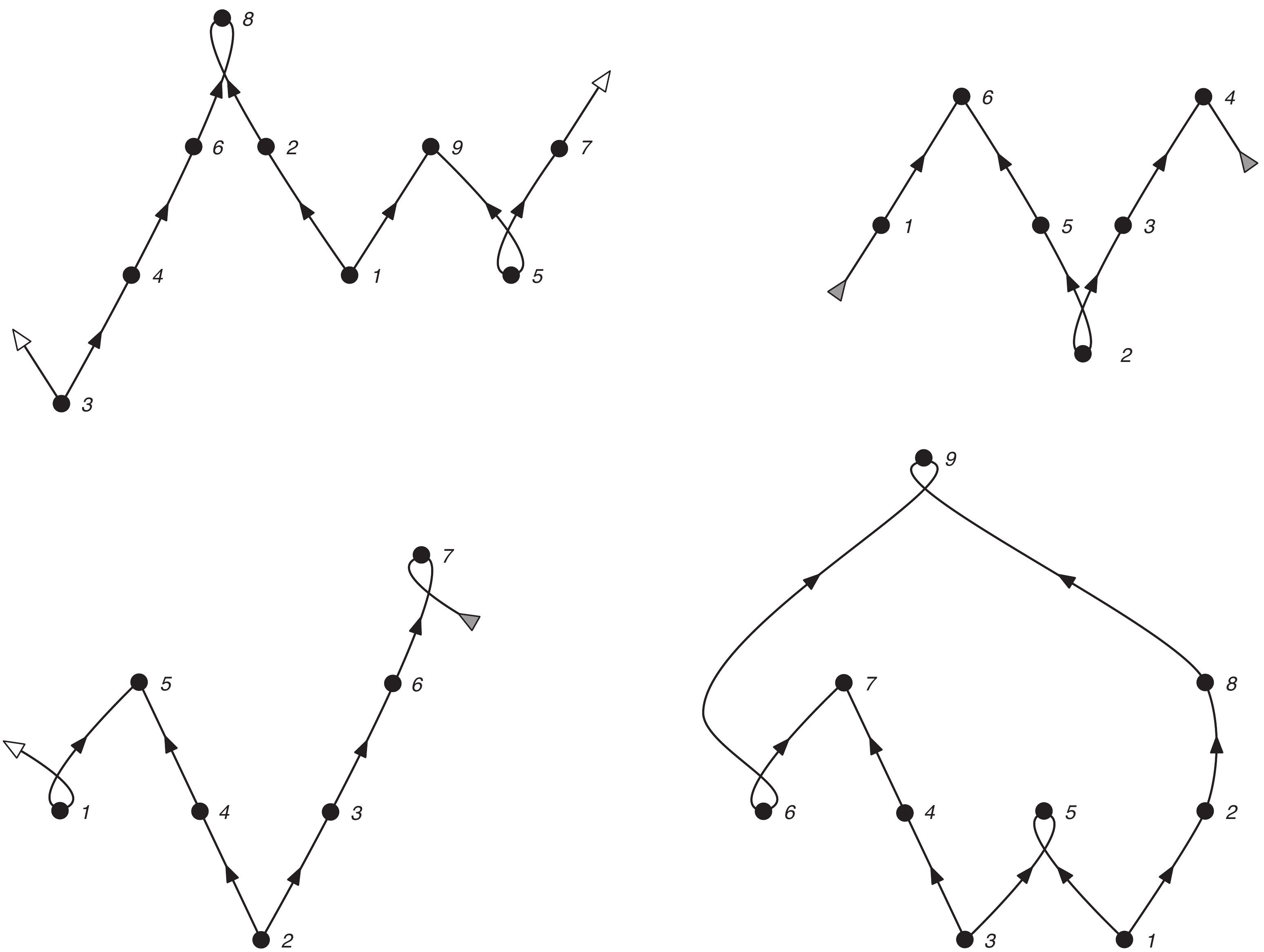}}}
\put(3.5,6.5){$(\cal A)$}
\put(8.0,6.5){$(\cal B)$}
\put(3.5,1.0){$(\cal C)$}
\put(8.0,1.0){$(\cal D)$}
\end{picture}
\end{center}

\vspace*{-1.2truecm}

\caption{\label{D2+X2+XD-fig}\small
The four classes of connected components $\cal A,\cal B,\cal C,\cal D$ associated with
$(\alpha\,D^2+\beta\,X^2+\gamma\,XD)$.
}
\end{figure}
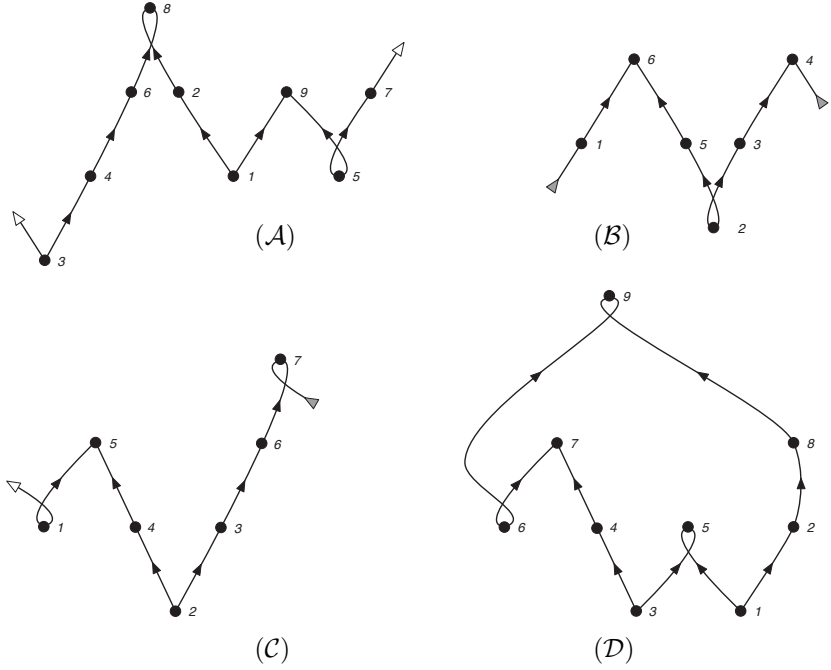

\begin{proposition}\label{perm-prop}
The generating function of all graphs associated with $(\alpha D^2+\beta X^2+\gamma XD)$
admits the explicit form
\begin{eqnarray}
G(z;u,v)&=&\exp\left(\left(\frac{u^2}{4\alpha}+\frac{v^2}{4\beta}\right)\left(\delta\,\tan\,(\delta\,z+\theta)-\gamma\right)\right)\nonumber\\
&&\qquad {} \cdot \exp\left(uv\left(\frac{\cos\,(\theta)}{\cos\,(\delta \,z+\theta)}-1\right)\right)\cdot e^{\tfrac{\gamma}{2}z}\,\sqrt{\frac{\cos\,(\theta)}{\cos\,(\delta \,z+\theta)}},
\end{eqnarray}
where 
\begin{equation}\label{abbrevs}
\delta=\sqrt{4\alpha\beta-\gamma^2}, \qquad
\theta=\arctan(\gamma/\delta).
\end{equation}
\end{proposition}  
\begin{proof}
The decomposition into connected components gives, 
for the class of all diagrams:
\begin{equation}\label{SpecD2+X2+XD}
\mathcal{G}=\textsc{Set}\,\left(u^2\,\mathcal{A}+v^2\,\mathcal{B}+uv\ \mathcal{C}+\mathcal{D}\right).
\end{equation}
The specification of the four types of components paralells that of~\eqref{zig0}:
\[\left\{\begin{array}{lll}
\mathcal{A}&=&\ds \beta\ \mathcal{Z}+ 2\gamma\ \mathcal{Z}^{\,\blacksquare}\star\mathcal{A}+ 4\alpha\ \mathcal{A}\star\mathcal{Z}^{\,\blacksquare}\star\mathcal{A}\, ,\\
\mathcal{B}&=&\ds\alpha\ \mathcal{Z}+ 2\gamma\ \mathcal{Z}^{\,\square}\star\mathcal{B}+ 4\beta\ \mathcal{B}\star\mathcal{Z}^{\,\square}\star\mathcal{B}\, ,\\
\mathcal{C}&=&\ds\gamma\ \mathcal{Z}+ \gamma\ \mathcal{Z}^{\,\blacksquare}\star\mathcal{C}+ 4\alpha\ \mathcal{A}\star\mathcal{Z}^{\,\blacksquare}\star(\mathcal{C}+\mathcal{E})\, ,
\\
\mathcal{D}&=&\ds 2\alpha\ \mathcal{Z}^{\,\blacksquare}\star\mathcal{A}\,.
\end{array}\right.
\]
The translation into equations binding EGFs is now automatic. First, we have
\[
G(z;u,v)=\exp\left(u^2\,A(z)+v^2\,B(z)+uv\ C(z)+D(z)\right).
\]
Next,
\[\left\{\begin{array}{lll}
A(z)&=&\ds \beta\,z+ 2\,\gamma\int_0^z A(t)\,dt+ 4\,\alpha\int_0^z A(t)^2\,dt ,\\
B(z)&=&\ds \alpha\,z+ 2\,\gamma\int_0^z A(t)\,dt+ 4\,\beta\int_0^z B(t)^2\,dt ,\\
C(z)&=&\ds \gamma\,z+ \gamma\int_0^zC(t)\,dt+ 4\,\alpha\int_0^z A(t)\cdot(C(t)+1)\,dt ,\\
D(z)&=&\ds 2\,\alpha\int_0^z A(t)\,dt\, .
\end{array}\right.
\end{equation*}

The corresponding differential equations are
\[\left\{\begin{array}{lll}
\partial_z\,A(z)&=&\ds \beta+2\gamma\,A(z)+4\alpha\,A(z)^2\ ,\\
\partial_z\,B(z)&=&\ds \alpha+2\gamma\,B(z)+4\beta\,B(z)^2\ ,\\
\partial_z\,C(z)&=&\ds (\gamma+4\alpha\,A(z))\cdot(C(z)+1)\ ,\\
\partial_z\,D(z)&=&\ds 2\alpha\,A(z)\, ,
\end{array}\right.
\]
with initial conditions $A(0)=B(0)=C(0)=D(0)=0$. There is again separation of variables. However,
as exemplified by the case of~$\cal A$, one 
now needs to solve (with $y=A(z)$)
\[
\frac{dy}{\beta+2\gamma y+4\alpha y^2}=dz,
\]
which neccessitates factoring the quadratic form in
the denominator, hence
solving a quadratic equation: this introduces  the auxiliary quantities~\eqref{abbrevs}.
The solutions found are
\[
\left\{\begin{array}{lllll}
A(z)&=& \ds \frac{\delta}{4\alpha}\,\frac{\gamma+\delta\,\tan\,(\delta\,z)}{\delta-\gamma\,\tan\,(\delta\,z)}-\frac{\gamma}{4\alpha} 
&=&\ds \frac{\delta}{4\alpha}\,\tan\,(\delta\,z+\theta)-\frac{\gamma}{4\alpha}\ ,
\\
B(z)&=&\ds \left. A(z)\right|_{\alpha\leftrightarrow\beta}\ ,
\\
C(z)&=&\ds \cos\,(\delta\,z)-\tfrac{\gamma}{\delta}\sin\,(\delta\,z)-1
&=&\ds \frac{\cos\,(\theta)}{\cos\,(\delta\,z+\theta)}-1\ ,
\\
D(z)&=&\ds -\tfrac{1}{2}\ln\left(\cos\,(\delta\,z)-\tfrac{\gamma}{\delta}\sin\,(\delta\,z)\right)-\tfrac{\gamma}{2}\,z
&=&
\ds \tfrac{1}{2}\ln\left(\frac{\cos\,(\theta)}{\cos\,(\delta\,z+\theta)}\right)-\tfrac{\gamma}{2}\,z\ .\label{DPerm}
\end{array}\right.
\]
from which the statement results.
\end{proof}

\begin{note} \emph{Generalized Eulerian numbers.}
In the same way that zigzags are closely related to alternating permutations, the diagrams considered here
are closely related to a quadrivariate statistics on permutations; namely that of the number of 
peaks, valleys, double rises, and double falls~\cite[Ex.~3.3.46, p.~195]{GoJa83}.
For instance, Carlitz and Scoville~\cite{CaSc74} found the corresponding multivariate EGF,
\[
\frac{e^{\alpha_2 z}-e^{\alpha_1 z}}{\alpha_2e^{\alpha_1 z}-\alpha_1e^{\alpha_2 z}},
\qquad
\alpha_1\alpha_2=u_1u_2, \quad
\alpha_1+\alpha_2=u_3+u_4
\]
($u_1,u_2,u_3,u_4$ mark, respectively, the four types of elements listed above),
which suitably generalizes the bivariate EGF of Eulerian numbers
(\OEIS{A008292} and~\cite[p.~244]{Comtet74}),
\begin{equation}\label{eulerian}
A(z,u)=\frac{1-u}{1-ue^{z(1-u)}},
\end{equation}
which enumerates permutations according to the number of rises (do $u_1=u_3u$, $u_2=u_4=1$).
See~\cite[p.~202]{FlSe09}
for a derivation along the lines above. The corresponding OGFs also have explicit continued fraction
expansions~\cite{Flajolet80b}, which are closely 
related to the combinatorial approach of 
the present study---see Subsection~\ref{cf-subsec}.
\end{note}

\section{\bf Semilinear forms $(\phi(X)D)$ and increasing trees} \label{semilin-sec}

We consider here first-order differential operators of the general
form $\phi(X)D+\rho(X)$, where~$\phi$ is a (nonlinear) polynomial of
arbitrary degree. The general combinatorial model is that of
increasing trees, which were discussed early by Leroux and
Viennot~\cite{LeVi86,LeVi88b} and 
further developed by Bergeron, Labelle, and Leroux in their reference text~\cite[Ch.~5]{BeLaLe98}.
(The subject is treated under the combinatorial--analytic angle by Bergeron
\emph{et al.}~\cite{BeFlSa92}.) In the physics literature, 
Lang~\cite{Lang00} seems to have been among the first to discuss the
normal ordering associated with $X^rD$;
see also his later paper~\cite{Lang09}.
Then Blasiak \emph{et al.}~\cite{BlDaHoPe06,BlHoPeDuSo05}
discovered the relationship between general abstract
varieties of increasing trees, as in~\cite{BeFlSa92},
 and the forms $(\phi(X)D)$,
which we propose to examine now. 
It is pleasant to note that the developments of this section also serve to
answer several algebraic--combinatorial 
questions first raised by Scherk in 1823;
see the discussion relative to Figure~\ref{scherktree-fig},
p.~\pageref{scherktree-fig}, of the Appendix.


\begin{figure}\small

\begin{center}
\begin{tabular}{cc}
\Img{1}{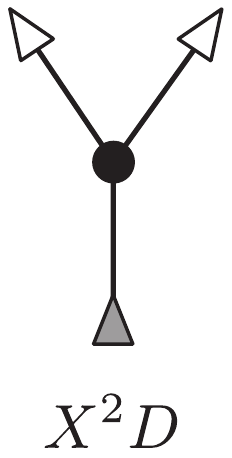}\qquad\qquad
& \fbox{\Img{7.5}{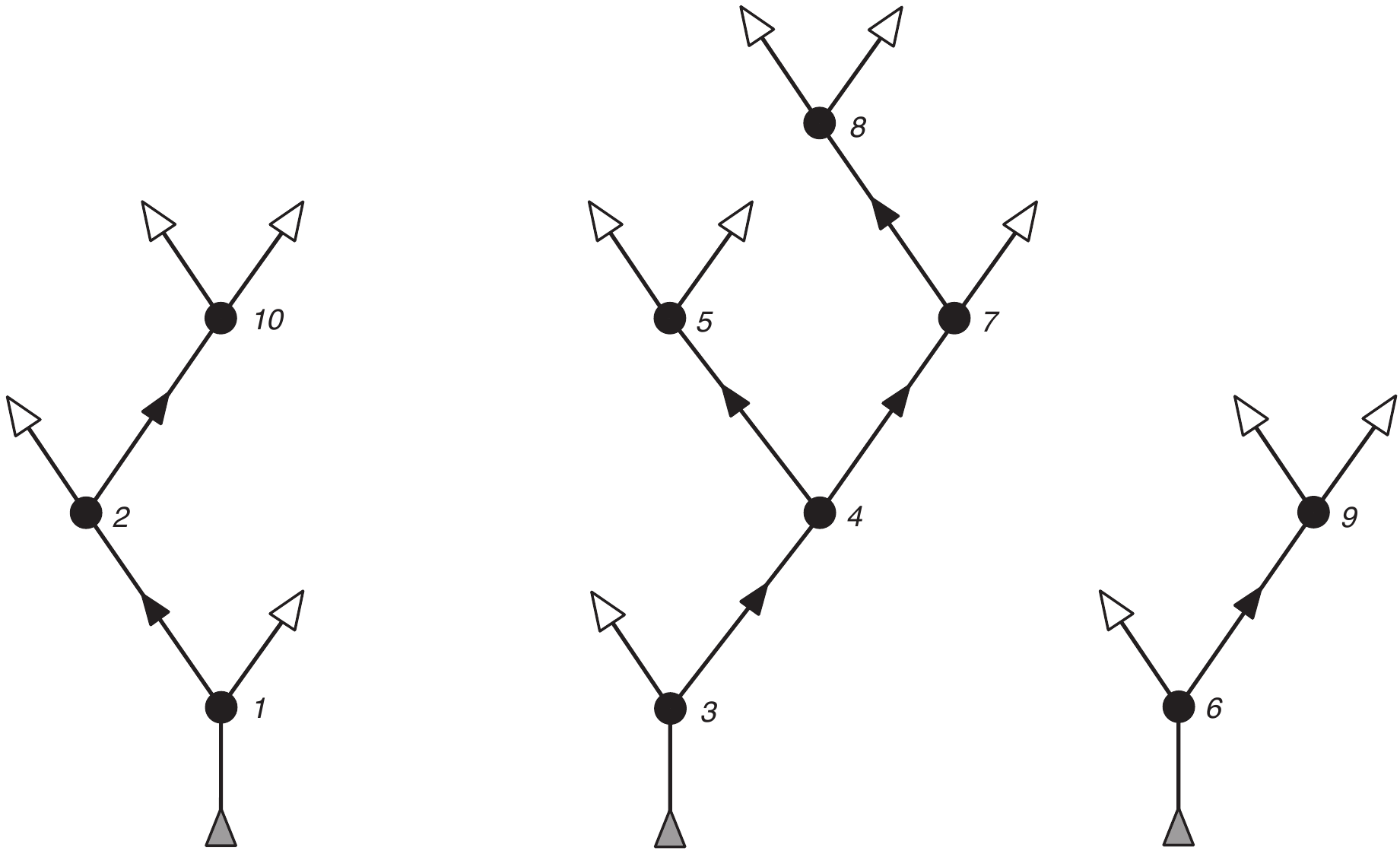}}
\end{tabular}
\end{center}

\caption{\label{x2da-fig}\small
The basic gate $\Y$ associated with $X^2D$ (left) and a corresponding
diagram (right box).
}
\end{figure}

\subsection{The form $(X^2D)$ and increasing trees.} \label{incr-subsec}

In the case of the operator~$X^2D$, the gates have one input  and two outputs.
They are thus shaped as a letter~``$\Y$''; see Figure~\ref{x2da-fig}. 
By successively grafting outputs of such $\Y$'s to inputs of other~$\Y$'s,
we obtain graphs that are necessarily acyclic, so that each connected component
is a \emph{tree}. These trees are \emph{plane trees} (there is a distinction between  left and 
right outgoing edges); they are \emph{binary}, by design; and, finally, they are
increasing in the sense that labels increase along any branch stemming from the root.
Such trees are quite well known in combinatorial theory under the name 
of \emph{binary increasing trees}, their importance being due
to the fact that they
are in bijective correspondence with permutations.
See for instance~\cite[\S II.6.3]{FlSe09} or \cite{Stanley86}.

We commence  with the class~$\cal T$ of binary increasing trees. A tree in~$\cal T$ is obtained by
starting with a root that bears the smallest label; then, each of the two output edges
is either left untouched or it is attached recursively to a similar tree.
Within the language of specifications summarized in Section~\ref{maindef-sec},
this is expressed as follows (``$\cal E$'' represents a neutral
structure of size~$0$):
\[
{\cal T}=\zz^{\Box}\star ({\cal E}+{\cal T}) \star ({\cal E} +{\cal T})\ .
\]
The translation to EGFs is then
\[
T(z)=\int_0^z \left(1+T(w)\right)^2\, dw,
\]
which, by differentiation, leads to the equation
\[
T'(z)=(1+T(z))^2,\qquad T(0)=0.\]
Since this last differential equation admits separation of variables,
we find $T'/(1+T)^2=1$, hence $-1/(1+T)=z-1$; that is,
\[
T(z)=\frac{z} {1-z} \equiv \sum_{n=1}^\infty n! \frac{z^n}{n!}.
\]
This verifies that  $T_n=n!$: binary increasing trees are equinumerous
with  permutations. (Combinatorially, here  are two easy combinatorial
arguments:  $(i)$~the number of choices  for successively adding a new
gate is $1,2,3,4,5,\ldots$; $(ii)$~a permutation is obtained when node
labels  are  read in infix,  i.e., left-to-right,   order.  The formal
relation with \emph{binary}  trees, is  well explained combinatorially
by Leroux--Viennot~\cite{LeVi86,LeVi88b}: see also~\cite[Ex.~12, p.~383]{BeLaLe98}.)

Now, a connected diagram has one input (the link into the root) and $(n+1)$
free outputs, if~$n$ is the number of nodes ($\Y$--gates associated with~$(X^2D)$) of the tree.
Thus, the trivariate EGF, where~$u$ marks outputs and~$v$ marks inputs is
\[
v u T(uz)=\frac{vu^2z}{1-uz}.
\]
The EGF of all diagrams is accordingly
\[
G(z;u,v)=\exp\left(\frac{vu^2z}{1-uz}\right).
\]
In particular, the total number of diagrams of size~$n\ge1$ is
\[
n![z^n]G(z;1,1)=n![z^n]e^{z/(1-z)}
=\sum_{k=1}^{n-1} \binom{n-1}{k-1}\frac{n!}{k!}.
\]
The sequence starts as $1, 1, 3, 13, 73, 501, 4051, 37633$,
and  is \OEIS{A000262} (``sets of lists''); see
also~\cite[p.~125]{FlSe09} (``fragmented permutations'').
We are indeed enumerating unordered forests of increasing binary trees,
equivalently, sets of nonempty permutations, 
i.e., structures of type $\set(\seq_{\ge1}(\cal Z))$.

This approach applies almost verbatim to 
the operator $X^{r}D$, with $r\ge2$. 
In this case, we are dealing with $r$--ary trees, where each node
has~$r$ outgoing edges:
\[
\hbox{\Img{1.8}{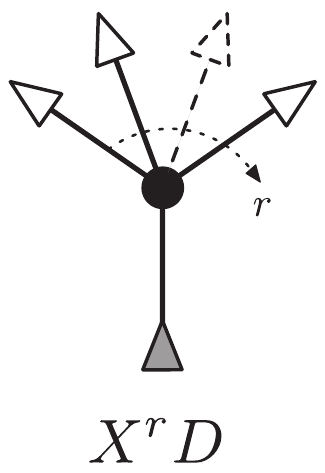}}
\]
The differential
equation becomes $T'=(1+T)^r$, with solution
\begin{equation}\label{odetreer}
T=\left[1-(r-1)z\right]^{-1/(r-1)}-1, \quad\hbox{and}\quad
n![z^n]T(z)=n!\binom{n+1/(r-1)-1}{n}.
\end{equation}
The balance between the number of outputs and size is dealt with as before (with an~$r$-ary tree of~$n$ internal nodes
having $n(r-1)+1$ outgoing edges). So:\footnote{
Interestingly enough, Comtet has a form equivalent to the one we give,
as an exercise in his book~\cite[Ex.~2, p.~220]{Comtet74} relative to
``Lie derivatives and operational calculus''.}  

\begin{proposition}
The EGF of diagrams associated with $(X^rD)$ is
\[
G(z;u,v)=\exp\left(\frac{uv}{\left[1-(r-1)u^{r-1}z\right]^{1/(r-1)}}-uv\right).
\]
For $r=2$, the normal forms associated with $\frak{h}=(X^2D)$ simplify to\footnote{
The coefficients, up to sign, are sometimes known as ``Lah numbers''
(\OEIS{A008297} and ~\cite[p.~44]{Riordan80}), 
the Lah polynomials being essentially Laguere polynomials.}
\[
\nor(\frak{h}^n)=\sum_{\ell=1}^{n-1} \binom{n-1}{k-1}\frac{n!}{k!}X^{n+k}D^k.
\]
For $r\ge3$, one has $\ds
\nor(\frak{h}^n)=\sum_{k} \gamma_{n,k}X^{k+(r-1)n}D^k$, 
with 
\[
\gamma_{n,k}=\frac{n!}{k!}\sum_{\ell=0}^k (-1)^{k-\ell}\binom{k}{\ell}
\binom{n+\ell/(r-1)-1}{k}(r-1)^n.
\]
\end{proposition}

Combinatorially, the number of
possibilities for successively adding a new gate
is now~$1,r,(2r-1),\ldots$, which is consistent with~\eqref{odetreer}.  
The derivation given above, via EGFs,
has however been adopted since it is the only one  
applicable to the general
case~$(\phi(X)D)$, as we see next.
The total number of graphs for $r=3,4$ starts as
\[
\begin{array}{rll}
r=3 : & 	1, 4, 25, 211, 2236, 28471, 422899, 7173580,~\ldots &
\hbox{\OEIS{A049118}}\\
r=4 : & 	1, 5, 41, 465, 6721, 117941, 2433145, 57673281, ~\ldots &
\hbox{\OEIS{A049119}},
\end{array}
\]
and these coefficients are expressible as binomial sums;
see Lang's
study~\cite{Lang00} for an early discussion of these numbers
revisited
in~\cite{Lang09}
and
the papers~\cite{BlDaHoPe06,BlHoPeDuSo05} by Blasiak \emph{et al.}
for relations with the normal ordering problem.

\subsection{The general case $(\phi(X)D)$.} \label{genincr-subsec}
In this case, we must consider a polynomial
\[
\phi(y)=\sum_{j=0}^R \phi_j y^r,
\]
with gates of the form $X^jD$, any such gate being assigned weight $\phi_r$.
The corresponding increasing trees can thus have various types of
internal nodes, as dictated by $\phi(y)$.
Graphically:
\[
\hbox{\Img{12.5}{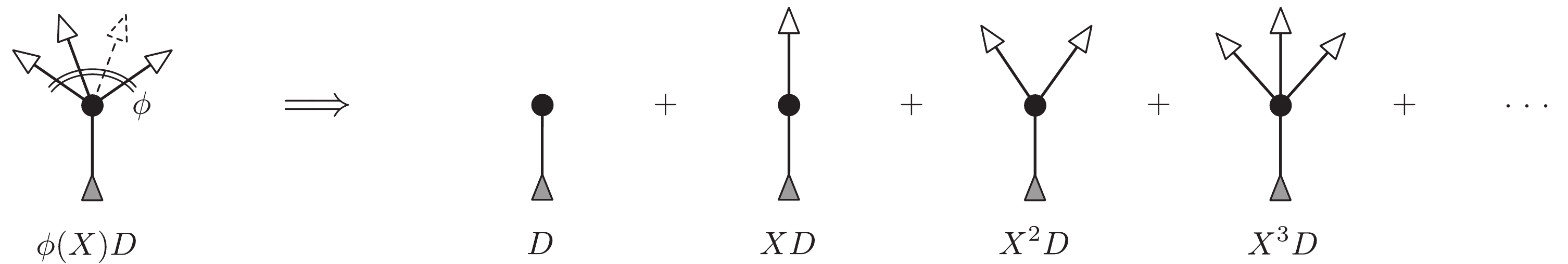}}
\]
In combinatorial terms, this defines a \emph{variety of increasing trees}: the basic and asymptotic theory
of these being the subject of the paper~\cite{BeFlSa92}. 

\begin{figure}
\begin{center}
\Img{9}{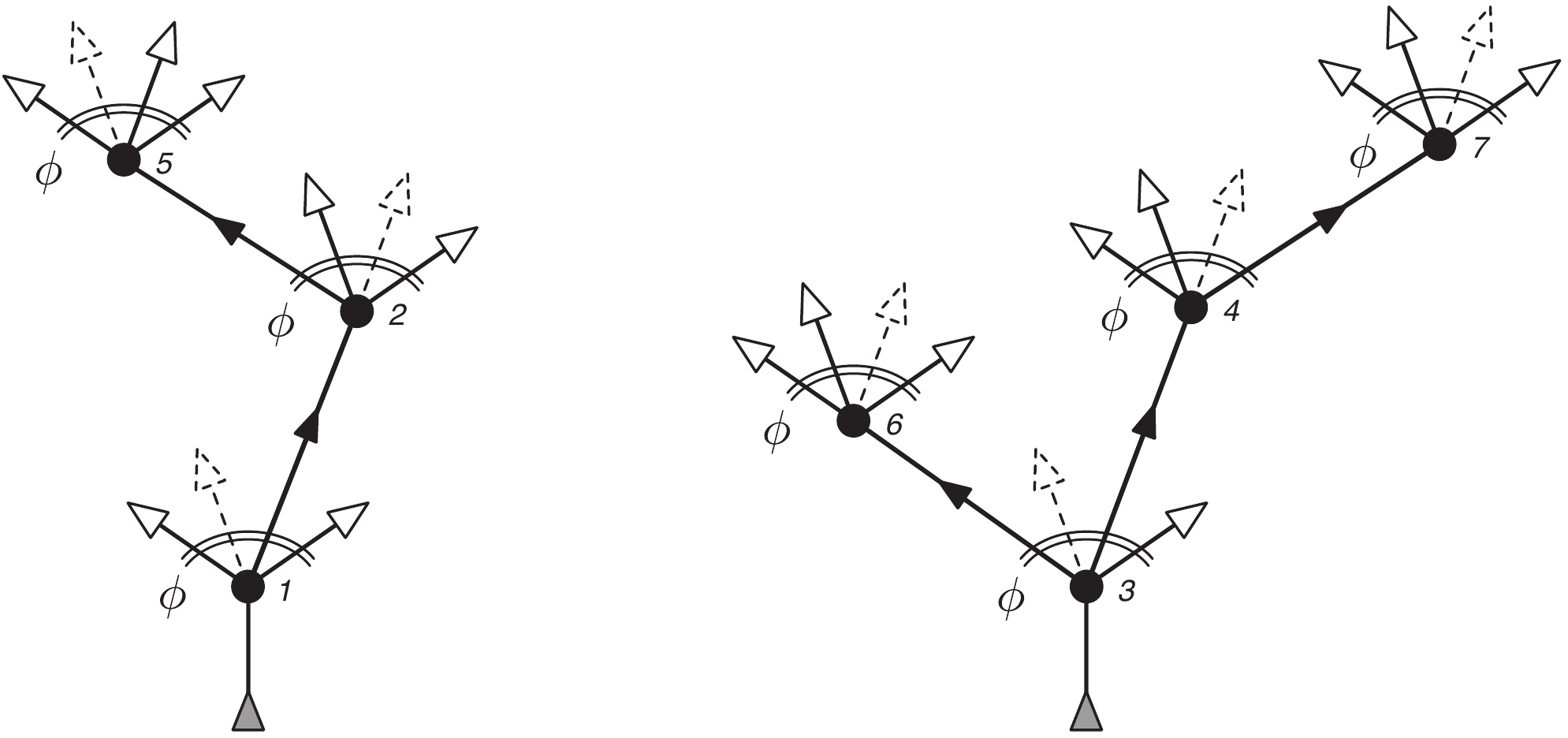}
\end{center}
\caption{\label{TreeVariety-fig}\small
A diagram with some $\phi(X)D$ that is comprised of two rooted trees.}
\end{figure}

The graphs associated with $\phi(X)D$ are still (unordered) forests of
increasing trees, themselves denoted  by $\cal T$, or ${\cal T}^\phi$,
whenever the dependency on~$\phi$ needs to be made explicit. The basic
equation for the EGF $T=T(z,u)$, with~$u$ a marker for ``unsaturated''
outgoing edges is then (Figure~\ref{TreeVariety-fig})
\begin{equation}\label{tdef}
{\cal T}=\sum_{j=0}^R \phi_r \bigg({\cal Z}^{\Box} \star \overbrace{(u+{\cal T})\star \cdots \star(u+ {\cal T})}^{\hbox{$r$ 
times}}\bigg)
\end{equation}
Thus, we have
\[
T(z,u) = \int_0^z \phi(u+T(w,u))\, dw,
\]
implying
\[
\frac{\partial}{\partial z}T(z,u)=\phi(u+T(z,u)),\qquad T(0,u)=0,
\]
and separation of variables yields, with~$u$ a parameter:
\[
\int_0^{T(z,u)} \frac{d\tau}{\phi(u+\tau)}=z.
\]
We can now conclude as follows.

\begin{proposition} \label{berg-prop}
Define 
\begin{equation}\label{Phidef}
\Phi(y):=\int_{y_0}^y \frac{d\tau}{\phi(\tau)},
\end{equation}
where~$y_0$ is an arbitrary nonnegative constant chosen so that~$\phi(y_0)\not=0$.
Let $T\equiv T^\phi (z,u)$ be defined as the solution of 
\begin{equation}\label{Tdef}
T~:\qquad \Phi(T+u)-\Phi(u)=z,
\quad \hbox{with $T(0,u)=0$}.
\end{equation}
Then the multivariate EGF of all diagrams relative to $(\phi(X)D)$ is
\[
G(z;u,v)=\exp\left(vT^\phi(z,u)\right).\]
\end{proposition}

Proposition~\ref{berg-prop} provides a solution to 
the general problem posed by Scherk in his doctoral
dissertation~\cite[\S8]{Scherk23} of 1823: see our Appendix~\ref{scherk-ap},
p.~\pageref{scherktree-fig}.
The remarks relative to Scherk's thesis,
 p.~\pageref{scherk8-pg} below, provide
a concrete illustration of the combinatorics of increasing trees in relation
to the normal form of $(\phi(X)D)^n$.

On another register, it is shown in~\cite{BeFlSa92} that the singular structure of $T(z,1)$ can be systematically determined,
from which there follow many \emph{asymptotic}
distributional analyses of parameters, including path length, 
node-degree profile, root degree, and so on. However, as regards explicit forms that are of concern for us here,
only a few functions~$\phi(y)$ are susceptible to exact expressions;
for instance,
\begin{equation}\label{exptan}
y^r, \quad
(1+y)^r, \qquad 
(ay^2+by+c),\quad
y(y+1)\cdots(y+r-1).
\end{equation}

\begin{note}
\emph{Some solvable varieties~\cite[\S2]{BeFlSa92}.}
First, $(X^2D+D)$, corresponding to $\phi(y)=y^2+1$, leads to a
solvable case:
\[
G(z;u,v)=\exp\left(\frac{v(1+u^2)\tan z}{1-u\tan z}\right),
\qquad
G(z;1,1)=\exp\left(\frac{2\tan(z)}{1-\tan z}\right).
\]
The coefficients $n![z^n]G(z;1,1)$ are of the form $2^n a_n$, where \[
(a_n)=1, 1, 2, 6, 23, 107, 583, 3633, 25444, 197620,\ldots
\] is \OEIS{A000772} [``the number of elevated(!) increasing binary
trees'']. More generally, any quadratic polynomial leads to a solvable
model.

On another register, if~$\phi(y)$ has only \emph{rational roots}, then
its partial fraction expansion only involves rational coefficients.
The inversion problem is then of the type
\[
\sum_j r_j \log(1-\alpha_jT^j)=z, \qquad r_j,\alpha_j\in\Q.\]
Thus, by inversion, \emph{$T$ is an algebraic function
  of~$e^z$}. This is in particular the case for Stirling polynomials,
such as $\phi(y)=(y+1)(y+2)(y+3)$, for which
\[
T(z,1)=-2+\frac{2}{\sqrt{4-3e^{2z}}}=6z+66\frac{z^2}{2!}+1158\frac{z^3}{3!}+28290\frac{z^4}{4!}+\cdots\,.
\]
(This last example is from~\cite[p.~30]{BeFlSa92}.)

Other interesting cases for combinatorics are $\phi(y)=e^y$ and $\phi(y)=(1-y)^{-1}$,
leading, respectively, to increasing Cayley trees (so-called ``recursive'' trees) and
increasing
Catalan trees (``plane ordered recursive trees'', also known as ``PORTs''); see~\cite[\S1]{BeFlSa92}.
The corresponding EGFs, $T(z)=T(z,1)$ are
\[
\left\{\begin{array}{lllllll}
T(z)&=&\ds \log\frac{1}{1-z}&=&\ds \sum_{n\ge 1} (n-1)! \frac{z^n}{n!}
& \big(\phi(y)=e^y\big)\\
T(z)&=&\ds 1-\sqrt{1-2z}&=&\ds \sum_{n\ge1}\big(1\cdot3\cdots(2n-3)\big)\frac{z^n}{n!}
&\big(\phi(y)=(1-y)^{-1}\big).
\end{array}
\right.
\]
We will not discuss them further  as they are out of
our scope, since they concern non-polynomial forms (see however
Scherk's result relative to the normal
ordering of $(e^{X}D)$, Proposition~\ref{rectree-aprop}, in the Appendix).. 
\end{note}

%
%

\begin{note}  \emph{Algebraic reduction of $X^rD^s$ after~\cite{Blasiak05,BlPeSo03b}.} \label{XrDs-note}
 One may apply the same procedure 
as in Note~\ref{algXrDr-note}, p.~\pageref{algXrDr-note}, 
to obtain coefficients of the normal form of $(X^rD^s)^n$ written as
\[
(X^rD^s)^n=X^{n(r-s)}\sum_k{n \brace k}_{r,s}\,X^rD^s
\]
We first assume that $r\geq s$.
The action of $X^rD^s$ on the exponential $e^x$ is
\begin{equation}\label{balrs20}
\left(X^rD^s\right)^n e^x =X^{n(r-s)}\sum_k {n \brace k}_{r,s} X^kD^k e^x
=\sum_k {n \brace k}_{r,s} x^k e^x,
\end{equation}
equivalently,
\begin{eqnarray}\label{balrs21}
\left(X^rD^s\right)^n e^x =\sum_j \prod_{p=1}^n\left(j+(p-1)(r-s)\right)^{\underline{s}}\ \frac{x^{j+n(r-s)}}{j !}\,.
\end{eqnarray}
Then the comparison of (\ref{balrs20}) and (\ref{balrs21}) produces
\[
\sum_k {n \brace k}_{r,s} x^k =e^{-x} \sum_j\prod_{p=1}^n\left(j+(p-1)(r-s)\right)^{\underline{s}}\ \frac{x^{j}}{j !}\,, 
\]
and extraction of coefficients yields
\begin{eqnarray}
    {n \brace k}_{r,s}
    =\frac{1}{k!}\sum_{j=s}^k(-1)^{k-j}\binom{k}{j}
    \prod_{p=1}^n\left(j+(p-1)(r-s)\right)^{\underline{s}}\,,
\end{eqnarray}
a formula, 
of which, remarkably enough, Scherk had non-trivial cases (see Proposition~\ref{scherkXrD},
p.~\pageref{scherkXrD}).
The case $r<s$ gives rise to similar coefficients,
as results from the 
duality argument of Note~\ref{dual-note}, p.~\pageref{dual-note}. 
\end{note}

\subsection{The form $(\phi(X)D+\rho(X))$ and planted trees.} \label{plant-subsec}
We   now have gates  of the   form $X^s$  (arising from  the monomials
contained  in~$\rho(X)$)  in  addition to  the earlier  ones,  of form
$X^rD$ that originate from $\phi(X)D$.  In combinatorial terms,  there
is thus the additional  possibility of ``planting'' any combination of
members of the  family of increasing trees~$\cal   T^\phi$ on a  root,
whose outdegree   has  possibilities    dictated  by  the    monomials
of~$\rho(X)$.  Graphically:
\[
\hbox{\Img{12.5}{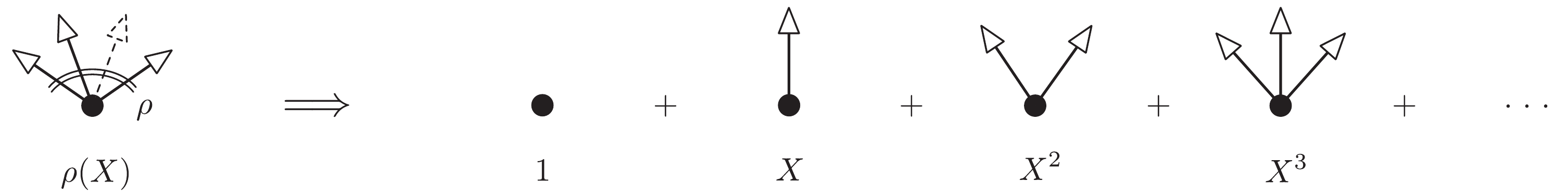}}
\]
\begin{figure}
\begin{center}
\Img{12}{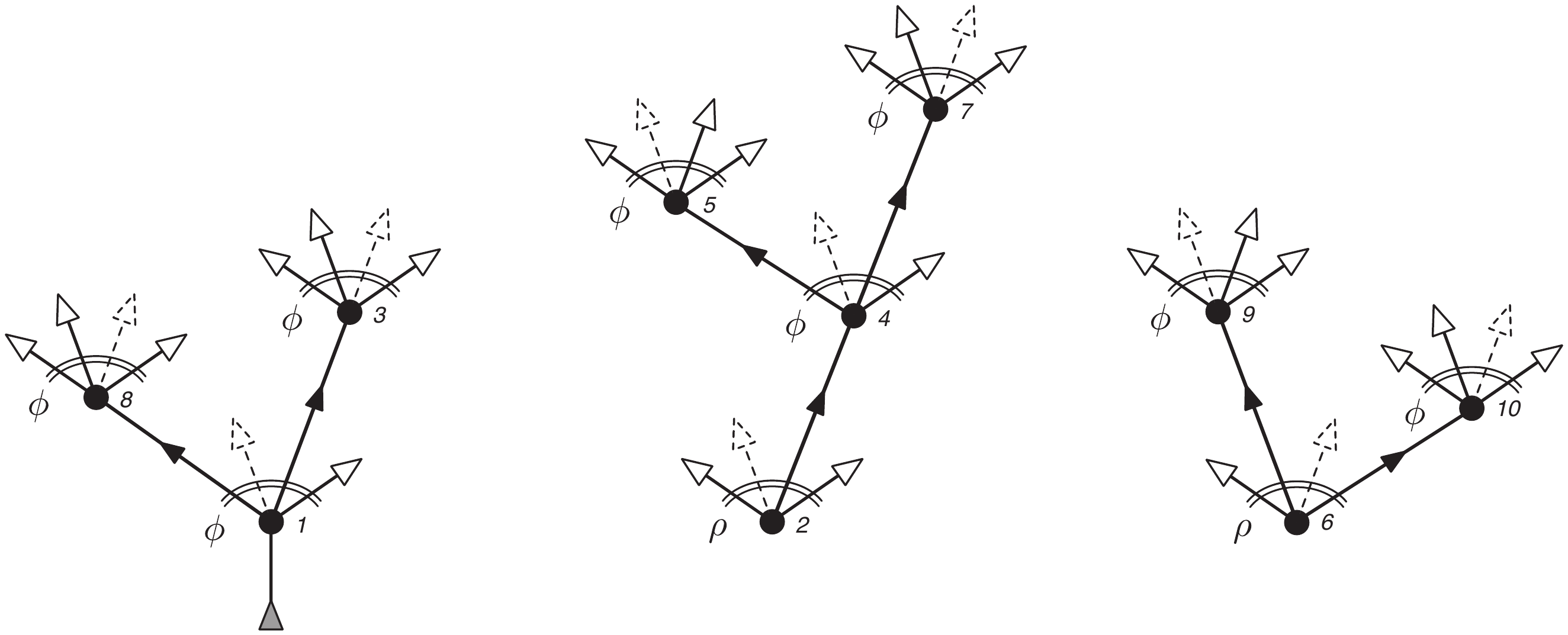}
\end{center}
\caption{\label{planted-fig}\small
A particular diagram (forest) associated to a form
$(\phi(X)D+\rho(X))$,
comprised of one rooted tree and two planted trees.}
\end{figure}
A connected component is then either a planted
tree~$\cal R^\phi$, or a rooted tree of~$\cal T^\phi$.

The corresponding specification is (with $\cal T^\phi$ as
before):
\[
\left\{
\begin{array}{lllll}
\cal R^\phi &=& \ds \sum_{s\ge 0} \rho_s \left(\zz^{\Box}\star \left(u+\cal
  T^\phi\right)^s\right)&& \hbox{(planted trees)}\\
\cal G & =&\ds \set\left(\cal R^\phi+v\cal T^\phi\right)
&&\hbox{(all graphs)}.
\end{array}
\right.
\]
We then have:

\begin{proposition} \label{dat-prop}
The graphs associated with the normal ordering of~$(\phi(X)D+\rho(X))$
have the trivariate generating function
\[
G(z;u,v)=e^{R^\phi(u,z)}\cdot e^{vT^\phi(u,z)},
\qquad
\hbox{where} \quad R^\phi(u,z):=\int_0^z\rho\left(u+T^\phi(u,t)\right)\, dt,
\]
and $T^\phi$ is specified by~\eqref{Phidef} and~\eqref{Tdef}.
\end{proposition}

\begin{note} \emph{Yet another PDE.}
In the case of the PDE
\begin{equation}\label{dat}
\frac{\partial}{\partial z}F(x,z) = \phi(x)\frac{\partial}{\partial x}F(x,z),
\end{equation}
with initial value condition $f(x,0)=f(x)$,
the solution (cf Proposition~\ref{berg-prop}) can be described in the following suggestive way.
\begin{itemize}
\item[] \em Let $Q(y)$ be the primitive function of $1/\phi(y)$ and $Q^{(-1)}(x)$ denote its inverse;
the solution of~\eqref{dat} is
\[
F(x,z)=f(Q^{(-1)}(z+Q(x))).
\]
\end{itemize}
This is, for instance, the form given by Dattoli \emph{et al.} in~\cite{Dattoli97},
Equation~(I.2.18), p.~6; Equation~(I.2.25) of that paper treats the general case 
of Proposition~\ref{dat-prop} above. 
See also the discussion in~\cite{BlHoPeDuSo05}, especially the formula~(1) there.
We shall  briefly return
to the combinatorics of increasing trees, when we discuss multivariate extensions
of these results in Section~\ref{mult-sec}, Note~\ref{char-note}, 
p.~\pageref{char-note}
below.
\end{note}



\section{\bf Binomial forms $(X^a+D^b)$, lattice path models, and continued fractions} \label{cf-sec}

The purpose of this section is to develop \emph{from first principles}
combinatorial models that are now
expressible as    \emph{lattice  paths}. The approach    is an
alternative to the  gates-and-diagrams model of previous sections, though
it is strongly related---we shall indeed present a
correspondence in Subsection~\ref{rook-subsec}.   
The main focus here is  on the  reduction of
powers of the ``\emph{binomial  forms}'' $(X^a+D^b)$. Contrary to what
has been the case until now, the reductions are  often far from being explicit,
in terms of  either  coefficients or  generating functions.  
Nonetheless, this section  reveals  interesting
connections   with other areas   of combinatorics.  In this context, the
``\emph{Fermat forms}''  $(X^r+D^r)$  stand out,  due to a  simple relation
with a yet mysterious class of \emph{continued fractions}.

\subsection{Normal ordering and lattice paths.} \label{cf-subsec}
We first revisit briefly the normal ordering problem. Given the (not 
necessarily normal) representation~$H$ of an operator, 
which may be a power~$\frak{h}^n$ or a generating function~$e^{z\frak{h}}$,
its normal ordering is, by definition, of the form
\[
\nor(H)=\sum_{\alpha,\beta\in\Z_{\ge0}}
c_{\alpha,\beta} X^\alpha D^\beta,
\]
for a family of constants $c_{\alpha,\beta}
\equiv c_{\alpha,\beta}(H)$ that depends on~$H$.
Of special interest for our subsequent discussion
 is the \emph{constant term}~$c_{0,0}$,
which we shall rewrite as
\[
\ct (H) \equiv c_{0,0}(H).
\]
Note that, since $D^\beta$, for $\beta\ge1$ is cancelled by a constant,
we have
\[
\nor(H\, \one)=\sum_{\alpha} c_{\alpha,0}X^\alpha,
\]
where `$H\,   f$' represents  the  application of  the operator $H$  to
the function~$f\equiv f(x)$ and ${\one}$ denotes the constant function equal to unity.
We then have the obvious constant-term identity
\begin{equation}\label{c00}
\ct(H)= 
\left. H\circ \one \right|_{X=0}.
\end{equation}
In this section, we shall mostly be concerned with \emph{constant term identities}.


There is   a known way  to represent  the   normal  ordering process   as a
\emph{transformation}   of  \emph{lattice paths}    in  the   cartesian plane
$\Z_{\ge0}\times\Z_{\ge0}$. This is for instance reviewed
in the elegant
discussion of Varvak~\cite{Varvak05}; see also Subsection~\ref{rook-subsec} below
for a quick review. 
Our approach has some analogy,
but it also differs in essential aspects.

First, we note that the linear differential operator $D$ is characterized by the
way it acts on the canonical basis $\ds\left\{x^k\right\}_{k=0}^\infty$ of monomials, in which
it is represented\footnote{
Given  a predicate~$P$, we denote by $[\![\,P\,]\!]$
its \emph{indicator}, whose value is~1 if $P$ is true and~0 otherwise
(Iverson's notation).} 
 by the infinite matrix
\[
{\bf D}=\left[\begin{array}{cccccc}
0&1&\cdot&\cdot&\cdot&\cdots\\
\cdot&0&2&\cdot&\cdot&\cdots\\
\cdot&\cdot&0&3&\cdot&\cdots\\
\cdot&\cdot&\cdot&0&4&\cdots\\
\vdots&\vdots&\vdots&\vdots&\vdots&\ddots
\end{array}
\right],
\qquad
{\bf D}_{i,j}= i\cdot [\![\,i=j-1\,]\!], \quad i,j\ge0.
\]
Similarly, the linear multiplication operator is 
represented by the matrix
\[
{\bf X}=\left[\begin{array}{cccccc}
0&\cdot&\cdot&\cdot&\cdot&\cdots\\
1&0&\cdot&\cdot&\cdot&\cdots\\
\cdot&1&0&\cdot&\cdot&\cdots\\
\cdot&\cdot&1&0&\cdot&\cdots\\
\vdots&\vdots&\vdots&\vdots&\vdots&\ddots
\end{array}
\right],
\qquad
{\bf X}_{i,j}= [\![\,i=j+1\,]\!], \quad i,j\ge0.
\]
If~$\bf H$ is the matrix obtained from an $\{X,D\}$--operator~$H$ by the substitutions
$X\mapsto\bf X$ and $D\mapsto\bf D$, then, the constant term 
of~$H$ is obtained  as
\begin{equation}\label{ctmat}
\ct(H) = (1,0,0,\ldots) \, {\bf H} \, (1,0,0,\cdots)^{\bf t};
\end{equation}
that is, the constant term of~$H$ equals the upper left corner element of $\bf H$.

We can  now avail ourselves  of  the basic isomorphism between  matrix
products and   paths  in graphs (see,  e.g.,   \cite[p.~9]{Biggs74} or
\cite[\S V.5.1]{FlSe09}).  We  consider here 
\emph{digraphs} (directed graphs),     whose     vertices    are 
the integers~$\Z_{\ge0}$.   The graphs  also have  edges that  are
allowed to bear  \emph{multiplicities}, with the 
multiplicity  of a path being the product
of the multiplicities of the edges that it comprises. Then, the 
\emph{transposed} matrix~$\bf \tilde X$ 
is the incidence matrix of the following graph (with edge-weights underlined)
\[
\begin{array}{cc}\tilde {\bf X}~:\qquad
& \begin{array}{c}\hbox{
\setlength{\unitlength}{1truecm}
\begin{picture}(8,1.2)
\put(0,0){\img{8}{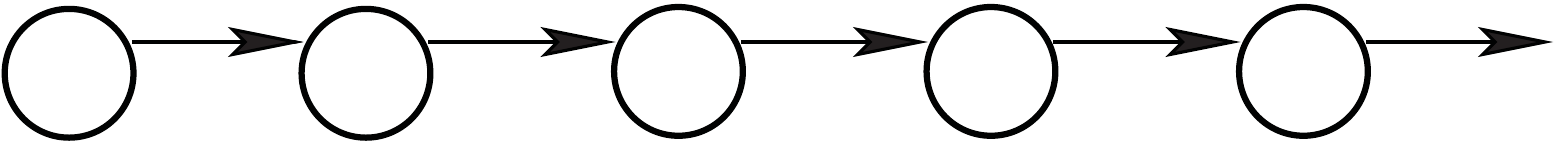}}
\put(0.3,0.25){\bf 0}
\put(1.8,0.25){\bf 1}
\put(3.4,0.25){\bf 2}
\put(5.0,0.25){\bf 3}
\put(6.55,0.25){\bf 4}
\put(0.9,0.7){\ul{\em 1}}
\put(2.4,0.7){\ul{\em 1}}
\put(4.0,0.7){\ul{\em 1}}
\put(5.6,0.7){\ul{\em 1}}
\put(7.15,0.7){\ul{\em 1}}
\end{picture}}\end{array}
\end{array}
\]
while the transposed~$\tilde{\bf D}$ corresponds to
\[
\begin{array}{cc}\tilde{\bf D}~:\qquad
& \begin{array}{c}\hbox{
\setlength{\unitlength}{1truecm}
\begin{picture}(8,1.2)
\put(0,0){\img{8}{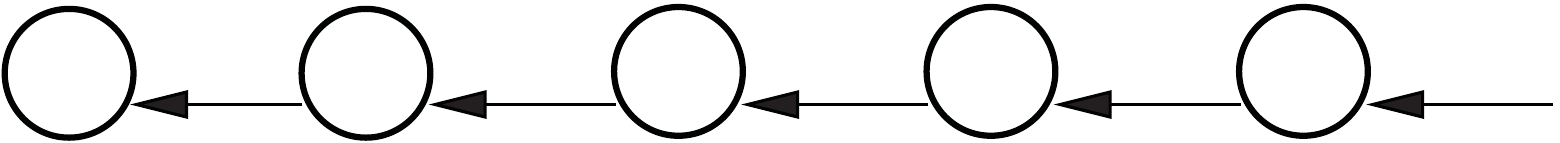}}
\put(0.3,0.25){\bf 0}
\put(1.8,0.25){\bf 1}
\put(3.4,0.25){\bf 2}
\put(5.0,0.25){\bf 3}
\put(6.55,0.25){\bf 4}
\put(1.1,-0.1){\ul{\em 1}}
\put(2.6,-0.1){\ul{\em 2}}
\put(4.2,-0.1){\ul{\em 3}}
\put(5.8,-0.1){\ul{\em 4}}
\put(7.35,-0.1){\ul{\em 5}}
\end{picture}}\end{array}
\end{array}
\]
The analogy with the way $X$ and $D$ operate is striking---just
\emph{interpret state~$\bf k$ as representing the quantity~$x^k$}.

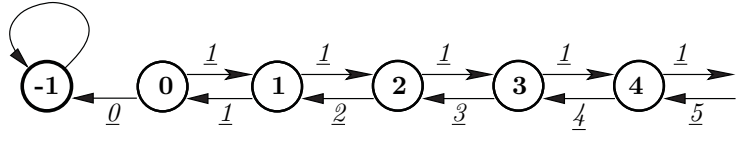
\begin{figure}
\vspace*{0.3truecm}
\[
\hbox{
\setlength{\unitlength}{1truecm}
\begin{picture}(8,1.2)
\put(-1.7,0){\img{1.75}{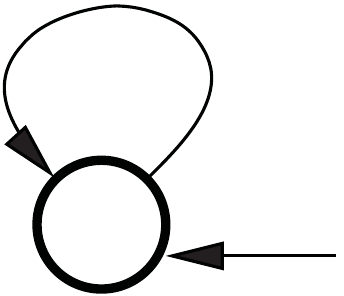}}
\put(-1.35,0.25){\bf -1}
\put(-0.4,-0.1){\ul{\em 0}}
\put(0,0){\img{8}{graphX}}
\put(0,0){\img{8}{graphD}}
\put(0.3,0.25){\bf 0}
\put(1.8,0.25){\bf 1}
\put(3.4,0.25){\bf 2}
\put(5.0,0.25){\bf 3}
\put(6.55,0.25){\bf 4}
\put(0.9,0.7){\ul{\em 1}}
\put(2.4,0.7){\ul{\em 1}}
\put(4.0,0.7){\ul{\em 1}}
\put(5.6,0.7){\ul{\em 1}}
\put(7.15,0.7){\ul{\em 1}}
\put(1.1,-0.1){\ul{\em 1}}
\put(2.6,-0.1){\ul{\em 2}}
\put(4.2,-0.1){\ul{\em 3}}
\put(5.8,-0.1){\ul{\em 4}}
\put(7.35,-0.1){\ul{\em 5}}
\end{picture}}
\]

\caption{\label{weyl-graph}
The Weyl graph corresponding to $(X+D)$.
}
\end{figure}

Of special interest is the graph similarly constructed from the 
matrix $(\tilde{\bf X}+\tilde{\bf D})$ associated  to the operator $(X+D)$;
it is displayed in Figure~\ref{weyl-graph} and we propose 
to call it the \emph{Weyl graph}.
To each monomial $\frak{f}=\frak{f}_1\cdots \frak{f}_n$
in~$X,D$, we associate a path $\pi(\frak{f})$ obtained by 
scanning~$\frak{f}$ \emph{backwards} and transcribing each letter as 
either a leftward move (for a~$D$) or a rightward move (for an~$X$):
\begin{equation}\label{pis}
\pi(\frak{f})=\pi_1\cdots \pi_n, \qquad\hbox{where}\quad
\pi_j = \left\{
\begin{array}{c}
\hbox{a leftward move $(\tilde{\bf D})$ if $\frak{f}_{n-j}=D$}\\
\hbox{a rightward move $(\tilde{\bf X})$ if $\frak{f}_{n-j}=X$.}\\
\end{array}
\right.
\end{equation}
Multiplicities 
along a path are to be cumulated multiplicatively, as said.
In addition,
a $D$-move from vertex~$0$ is simply to be interpreted as
carrying a weight~$0$ (since $D\cdot \one=0$),
which must then make the weight of the whole path to vanish---to take
care of this case, its is convenient to add a \emph{``sink node''}
(tagged by $\bf-1$ in Figure~\ref{weyl-graph})
 to our graph.
With these conventions, we can state:

\begin{proposition} \label{weyl-prop}
Consider a non-commutative monomial~$\frak{f}$ in~$X,D$. The constant term of its normal form
is nonzero if and only if the associated path $\pi(\frak{f})$ in the Weyl graph
of Figure~\ref{weyl-graph},
starting from vertex~$0$, returns to vertex~$0$.
In that case, this constant term is
equal to the multiplicative weight of the path $\pi(\frak{f})$ as
described in~\eqref{pis}.
\end{proposition}
\begin{proof}
The constant term is given by the matrix form of~\eqref{ctmat}.  Here
the  matrix $\bf   H$  is  the   one corresponding  to   the   product
$\frak{f}_1\cdots \frak{f}_n$ of the $X$s  and~$D$s that $\frak{f}$ is
composed  of.  This constant term also equals the upper left corner of the
transposed matrix $\tilde{\bf H}=\pi_1\cdots \pi_n$ (see Equation~\eqref{pis}).
Then, the  classical isomorphism between matrix
products and graphs yields the statement.
\end{proof}

For instance, $\frak{f}=XXDD$ has $\ct(\frak{f})=0$, which corresponds to the fact that, in the reversed form $DDXX$, the first~$D$ takes us to the sink state. For $\frak{g}=DDXX$, we have $\ct(\frak{g})=2$, since, to the reverse form~$XXDD$, there corresponds a path
\[
\def\ra{\mathop{\longrightarrow}}
0~\ra^X ~1~\ra^X ~2~ \ra^D ~1~ \ra^D~ 0,
\]
with multiplicity $1\times 1\times 2\times 1=2$. 
By design, this agrees with the interpretation of state~$k$ as a token
for the monomial~$x^k$:
\[
\def\ra{\mathop{\longrightarrow}}
\ct(DDXX)=(DDXX)\circ \one=2, \qquad\hbox{since}\quad 
1~\ra^X ~x~\ra^X x^2 ~ \ra^D ~2x~ \ra^D~ 2.
\]
This interpretation also gives back the basic property that
the constant term of a monomial~$\frak{f}=\frak{f}_1\cdots\frak{f}_n$ is nonzero iff each 
suffix $\frak{f}_j\cdots \frak{f}_n$ has at least as many $X$s as $D$s  
\emph{and} the number of~$X$s in~$\frak{f}$ equals the number of~$D$s.

Finally, a path in the graph $\Z_{\ge 0}$, with edges of the form $(j,j+1)$
and $(j,j-1)$ is classically interpreted as a 
\emph{lattice path of Dyck type}~\cite[p.~221]{Stanley99}, that is, a polygonal line
in the cartesian plane $\Z\times\Z$: start from the origin and simply associate a 
North-East move (``ascent'')
$\binom{+1}{+1}$
to a rightward step and a South-East move (``descent'')
$\binom{+1}{-1}$
to a leftward step. 
\emph{The multiplicity of such a path is the product
of the (starting) altitudes of descents}.
For instance to $\frak{h}=DXDDXDXX$, there corresponds,
by reversion, the path $\pi=XXDXDDXD$ in the Weyl graph,
which gives rise to the Dyck representation
\[
\hbox{\setlength{\unitlength}{0.65truecm}
\begin{picture}(9,4.0)
\put(0,0.5){\img{6}{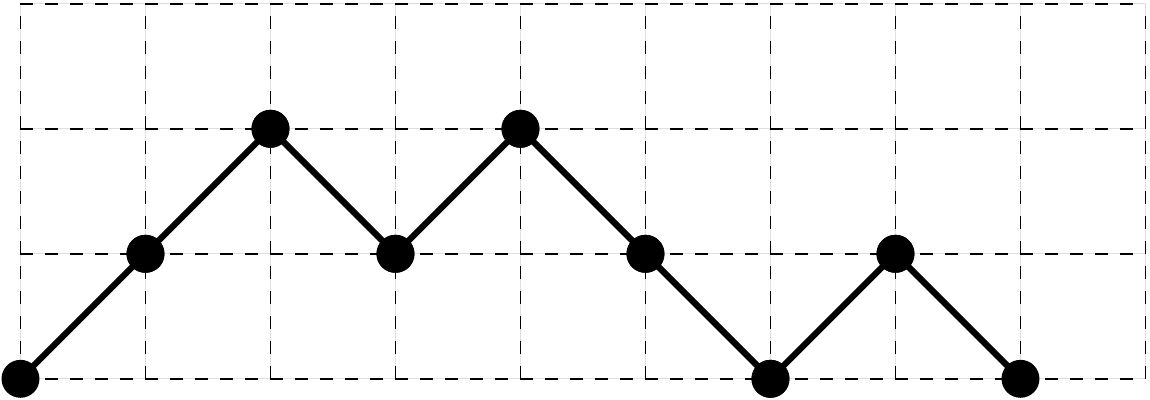}}
\put(0.40,0.0){$X$}
\put(1.40,0.0){$X$}
\put(2.40,0.0){$D$}
\put(3.40,0.0){$X$}
\put(4.40,0.0){$D$}
\put(5.40,0.0){$D$}
\put(6.40,0.0){$X$}
\put(7.40,0.0){$D$}
\end{picture}}.
\]
The multiplicity in this case is $\ct(\frak{h})=
1\times1\times2\times1\times2\times1\times1\times1=4$.

\subsection{Fermat forms $(X^r+D^r)$ and continued fractions.}
Motzkin paths\footnote{
	Such paths are associated to an enriched Weyl graph in which self loops
of the form $(j,j)$ are permitted.
} are lattice paths that, in addition to ascents~$\binom{+1}{+1}$
and descents~$\binom{+1}{-1}$, are also allowed to contain level steps~$\binom{+1}{0}$.
Flajolet~\cite{Flajolet80b} has built elements of a combinatorial theory of
continued fractions, which can be viewed as based 
on the following generating function of Motzkin
paths:
\begin{equation}\label{fla}
F({\bf a}, {\bf d}, {\bf \ell}\!\!{\bf\ell})=\cfrac{1}{1-\ell_0-\cfrac{a_0d_1}{1-\ell_1-\cfrac{a_1d_2}{1-\ell_2-\cfrac{a_2d_3}{\ddots}}}}\,.
\end{equation}
(See also~\cite[\S V.4]{FlSe09} for a concise exposition.)
Here the variables $a_j$, $d_j$, and $\ell_j$ mark, respectively, the 
ascents, descents, and level steps, with (starting) altitude equal to~$j$.
A substitution
\begin{equation}\label{subs}
a_j\mapsto \alpha_j z, \quad d_j\mapsto \delta_j z, \quad \ell_j\mapsto \lambda_jz,
\end{equation}
then yields the ordinary
generating function of Motzkin paths when multiplicities
($\alpha_j,\delta_j,\lambda_j$) are present, with~$z$ marking length.
In this case, 
the continued fraction of~\eqref{fla} becomes
\begin{equation}\label{fla2}
F(z)=\cfrac{1}{1-\lambda_0z-\cfrac{\alpha_0\delta_1z^2}{1-\lambda_1z-\cfrac{\alpha_1\delta_2z^2}{1-\lambda_2z-\cfrac{\alpha_2\delta_3z^2}{\ddots}}}}\, ,
\end{equation}
which is known as a \emph{Jacobi fraction} or \emph{$J$--fraction}~\cite{Perron57,Wall48}

We note that the continued fractions in~\eqref{fla} and~\eqref{fla2}
are \emph{ordinary generating functions}, whereas we have been considering so far 
\emph{exponential} generating functions in association with $e^{z\frak{h}}$.
The connection is via the formal \emph{Laplace transform}~$\lp$ defined as
\[
\lp\left[ \sum_{n=0}^\infty f_n \frac{z^n}{n!} \right] =  \sum_{n=0}^\infty f_n z^n,
\]
which is (formally; sometimes asymptotically, or even analytically) representable
by
\[
\lp[\varphi(z)]=\int_0^\infty e^{-t}\varphi(tz)\, dt.
\]
Thus, we shall obtain here constant term identities for the OGF
\[
\lp\left[ e^{z\frak{h}}\right] \equiv \frac{1}{1-z\frak{h}},
\]
instead of the more customary $e^{z\frak{h}}$.

\smallskip

As a first illustration, we revisit the normal ordering problem relative to 
$\frak{h}=(X+D)$. The constant term of $e^{z\frak{h}}$ is in this case 
the EGF of ``closed'' diagrams; i.e., diagrams with no free input or output
that are relative to $X$--gates and $D$--gates.
For size~$2n$, these are enumerated by the
odd factorials, $1\cdot 3\cdots(2n-1)$, with EGF equal to $e^{z^2/2}$, as we saw already.
On the other hand, the interpretation as lattice paths, with weights~\eqref{subs}
of the form
\begin{equation}\label{hermw}
\alpha_j=1,\quad \delta_j=j, \quad \ell_j=0,
\end{equation}
leads to a continued fraction~\eqref{fla2} that must
correspond to the OGF of odd factorials.
We thus obtain:

\begin{proposition}\label{herm-prop} The normal ordering of~$(X+D)$ corresponds
to the continued fraction expansion
\[
\begin{array}{lll}
\ds \ct\left(\frac{1}{1-z(X+D)}\right) &\equiv &
\ds \ct\left(\lp\left[e^{z(X+D)}\right]\right) \\
\ds {}= \sum_{n\ge0} \left[1\cdot 3\cdots (2n-1)\right]z^{2n}
&=&\ds
\cfrac{1}{1-\cfrac{1\cdot z^2}{1-\cfrac{2\cdot z^2}{1-\cfrac{3\cdot z^2}{\ddots}}}}\,.
\end{array}
\]
\end{proposition}
Analytically, this formal continued fraction is easily derived as 
a special case of Gau{\ss}'s expansion of the quotient of contigous
${}_2F_1$ hypergeometric functions~\cite{Perron57,Wall48}.
Asymptotically, it is associated with the expansion at infinity
of the Gaussian error function. Combinatorially,
its proof~\cite{Flajolet80b} reduces to a bijection originally due to Fran{\c c}on and
Viennot~\cite{FrVi79}, itself based on 
a linear scan of the arch diagram representation
of involutions (see, e.g., \cite[Ex.~5.10, p.~333]{FlSe09}). 
The expansion of Proposition~\ref{herm-prop} is finally 
tightly coupled with Hermite polynomials, hence the name ``Hermite histories'' chosen
by Viennot for Motzkin--Dyck paths weighted according to~\eqref{hermw}.

We next turn to the general normal ordering problem of 
the \emph{Fermat form}~$(X^r+D^r)$, with~$r$ a natural
integer. 

\begin{proposition}\label{lag-prop} The normal ordering of~$(X^r+D^r)$ corresponds
to the continued fraction expansion (with $x^{\overline r}=x(x+1)\cdots (x+r-1)$)
\[
\begin{array}{lll}
\ds \ct\left(\frac{1}{1-z(X^r+D^r)}\right) &\equiv &
\ds \ct\left(\lp\left[e^{z(X^r+D^r)}\right]\right)
\\
&=&\ds
\cfrac{1}{1-\cfrac{1^{\overline{r}}\cdot z^2}{1-\cfrac{(r+1)^{\overline{r}}\cdot z^2}{1-\cfrac{(2r+1)^{\overline{r}}\cdot z^2}{\ddots}}}}\,.
\end{array}
\]
\end{proposition}

\begin{proof}
What is involved is the collection of all Dyck
paths in the Weyl graph, such that $X$s and $D$s go by groups of~$r$ identical letters;
for instance $X^3 X^3D^3X^3D^3D^3$, for $r=3$. Then, only vertices
whose values are multiples of~$r$ are reachable. By grouping 
steps~$r$ by~$r$, these paths are seen to be equivalent to paths in
the nearest-neighbor graph with vertex set $\Z_{\ge0}$, but with weights
taken according to the rule (with $x^{\underline r}=x(x-1)\cdots (x-r+1)$)
\[
\alpha_j=1, \qquad \delta_j=(rj)^{\underline{r}},
\qquad \ell_j=0.
\]
An appeal to the continued fraction theorem (as summarized by~\eqref{fla})
 applied to
these condensed paths completes the proof.
\end{proof}

For $r=2$, we can make use of the computations of Subsection~\ref{zig-subsec}
relative to the normal ordering of $(X^2+D^2)^n$, 
to derive the continued fraction identity:
\begin{equation}\label{sqrtcos}
\lp\left[\frac{1}{\sqrt{\cos(2z)}}\right]=
\cfrac{1}{1-\cfrac{1\cdot 2\cdot z^2}{1-\cfrac{3\cdot 4\cdot z^2}
{1-\cfrac{5\cdot 6 \cdot z^2}{\ddots}}}}\,.
\end{equation}
This fraction can otherwise be deduced from expansions due to Stieltjes and Rogers and
relative to $\lp[\sec^\theta z]$; here, $\theta=\frac12$. Rescaling~$z$ to $z/\sqrt{2}$ leads to
a continued fraction for the OGF of the sequence
\[
	1,~ 1,~ 7,~ 139,~ 5473,~ 357721,~ 34988647, ~\ldots,
\]
which is \OEIS{A126156}. 

It is unclear whether explicit expressions can be distilled out of the
expansion of Proposition~\ref{lag-prop}, when $r\ge3$. 
The most intriguing questions in this range is to identify 
the special functions associated to the simplest case~$r=3$, namely,
\begin{equation}\label{f3}
\begin{array}{lll}
\Phi_3(z)&=&
\ds \cfrac{1}{1-\cfrac{1\cdot 2\cdot 3\cdot z^2}
{1-\cfrac{4\cdot 5\cdot6\cdot z^2}
{1-\cfrac{7\cdot 8\cdot9 \cdot z^2}{\ddots}}}}\\
&=&\ds
1+6\,z^2+756\,z^4+458136\,z^6+765341136\,z^8+
\cdots\,.
\end{array}
\end{equation}

\begin{note} 
\emph{On cubic continued fractions.}
Only a few  cubic analogues of~$\Phi_3$
are known. One group is related to the Dixonian elliptic
functions~\cite{Conrad02,CoFl06}; for instance, with a coefficient law that alternates,
depending on the parity of levels,
\[
\int_0^\infty e^{-t}\sm(zt)\, dt
=
\cfrac{z}{1-\cfrac{1\cdot 2^2\cdot z^3}{1-\cfrac{3^2\cdot 4\cdot z^3}
{1-\cfrac{4\cdot 5^2 \cdot z^3}{\ddots}}}}\,,
\]
where the elliptic function~$\sm(z)$ is defined as inverse of an Abelian integral:
\[
\int_0^{\sm(z)} \frac{dy}{(1-y^3)^{2/3}}=z.\]
Another group, of Stieltjes--Rogers--Ramanujan--Ap\'ery fame, is related to the Hurwitz
zeta function,
\[
\zeta(3,x+1)=\sum_{k=1}^\infty \frac{1}{(x+k)^3},
\]
and it contains, for instance (see~\cite[p.~153]{Berndt89}):
\[
\zeta(3,x+1)=\cfrac{1}{2x(x+1)+\cfrac{1^3}{1+\cfrac{1^3}{6x(x+1)+\cfrac{2^3}{1+\cfrac{2^3}
{10x(x+1)+\cdots}}}}}\,,
\]
which is somehow related to Ap\'ery's proof~\cite{Poorten79}
of the irrationality of~$\zeta(3)$.
\end{note}

From the previous examples, it is easily realized that
the general scheme giving rise to explicit continued fraction 
expansions is when $\frak{h}$ is of the form
\[
X^rD^s+X^sD^r+\sum h_jX^jD^j.
\]
The steps 
in the cartesian plane are now of the three 
vectorial types $\binom{r-s}{1}$, $\binom{s-r}{1}$, and $\binom{0}{1}$,
which can be collapsed by a linear change of coordinates to the three types
that serve to form Motzkin paths---hence continued fractions.
The weights are invariably a polynomial function of the altitude 
(i.e., the index~$k$). In this way, continued fractions 
with polynomial coefficients of all degrees can be 
constructed, though both the special functions aspects 
(the existence of explicit forms) and the combinatorics 
(bijections with simple ``natural'' combinatorial structures)
 remain unclear at this level of generality.

\begin{note} \emph{Horzela structures and $(XD^2+X^2D)$.}
Imagine\footnote{
This case was suggested to us by Andrzej Horzela (private communication, 2010).}
 a universe where particles may be subject both to fission (a particle
gives rise to two particles) and fusion (two particles merge to give rise
to a single particle). The diagrams are those associated with
$(XD^2+X^2D)$, where gates are of type either $\mathbb{Y}$ or its horizontally
flipped image. Thus, the graphical representations are
a complex network of trees and ``inverted'' trees. 
The ordinary generating function 
$H(z)$ of the diagrams with one root (one input)
 and one surving particle (one output) is then
\[
\begin{array}{lll}
H(z)&=& \ds \cfrac{1}{1-\cfrac{1^2\cdot 2\cdot z^2}{1-\cfrac{2^2\cdot 3\cdot z^2}
{1-\cfrac{3^2\cdot 4 \cdot z^2}{\ddots}}}}
\\
&=& \ds 1+2\,{z}^{2}+28\,{z}^{4}+1256\,{z}^{6}+129904\,{z}^{8}+
25758368\,{z}^{10}+\cdots\,.
\end{array}
\]
Such $H$-structures are loosely evocative of cellular decompositions (combinatorial maps)
of surfaces of arbitrary genus, a subject of active research (see,
e.g.,~\cite{BrItPaZu78,Chapuy09} and references therein).
\end{note}

Whenever available, continued fractions are associated with a rich 
set of identities, due most notably to their connection 
with orthogonal polynomials~\cite{Flajolet80b,Perron57,Wall48},
and they potentially give rise to efficient computational procedures~\cite{KaDu95}.

\subsection{The general binomial case $(X^a+D^b)$.}
The general correspondence expressed by Proposition~\ref{weyl-prop} (for monomials) is still
applicable to the normal ordering of $(X^p+D^q)$. Thus, there 
exists a transcription in terms of  paths in the Weyl graph, where
the allowed steps are either rightward moves of amplitude~$a$
or leftward moves of amplitude $-b$;
the multiplicity of a path, as before, is the product of
the starting altitudes of descents.

However, when $a\not=b$, the connection with continued fractions is lost, as the 
case  is no longer reducible to the Dyck paradigm.
The simplest instances are $(X^3+D^2)$ and $(X^4+D^2)$
(or their duals, $(X^2+D^3)$ and $(X^2+D^4)$).
These formally correspond to the \emph{anharmonic quantum oscillator}
with a \emph{cubic} or \emph{quartic potential}. The 
extensive quest for explicit solutions
in this context indicates the difficulty of finding
connections with the most classical special functions.
For the record, we tabulate here the following constant terms:
\begin{equation}\label{xd234}
\begin{array}{lll}
\ds \left.\ct\left[(D^2+X^3)^n\right]\right|_{n=0,\ldots,10}~: &
1, 0, 0, 0, 0, 864, 0, 0, 0, 0, 1157815296\\
\ds \left.\ct\left[(D^2+X^4)^n\right]\right|_{n=0,\ldots,10}~: &
1, 0, 0, 24, 0, 0, 49536, 0, 0, 828002304, 0\,.
\end{array}
\end{equation}
These seem not to be related to existing sequences in the \emph{OEIS}.

\begin{note} \emph{Duchon's clubs.} 
The lattice paths relative to $(X^3+D^2)$, 
but when weights of both leftward and rightward steps
are set to~$1$, appears in the literature under the name of 
``Duchon's numbers'', which enumerate the combinatorial class of 
``Duchon's clubs''~\cite{BaFl02,Duchon00}.
The sequence of nonzero Duchon numbers $(\delta_{5n})$, 
\[
1, 2, 23, 377, 7229, 151491, 3361598, 77635093,
\]
is \OEIS{A060941}. In the figurative description of~\cite[p.~53]{BaFl02}:
\begin{quote}\small\noindent
``A club opens in the evening
and closes in the morning. People arrive by pairs and leave in threesomes. What is
the possible number of scenarios from dusk to dawn as seen from the club's entry?''
\end{quote}
In this simplified situation (multiplicities of steps are disregarded),
we are only considering the possible evolutions in time of the club's population.
The 
\emph{ordinary} generating function~$\delta(z)$
 is then an \emph{algebraic function}, here of degree~10,
 \[
 \hbox{\footnotesize$z{\delta}^{10}+5z{\delta}^{9}+5z{\delta}^{8}-10z{\delta}^{7}-15z{\delta}^{6}+11z{\delta}^{
 5}+ \left( 15z-1 \right) {\delta}^{4}+ \left( 1-10z \right) {\delta}^{3}-5
 z{\delta}^{2}+5z\delta-z=0$},
 \]
and one has the 
following explicit expression 
\[
\delta_{5n}=\sum_{i=0}^n \frac{1}{5n+i+1}\binom{5n+i}{n-i}\binom{5n+2i}{i}.
\]
It is interesting to note  that the case of $(D^2+X^3)$,
or, equivalently~$(X^2+D^3)$, in the first line of~\eqref{xd234}, corresponds to
the situation where,
furthermore, we count complete evolutions,
in which, additionally, \emph{identities of individuals are taken into account}.
\end{note}

%
%



%


\section{\bf Related frameworks} \label{frameworks-sec}

In this section,  we  return to the   the gates-and-diagram   model of
creation--annihilation  operators  of Sections~\ref{maindef-sec}--\ref{semilin-sec}.  We
first   discuss a  model  of the   reduction to   normal form that  is
expressed in terms of rook placements on a board.  This rook model can
be  derived from first  principles  (Wick's Theorem, cf
Note~\ref{defder-note}, p.~\pageref{defder-note}, and~\cite{Varvak05}), 
but it
can  also be  attached to the  basic construction  of diagrams and the
equivalence asserted by   Theorem~\ref{eqp-thm}. 
As we explain in Subsection~\ref{rook-subsec}, one of the  interesting features
of the latter approach is the possibility
of  relating diagrams to the lattice-path methods of
the previous  section, Section~\ref{cf-sec}:
the connection is achieved by a simple ``scanning algorithm''.
Next, in Subsection~\ref{qana-subsec}, we briefly revisit diagrams
within the framework  of $q$-analogue theory, where the $q$-difference
operator~$\Delta$ replaces the ordinary differential operator~$D$
and crossing numbers of (plane embedded) diagrams 
are shown systematically to produce $q$-analogues.

\subsection{Rook placements, lattice paths, and diagrams.}\label{rook-subsec}
We first describe some simple combinatorics that relates diagrams
and rook placements on a chessboard. This thread  closely follows an insightful
article of Varvak~\cite{Varvak05}; see also~\cite{BlDuHoPeSo08,SoDuBlHoPe04}.
 The message here
is that one can describe the \emph{complete history} of the construction
of diagrams by means of  certain kinds of lattice configurations.

\def\cont{\operatorname{cont}}

We fix a basis $\cal H$ of gates that, for notational convenience, we take to be unweighted.
In other words, we are considering a polynomial with coefficients in $\{0,1\}$,
\[
\frak{h}:=\sum_{j=1}^m X^{r_j}D^{s_j},
\]
for a finite set of distinct pairs $(r_j,s_j)\in\Z_{\ge 0}\times\Z_{\ge 0}$. (The general case is easily 
treated by suitably accommodating weights.)  Let $g_j$ represent a generic gate of
type $X^{r_j}D^{s_j}$. A diagram~$\delta$ of size~$n$ is determined by the collection
$(g_{i_1},\ldots,g_{i_n})$ of its gates, where $g_{i_j}$ is the 
type of the gate associated to the inner node labelled~$j$,
together with the interconnection pattern, which describes the way the outputs of gates
are connected with the inputs of some other (later arrived) gates. 

\begin{figure}\small
\begin{center}
\Img{9}{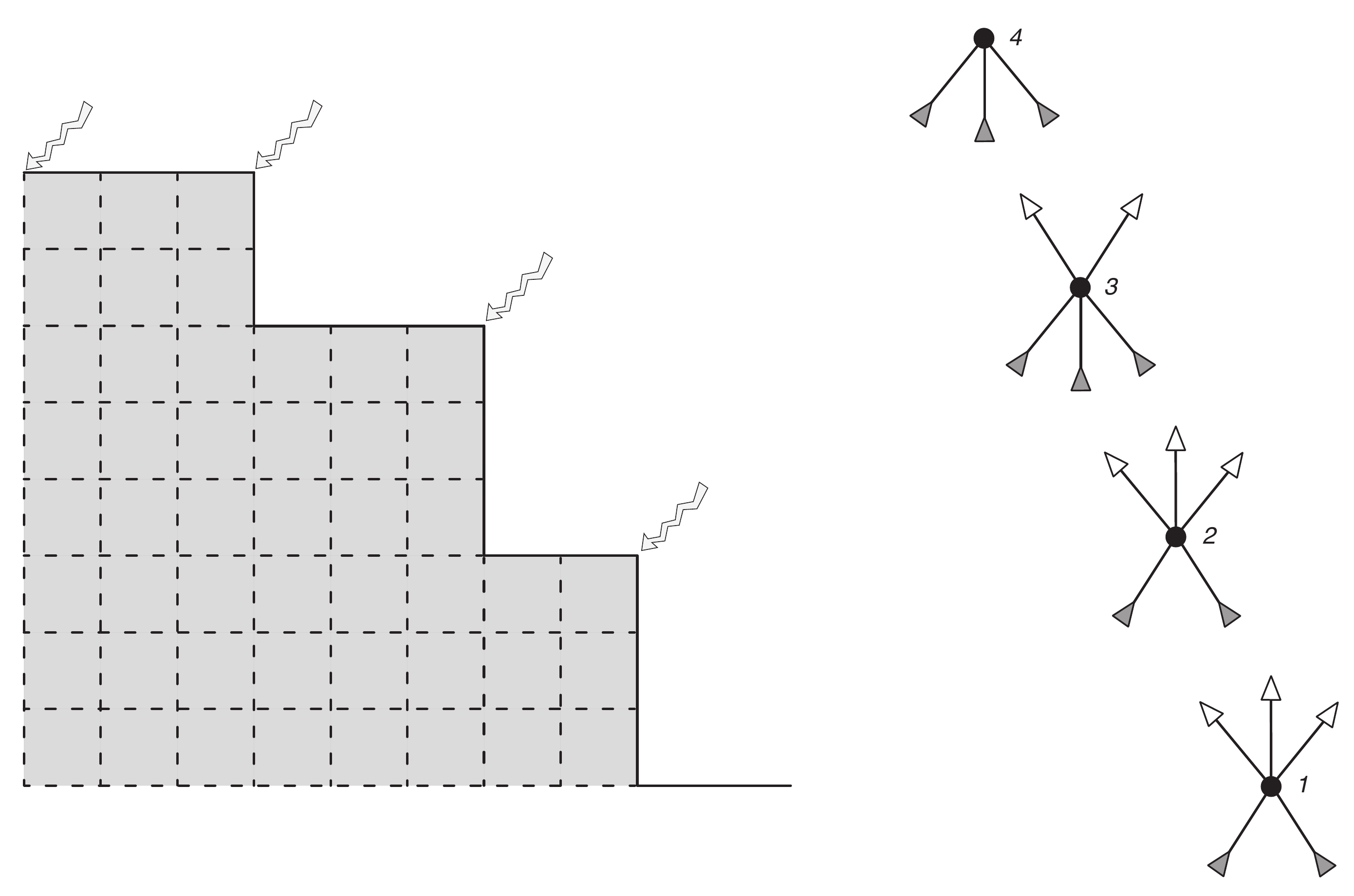}
\end{center}
\caption{\label{rook0-fig} \small
The correspondence between the
composition of a diagram and its contour,
\[
|D^3 \, |X^2D^3\, | X^3D^2\, | X^3D^2,
\]
interpreted in the discrete plane.}
\end{figure}

As we know (Note~\ref{defder-note}),
 a diagram is uniquely associated with a particular reduction of a monomial
\begin{equation}\label{cont0}
X^{r_{i_n}} D^{s_{i_n}}\, X^{r_{i_{n-1}}}D^{s_{i_{n-1}}}\,\cdots\, X^{i_{r_1}}D^{s_{i_1}},
\end{equation}
using at each stage either of the two rewrite rules $DX\mapsto XD$ or $DX\mapsto 1$.
We then define the \emph{contour} of~$\delta$ as a word over the extended alphabet $\{X,D,|\}$,
where ``$|$'' serves as a separator and is called a \emph{pin},
as follows:
\begin{equation}\label{cont1}
\cont(\delta):=
| X^{r_{i_n}}D^{s_{i_n}}\,|\,  X^{r_{i_{n-1}}}D^{s_{i_{n-1}}}|\,\cdots\,| X^{i_{r_1}}D^{s_{i_1}}.
\end{equation}
The contour in this sense is thus an  unambiguous representation\footnote{
The pins serve to disambiguate the parsing of a  word over the alphabet $\{X,D\}$.
They are needed in a few cases; for instance, if $\frak{h}$ contains
$(X+X^2)$, since $X^2=X\cdot X$ can be parsed in two different ways;
or in the case of $(X+D+XD)$, since $X^2D^2=X\cdot XD\cdot D = X\cdot X \cdot D \cdot D$.
 They are superfluous in cases
such as $(X+D)$, ($X^2+D)$, $(X^2D^3)$, and so on, as considered by Varvak~\cite{Varvak05}.}
of the gates that $\delta$
is comprised of.  Next we represent
the contour as a polygonal path in the discrete plane $\Z\times\Z$, 
with~$X$ being interpreted as the vertical unit vector $(0,-1)$ and~$D$ being the horizontal
unit vector~$(1,0)$. In the discrete plane, this polygonal line determines what is
known as a \emph{Ferrers board}~\cite{Comtet74,Varvak05}, see~Figure~\ref{rook0-fig},
where the pins are represented by arrows.

\begin{figure}\small
\begin{center}
\Img{9}{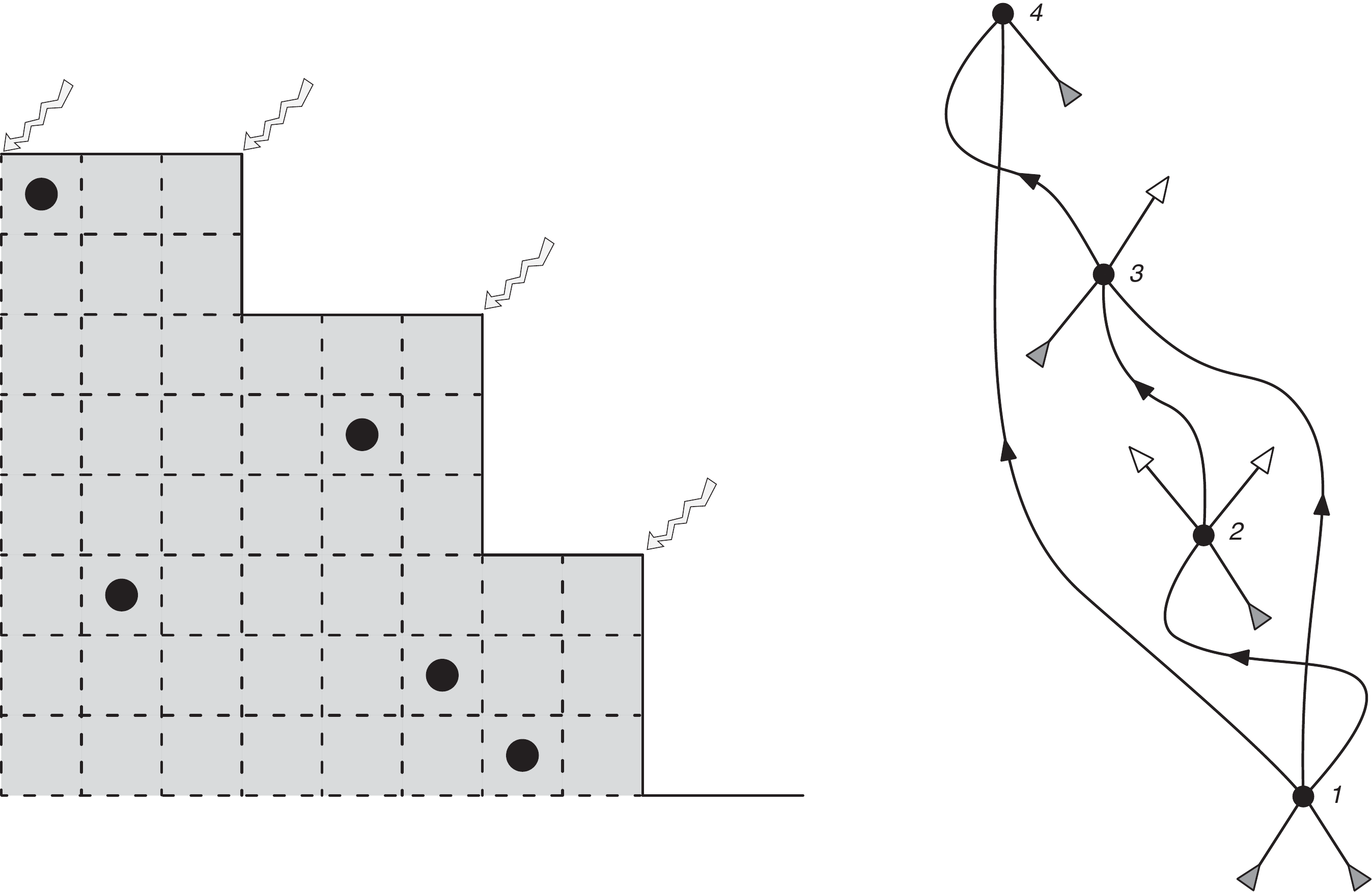}
\end{center}

\caption{\label{rook1-fig} \small
The correspondence between a diagram and 
a rook placement in the square lattice.}
\end{figure}

In order to obtain a bijective encoding of circuits, one needs to augment  the
contour~$\cont(\delta)$
 so as to encode all the information relative to interconnections of links in the diagram~$\delta$.
We now consider the Ferrers board whose upper envelope is 
the contour. Since the order of application of operators 
in a monomial is from the right, the contour (as of~\eqref{cont0}
or~\eqref{cont1}) is scanned  \emph{from the right}. The procedure is
as follows (Figure~\ref{rook1-fig}):
\begin{itemize}
\item[] In each vertical column (which corresponds to a letter~$D$) do one of two
things:
\begin{itemize}
\item[---] either put a single dot (a ``rook'') in one of the column's cells:
a dot  in the $j$th cell from the bottom of the board means that the $j$th outgoing link
(conventionally starting from the right)  present in the partial diagram 
at this stage is connected to the currently active input link of the current gate 
(this corresponds
to the~$D$ currently taken into account);
\item[---] or put nothing in the column: this corresponds to a
  dangling (unattached) $D$-link.
\end{itemize}
\end{itemize}
A moment's reflection should convince our  reader that
at most one dot/rook can be placed in each line (since the output of
a gate can be ``closed'' at most once by the input of a later gate),
as well as in each column (by construction). This is precisely the rule that constrains the
placement of \emph{non-attacking rooks} on a chessboard---here in the case
of a board with unconventional right and upper boundaries~\cite[Ch 7]{Riordan80}.

We now arrive at  a general statement, which is a version adapted to our needs of
Varvaks' Theorem 3.1 in~\cite{Varvak05}.

\begin{proposition} \label{rook-prop}
The coefficient of $X^aD^b$ in the normal form $\nor(\frak{h}^n)$ is equal to the number
of rook placements in Ferrers boards, whose contour is consistent 
with~$\frak{h}^n$, that have~\underline{$a$} rook-free rows and~\underline{$b$}
 rook-free columns.
\end{proposition}

Varvak's proof essentially amounts to an appeal to Wick's Theorem (Note~\ref{defder-note},
p.~\pageref{defder-note}). The proof given above, without being 
radically different, amounts to applying 
to diagrams a \emph{scanning algorithm} 
whose general scheme is as follows:
\begin{quote}\small
\noindent
{\bf Scanning algorithm.}
Gates are scanned in increasing order of their labels;
gate types, encoded by corresponding vectors of~$\Z\times\Z$, are generated 
(these form the steps of a ``contour'' read from the bottom right);
additional information (a sequence of numbers) 
is supplied to specify the interconnection pattern of
the new gate with its predecessors. 
\end{quote}

A variety  of encodings are  possible, due  to the  flexibility of the
coding conventions, when implementing 
the scanning  algorithm.  For instance,  regarding the reduction of
$(X+D)$, we  may  associate a northeast step~$\binom{1}{1}$  to an~$X$
and a southeast step~$\binom{1}{-1}$ to a~$D$. The constant term in the 
normal form of $(X+D)^{2n}$ is then associated with the collection of all
Dyck paths of length~$2n$, each such path being augmented
by a sequence of numbers that serves to encode the interconnection pattern 
of gates in a particular $\{X,D\}$-diagram. In this specific case, it can be seen that 
the allowed number sequences  are such that
an ascent has one possibility whereas a descent has $\ell$ possibilities,
if it corresponds to a step with initial altitude~$\ell$. 
The augmented paths produced by the scanning algorithm thus
correspond exactly to the weighted Dyck paths considered
in Equation~\eqref{hermw}
and Proposition~\ref{herm-prop}, to be later revisited in Proposition~\ref{invinv-prop}
and Figure~\ref{invinv-fig} in relation with $q$-analogues. 
Beyond this particular example,
\emph{the scanning algorithm generally links circuit-based models and 
the direct matrix--Weyl graph approach of Section~\ref{cf-sec}.}

%


\subsection{$q-$analogues and the difference operator.} \label{qana-subsec}
In this  subsection, we propose to  discuss briefly the way the theory
of  gates   and    diagrams    leads    in  a   systematic   way    to
\emph{$q$-analogues}. The starting point is the \emph{$q$-difference operator} 
$\Delta\equiv\Delta_q$ defined by
\begin{equation}\label{defdel}
\Delta f (x)=\frac{f(qx)-f(x)}{(q-1)x}.
\end{equation}
We  shall take~$q$ to be a  real number in~$[0,1]$   and note that, as
$q\to1$, the operator   $\Delta_q$  becomes the   standard  derivative
operator $D$.  
The operators~$X$ and~$\Delta$ satisfy 
the commutation relation 
\begin{equation}\label{qcom}
\Delta X - qX\Delta = 1,
\end{equation}
to be compared to~\eqref{cran}. A normal form, with all $X$s preceding all~$\Delta$s
can always been attained by the rewrite rule analogous to~\eqref{red}:
\[
\Delta X \redu 1+qX\Delta.
\]
(The book  by Kac and Cheung~\cite{KaCh02}  provides an
undemanding introduction to basic properties of such operators.)

We shall now build in stages a combinatorial interpretation of 
arbitrary compositions of~$\Delta$s and $X$s. 
To start with, we observe the effect of $\Delta$ on (formal or analytic)
power series: if $f(x)=\sum f_n x^n$, then
\begin{equation}\label{defdel2}
\Delta_q(f)(x) = \sum_{n\ge0} f_n \frac{1-q^n}{1-q} x^{n-1}.
\end{equation}
In other words, $\Delta$ operates linearly and, on the monomial~$x^n$, its effect is
to produce a monomial  of degree~$(n-1)$:
\begin{equation}\label{delxn}
\Delta x^n= [n]x^{n-1}, \qquad\hbox{where}\quad
[n] \equiv [n]_q :=\frac{1-q^n}{1-q}=1+q+\cdots+q^{n-1},
\end{equation}
using classical notations.

\smallskip
{\bf\em Combinatorics of $\Delta x^n$.}
From~\eqref{delxn},    the   operator~$\Delta$  admits   an    obvious
interpretation: think of the monomial~$x^n$ as a row of $n$
occurrences of the variable~$x$;
pick up (in all possible ways) one of  the $x$--occurrences and record
with a power $q^{k-1}$ the  situation where the $k$th occurrence  from
the left  has  been picked up;  finally replace  the chosen occurrence
of~$x$ with the neutral element~$1$  (the identity). 
(Under this form,
it is  apparent that   the  $\Delta$--operator is a deformation    of the
standard  derivative    operator:   see Note~\ref{defder-note}.)

A  visual  image of the   action of the $\Delta$ operator on an
arbitrary monomial can be   given in terms of
``speed dating clubs'' as follows.
\begin{quote}\noindent\small
Imagine a longish  hall in which there  is a
long row of tables. At each table there sits one $x$--element.  A
particular operation consists in  letting in, at  the  entrance of the
hall,   on the \emph{left}, a~$\Delta$--element,  who  will     eventually   pick up  a
table. However,  for each table  that the $\Delta$--element  passes by
(but does not pick up), he has to pay for  a drink (the cost of drinks
is recorded by~$q$).  Once  he settles for a  table,
the  $x$--element  at that table ceases  to  become available for further drinks
and solicitations. 
What  the $\Delta$ operator does is  simply to keep  track  of all the possible
scenarios, when \emph{one} $\Delta$--individual is let in. \par
\end{quote}  
For instance, with the 
example of Note~\ref{defder-note}:
\[
\Delta(xxxx)~=~\overbrace{\not{x}xxx}^{q^0}+\overbrace{x\!\!\!\not{x}xx}^{q^1}
+\overbrace{xx\!\!\!\not{x}x}^{q^2}+\overbrace{xxx\!\!\!\not{x}}^{q^3}
~=~[4]_qx^3.\]

\smallskip
{\bf\em Combinatorics of $\Delta (x^nf) \equiv \Delta X^n (f)$.}
In our treatment of operator calculus, an identity $\frak{U}=\frak{V}$
  between operators means that, for an arbitrary $f$ (on which the
  operators act),
we have $\frak U f=\frak V f$. Here, the nature of $f$ is immaterial
and, in particular, when dealing with normal forms, quantities such as
$\Delta f, \Delta^2 f, \ldots$ are to be considered as
non-simplifiable. 
We can then amend the combinatorial interpretation of the previous
paragraph as follows. Imagine now that the longish hall has an exit on
the right,
leading to a courtyard (biergarten) designated as ``$f$'',
where $\Delta$ elements can
accumulate
if they haven't picked up an $x$--element in the hall. This situation
is seen to model the action of $\Delta$ on $(x^n f)$ or what amounts
to the same, the action of $\Delta X^n$ on an arbitrary~$f$.
Here is an example:
\[
\hbox{\Img{11}{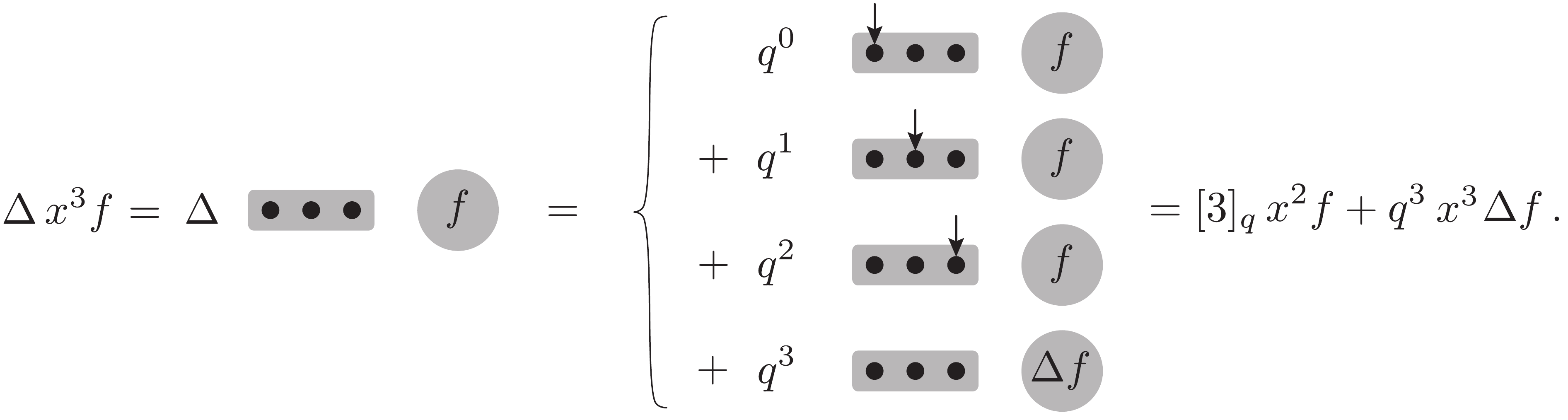}}
\]
(As a simple exercise, the reader may wish to verify combinatorialy the general identity
$\Delta (f\cdot g)(x)=\Delta f(x)\cdot g(x)+f(qx)\cdot \Delta g(x)$.)

\smallskip
{\bf\em Compositions.}
The interest of this visual image is that it describes well what goes on upon 
iteration. For instance, we can see that
\[
\Delta^n(x^n)=[n]\cdot [n-1]\cdots [2]\cdot [1]\equiv [n]!_q,
\]
where   the right  hand side     gives   the generating function    of
permutations counted according to the 
number of inversions\footnote{
With our notations, the number of inversions is the number of pairs of values
$(j,k)$ such that $j<k$ and the value~$j$ is placed on the right of the value~$k$. 
}; see Figure~\ref{perminv-fig}.
This computation is also nothing but the reduction of $\Delta^nX^n(1)$ to normal form.

More generally, given an arbitrary term~$X$ and~$\Delta$, 
each possible expansion that it can give rise to, when
applied to an arbitrary~$f$, can be described by a succession of
operations of adding a table with an $x$-girl\footnote{
	The operator $X$ clearly corresponds
	to placing a new table
 at the beginning of an existing hall--courtyard configuration.}
or launching a $\Delta$-boy into the game.

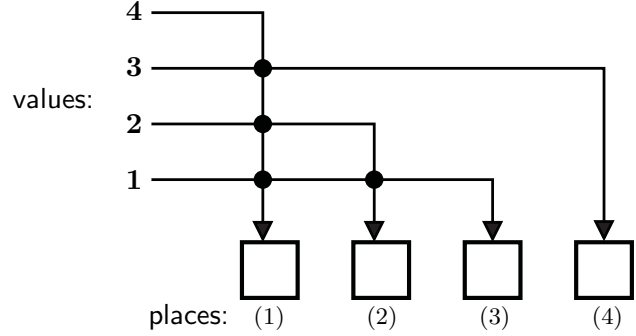
\begin{figure}\small

\begin{center}
\setlength{\unitlength}{1truecm}
\begin{picture}(7,4.2)
\put(0.5,0){\img{6.5}{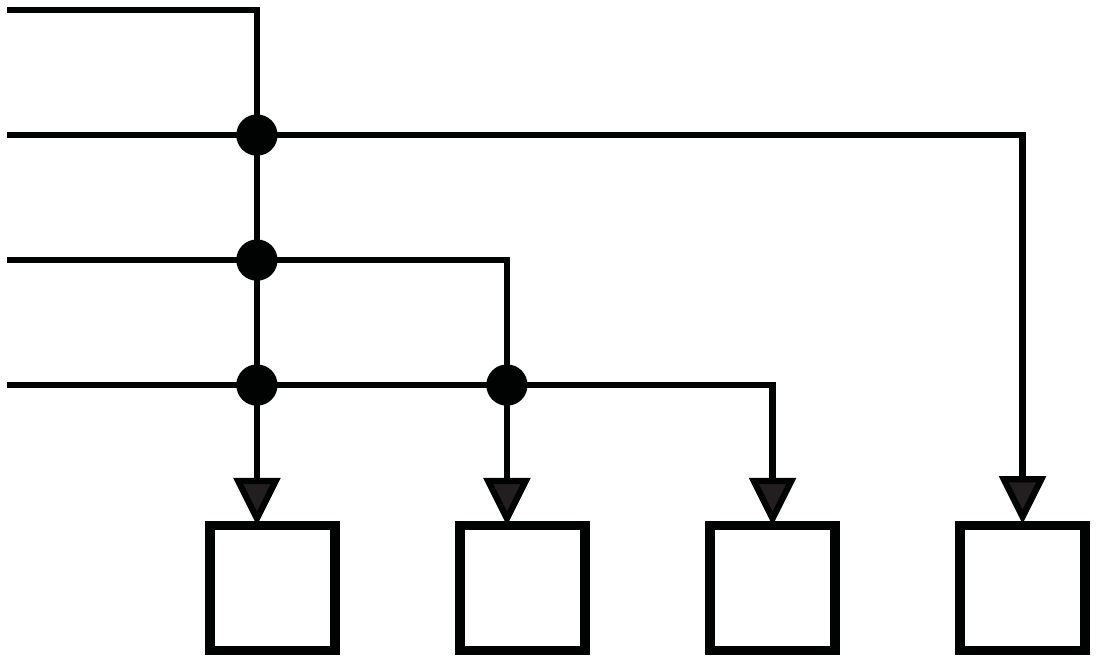}}
\put(0.2,1.5){\large\bf 1}
\put(0.2,2.25){\large\bf 2}
\put(0.2,3.0){\large\bf 3}
\put(0.2,3.75){\large\bf 4}
\put(1.9,-0.3){$(1)$}
\put(3.4,-0.3){$(2)$}
\put(4.9,-0.3){$(3)$}
\put(6.4,-0.3){$(4)$}
\put(0.5,-0.3){{\sf\large  places:}}
\put(-1.3,2.585){{\sf\large   values:}}
\end{picture}
\end{center}

\caption[A permutation  of size~ 4, such as
where the first value~1 goes to place~3, and so on,
corresponds to a particular expansion of : it is seen that the number of
crossing links equals the number of inversions in the permutation.
]
{\label{perminv-fig}\small
A permutation  of size~ 4, such as
$\sigma=\left(\begin{array}{cccc} 1&2&3&4\\ 3&2&4&1\end{array}\right)$,
where the first value~1 goes to place~3, and so on,
corresponds to a particular expansion of $\Delta^4 [x^4]$: it is seen that the number of
crossing links equals the number of inversions in the permutation.
}
\smallskip

\noindent
\hrule
\end{figure}

\smallskip
{\bf\em Diagrams.}
The discussion above leads to a natural extension of
the notion of diagram. A quantity $X^r\Delta^s$ will again be represented 
by a gate in the sense of Section~\ref{maindef-sec}, Equation~\eqref{simpf}.
However, when gates are composed to form graphs, 
a definite convention should be observed.
\begin{quote}\noindent\small
All edges are drawn as segments or half-lines parallel to the axes.
Each new gate $\gamma$, which is added to  an existing
diagram~$\delta$ (where the latter
involves  only  smaller labels)  is placed   on  the north-west of the
diagram.   The inputs of  gates~$\gamma$ are  drawn  horizontally, pointing to
the right; the outputs are  drawn  vertically, pointing upwards.
Inputs of~$\gamma$ not connected to an output of $\delta$ are prolonged
as half-lines to the right of the diagram. Outputs of~$\gamma$ not connected to a
later gate's input are prolonged upwards as half-lines. See Figure~\ref{ediag-fig}.
\par
\end{quote}
This convention corresponds to the fact that, in a gate of type (say) $X^2\Delta^3$,
the first input corresponds to a first application of $\Delta$, and so on.
A diagram represented in this way will be called an \emph{embedded diagram}
(Figure~\ref{ediag-fig}). The number of pairs of edges that cross is called the 
\emph{crossing number} of the embedded diagram.

The observations relative to the combinatorics of~$\Delta$ and the case
of permutations then immediately lead to the following statement.

\begin{theorem}[$\Delta$--Equivalence Principle] \label{eqp2-thm} 
Consider a polynomial $\frak{h}$ with normal form form
\begin{equation}\label{preeqp2}
\frak{h}:=\sum_{(r,s)\in \cal H} w_{r,s} X^r \De^s.
\end{equation}
Then the normal ordering of the power~$\frak{h}^n$,
\begin{equation}\label{eqp2}
\nor(\frak{h}^n)=\sum_{n,a,b} c_{n,a,b}(q) X^a \De^b,
\end{equation}
is such that the polynomial $c_{n,a,b}(q)$ coincides with the \emph{total weight}
of (labelled) embedded diagrams 
that admit~$\cal H$ as a basis  weighted by~$w$,
have size~$n$, and are comprised of~\underline{$a$} outputs and~\underline{$b$} inputs, 
where the variable~$q$ marks the number of crossings.
\end{theorem}

\begin{figure}{\small
\begin{center}
\hbox{\setlength{\unitlength}{0.8truecm} 
\qquad\begin{picture}(6.2,6.7)
\put(0,0){\img{4.8}{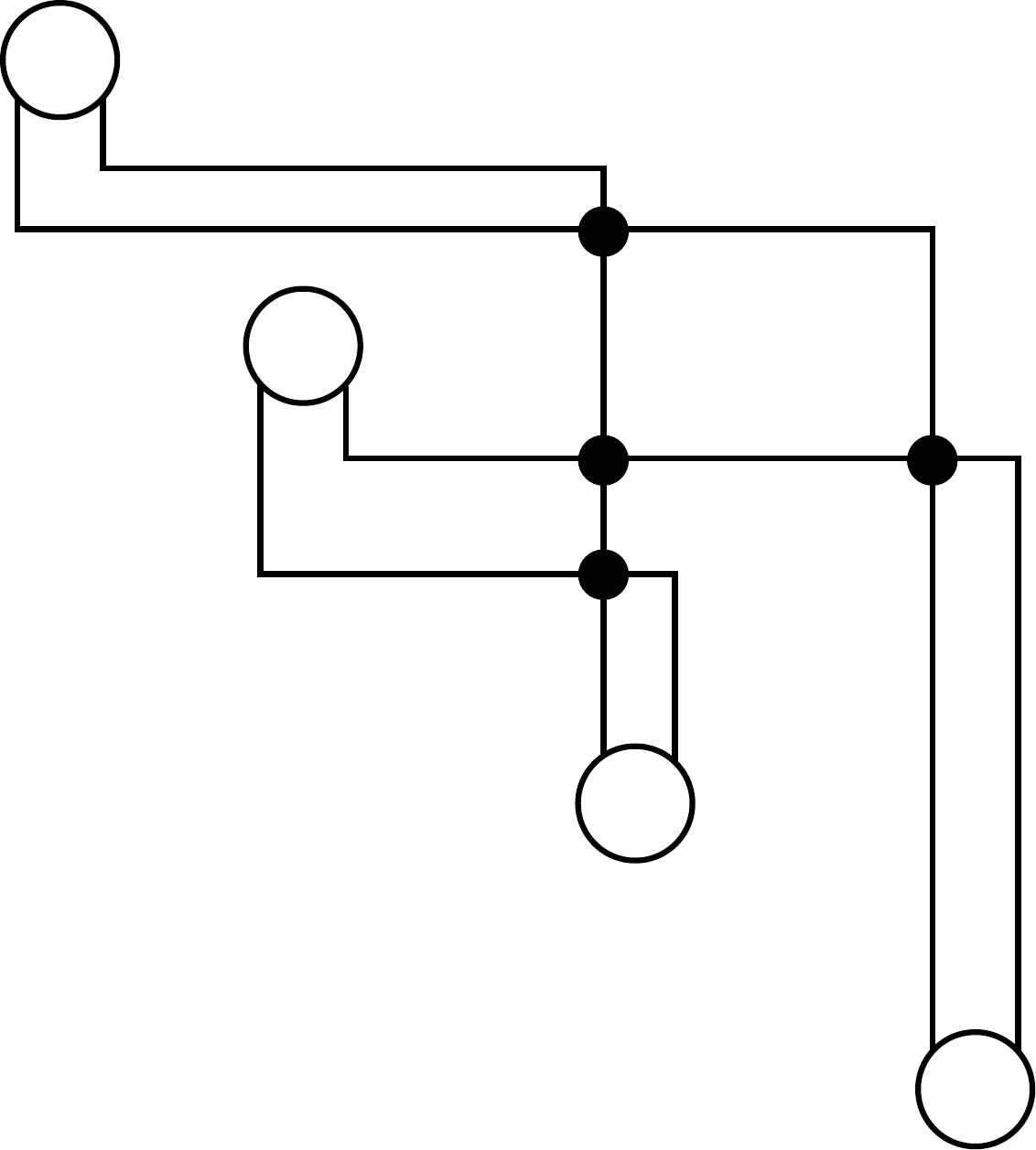}}
\put(5.45,0.25){\large\bf 1\ \ \small$(X^2)$}
\put(3.5,1.9){\large\bf 2\ \ \small$(X^2)$}
\put(1.6,4.5){\large\bf 3\ \ \small$(\Delta^2)$}
\put(0.15,6.2){\large\bf 4\ \ \small $(\Delta^2)$}
\end{picture}\qquad
\setlength{\unitlength}{0.8truecm}
\begin{picture}(6.2,6.7)
\put(0,0){\img{5.0}{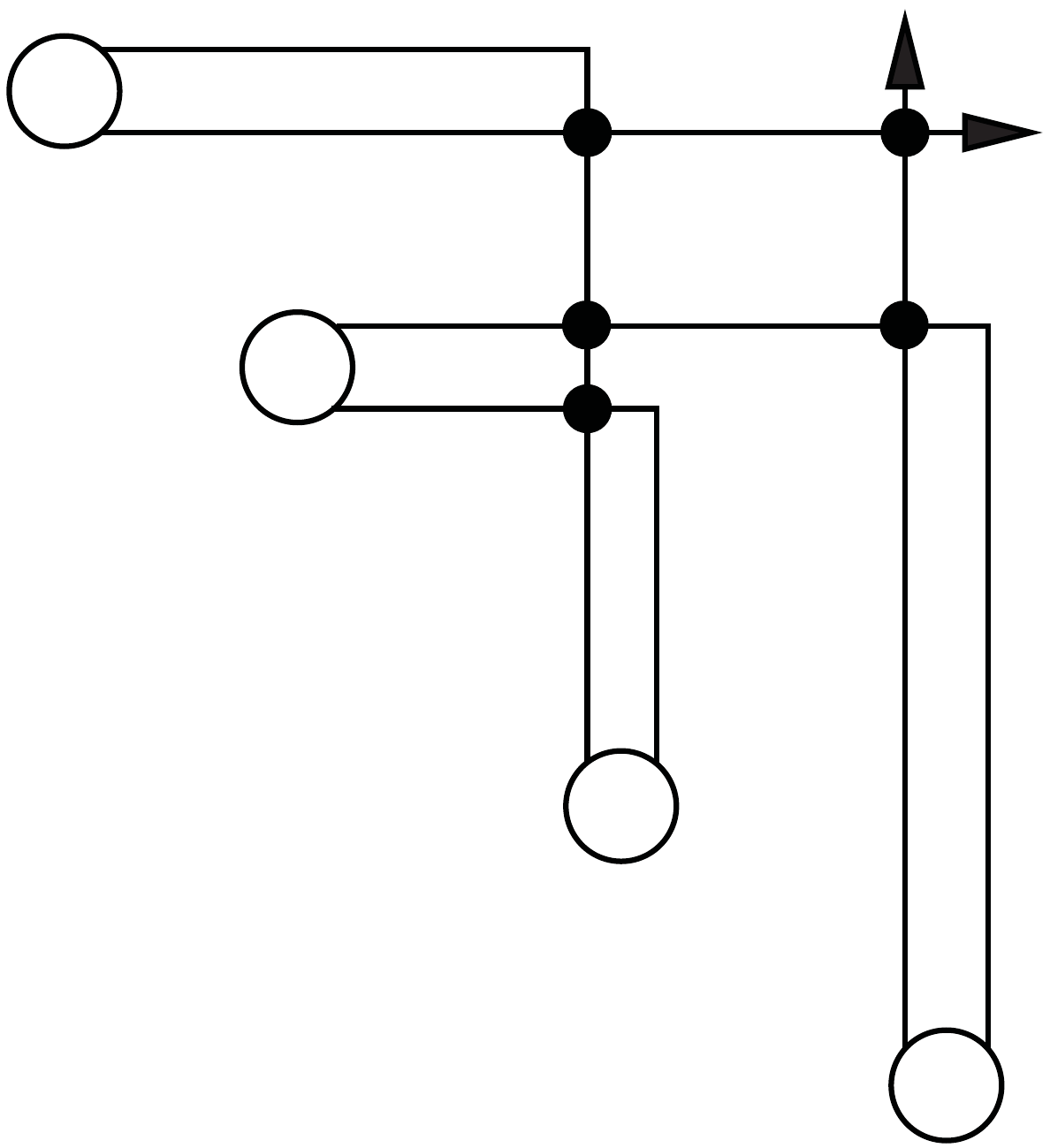}}
\put(5.45,0.25){\large\bf 1\ \ \small$(X^2)$}
\put(3.5,1.9){\large\bf 2\ \ \small$(X^2)$}
\put(1.6,4.5){\large\bf 3\ \ \small$(\Delta^2)$}
\put(0.15,6.2){\large\bf 4\ \ \small $(\Delta^2)$}
\end{picture}}
\end{center}}

\caption{\label{ediag-fig}{\small
\emph{Embedded diagrams} formed with two gates of type $X^2$ and two gates of
type~$\Delta^2$: (\emph{left}) a diagram with 4 crossings and $a=b=0$;
(\emph{right}) a diagram with 5 crossings and $a=b=1$.}}

\end{figure}
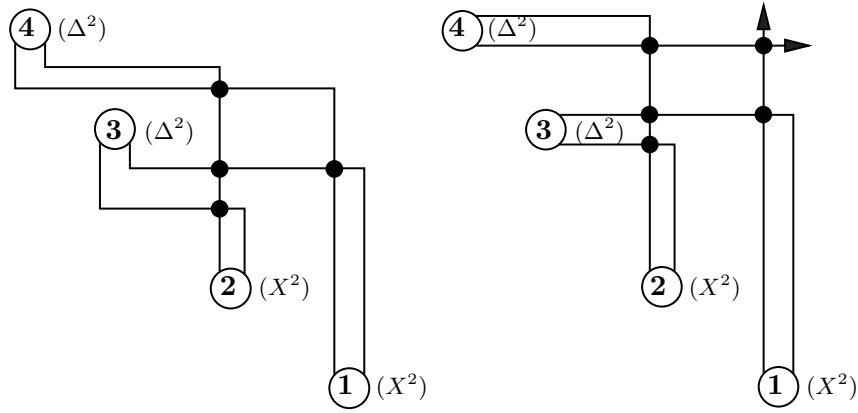


In    the     case    of    the   reduction     of     expressions  of
type~$X^{r_1}\De^{s_1}\cdots    X^{r_n}\De^{s_n}$,      M\'endez   and
Rodriguez~\cite{MeRo08} developed an interpretation similar to ours, in
terms of crossings.
This theme is further explored by Mansour, Schork, and Severini~\cite{MaScSe07} 
in the context of
``Wick's formula''. We also refer to the early work of 
 Katriel, Kibler, Solomon, and Duchamp~\cite{KaDu95,KaKi92,KaSo91}
and to the subsequent studies~\cite{Katriel00,Katriel02,MeRo08b,Schork03} for results
relative to
the $q$-Stirling and Bell numbers. 

The  interest of previous developments,  where the $\Delta$ difference
operator replaces the usual derivative~$D$ is that they systematically
lead to a natural  class  of $q$-analogues, which   may or may not  be
classical. The case of Stirling  numbers ($(X\De)$ and generalizations)
being  well covered  in the  literature,  we limit  ourselves here  to
examine two cases: $(X+\De)^n$, which leads  to a $q$-analogue of the
involution   numbers (\S\ref{baslin-subsec} above);   $(X^2+\De^2)^n$,
which   is   related  to  alternating permutations (\S\ref{zig-subsec}
above).  Our     brief   treatment will  be   along    the  lines of
Section~\ref{cf-sec} dedicated   to binomial forms  $(X^a+D^b)^n$,  with
emphasis placed on ``constant terms'' and continued fraction aspects.
We state (cf Proposition~\ref{lag-prop}):

\begin{proposition} \label{invinv-prop}
Consider the constant term of the normal form of $(X+\De)^n$,
\[
I_n(q):=\ct \nor \left[ (X+\De)^n\right].
\]
Its ordinary generating function satisfies
\begin{equation}\label{invcf}
\sum_{n\ge0} I_n(q)z^n=
\cfrac{1}{1-\cfrac{[1]_q\cdot z^2}{1-\cfrac{[2]_q\cdot z^2}{1-\cfrac{[3]_q\cdot z^2}{\ddots}}}}\,.
\end{equation}
Furthermore, there exists an explicit form for~$I_n(q)$,
\begin{equation}\label{invcf2}
I_n(q)=\frac{1}{(1-q)^n}\sum_{k=-n}^n (-1)^k q^{k(k-1)/2}\binom{2n}{n+k}.
\end{equation}
\end{proposition}
\begin{proof} 
First, the embedded diagrams assocated with the reduction of powers of $(X+\De)$ are
similar to the ones for permutations  (but the distribution of integer labels differs):
see Figure~\ref{invinv-fig} and compare with Figure~\ref{perminv-fig}.
Under this form, it is easily recognized that the embedded diagrams are isomorphic to the arch 
representations of involutions, themselves in bijective correspondence with ``Hermite  histories'';
that is, weighted Dyck paths where a descent from altitude~$j$ has multiplicity~$j$.

\begin{figure}\small
\begin{center}
\begin{tabular}{cc}
\begin{tabular}{c}
\setlength{\unitlength}{1truecm}\begin{picture}(4.5,4.5)
\put(0,0){\img{4.5}{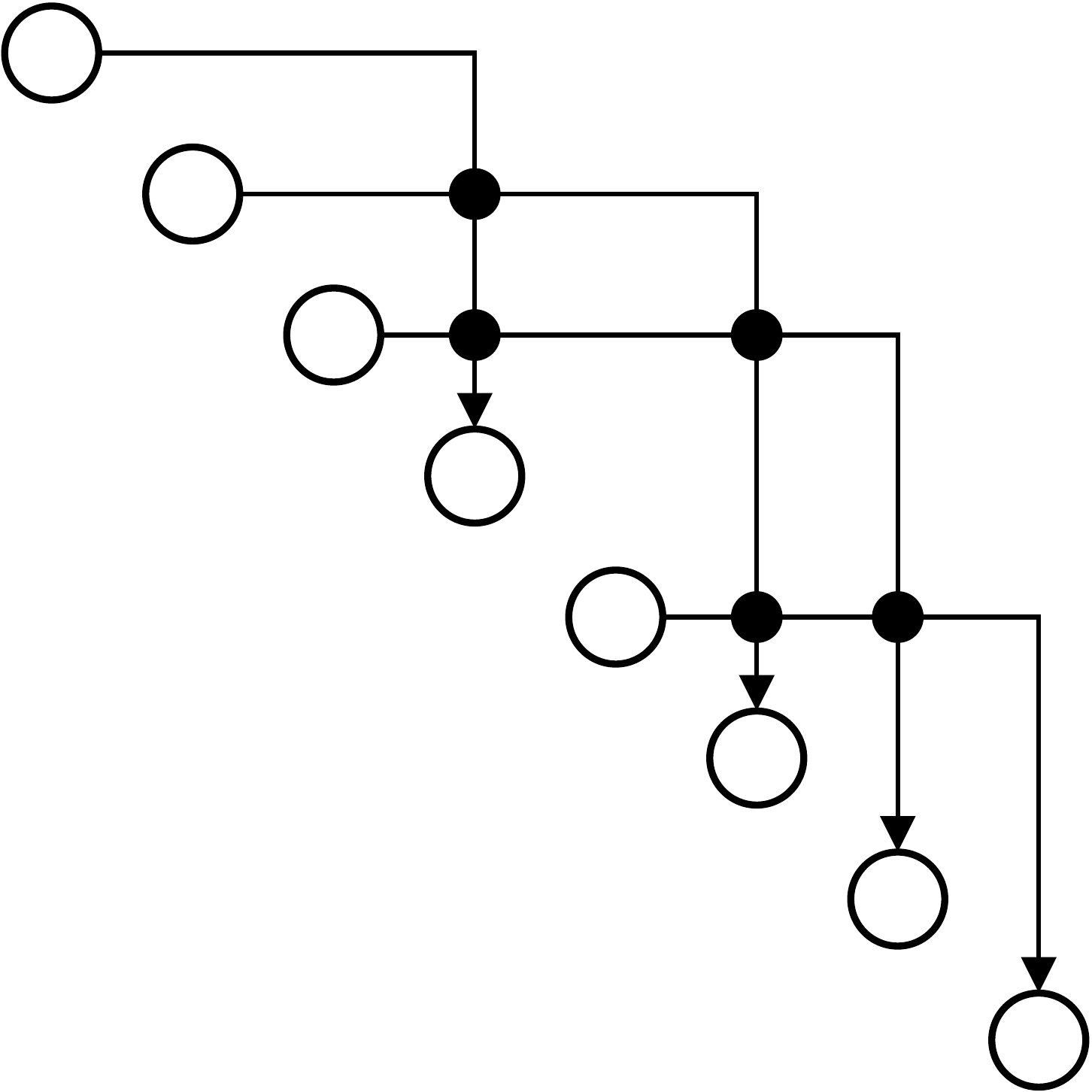}}
\put(0.1,4.2){$\bf 8$}
\put(0.7,3.6){$\bf 7$}
\put(1.3,3.05){$\bf 6$}
\put(1.88,2.43){$\bf 5$}
\put(2.43,1.88){$\bf 4$}
\put(3.05,1.3){$\bf 3$}
\put(3.6,0.7){$\bf 2$}
\put(4.2,0.15){$\bf 1$}
\end{picture}
\end{tabular}
&
\begin{tabular}{c}\def\sp{\hspace*{0.51truecm}}
\setlength{\unitlength}{1truecm}
\begin{picture}(6.0,2.8)
\put(0,0){\img{6}{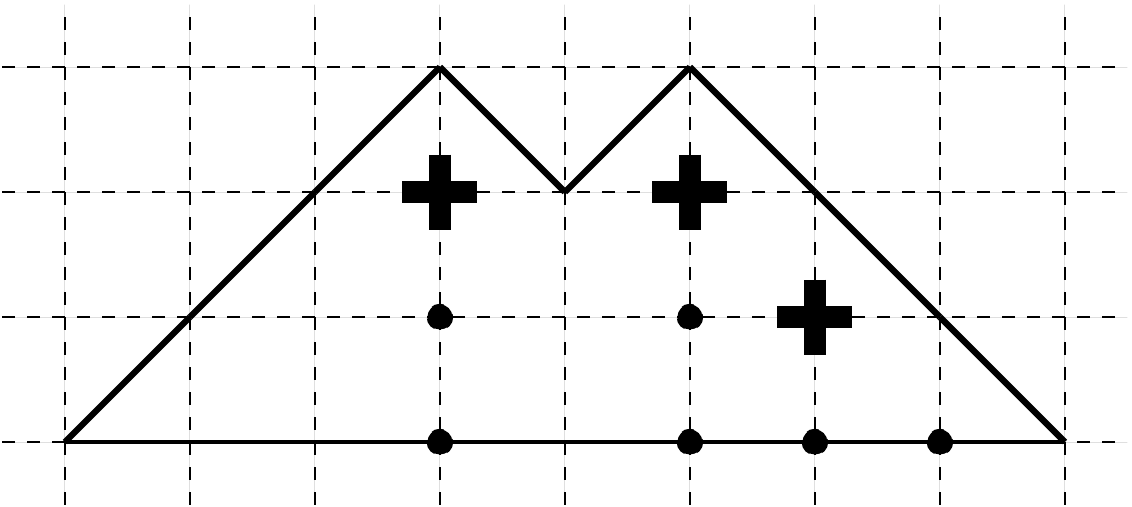}}
\put(0.5,0.0){\bf 1\sp{}2\sp{}3\sp{}4\sp{}5\sp{}6\sp{}7\sp{}8}
\end{picture}
\end{tabular}
\end{tabular}
\\
\img{4.0}{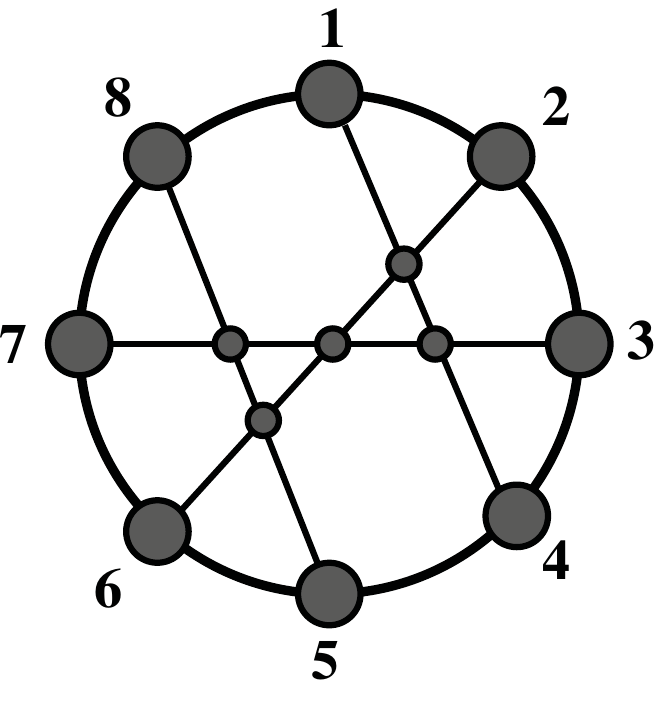}
\end{center}

\caption{\label{invinv-fig}\small
Three representations of the involution
\[
\sigma=\left\{\, (1~4),~~ (2~6),~~ (3~7),~~ (5~8)\,\right\}.\]
From top to bottom: $(a)$~the embedded diagram; $(b)$ the associated Hermite history; $(c)$~
the  chord diagram. The number of crossings equals~5.}
\end{figure}

To see the bijective correspondence\footnote{%
This correspondence is  classical 
and due to Fran{\c c}on and Viennot (see, e.g., \cite{Flajolet80b,FrVi79}
or the accounts in the books~\cite[p.~333]{FlSe09} and~\cite[\S5.2]{GoJa83}).}
 with Hermite histories, start from the embedded diagram
(Figure~\ref{invinv-fig}~$(a)$).
Scan the values from 1 to~$n=2\nu$ and associate an ascent to a value that is smaller
in its cycle, a descent otherwise. This gives rise to a Dyck path exemplified by
Figure~\ref{invinv-fig}.$(b)$. When a descent from altitude~$k$ is produced, say 
at time~$\tau$, there are
$k$~possible choices for the $2$--cycle that could be closed,
the possibilities being enumerated by $[k]_q$. Number conventionally these possibilities
according to ``age'' (oldest first); that is, the open cycle with largest element is numbered~0
the second oldest is numbered~1, and so on. The rank of the possibility chosen 
appears to be equal to the number of crossings that the cycle ending at~$\tau$ has 
with its preceding cycles. Thus, Dyck paths where any ascent is weighted by~1 and 
a descent from altitude~$k$ is weighted by~$[k]_q$ are in bijective correspondence 
with involutions, where each crossing is weighted by~$q$.
The
expansion~\eqref{invcf}
now results from the basic theorem of continued fraction combinatorics~\cite[Th.~1]{Flajolet80b}.

The explicit form~\eqref{invcf2} is none other than the celebrated Touchard--Riordan
formula~\cite{Read79,Riordan75,Touchard52} in the enumeration of chord crossings: 
see Figure~\ref{invinv-fig}~$(c)$.
\end{proof}

Similarly, we have: 


\begin{proposition} \label{invinv2-prop}
The constant terms $I_n^{(2)}(q)$ of the normal form of $(X^2+\De^2)^n$,
satisfy
\begin{equation}\label{invcf22}
\sum_{n\ge0} I_n^{(2)}(q)z^{2n}=
\cfrac{1}{1-\cfrac{[1]_q\,[2]_q\cdot z^2}{1-\cfrac{[3]_q\,[4]_q\cdot z^2}{1-\cfrac{[5]_q\,[6]_q\cdot z^2}{\ddots}}}}.
\end{equation}
\end{proposition}
This is a $q$-generalization   of the continued fraction attached  to
$1/\sqrt{\cos z}$,  for which the  authors have not found  an explicit
form---perhaps  some       of      the   methods        developed   by
Josuat-Verg\`es~\cite{Josuat09b,Josuat09c,Josuat10} could be relevant.\footnote{In a recent paper H. Shing and J. Zeng have given an explicit formula for $I_n^{(2)}$ in terms of a double sum, see Thm.~12 in~\cite{ShZe10}.}
Also, it would seem interesting to elicit possible connections
with $q$-analogs of expansions of trigonometric functions,
as considered by Prodinger~\cite{Prodinger08,Prodinger10}. At any rate,   we can
observe   that  the  continued   fraction  expansions     provided   by
Propositions~\ref{invinv-prop}  and~\ref{invinv2-prop} are not devoid
of content:   they can in  particular  be  employed to  derive  useful
asymptotic  information,      as    is  exemplified       by   several
studies~\cite{FlNo00,FlPuVu86,JaLoMa10,Louchard87b}.

\section{\bf Multivariate schemes.} \label{mult-sec}
The principles underlying the construction of circuits by means of gates
can be extended painlessly to certain multivariate calculi and we provide here 
brief indications to that effect. Algebraically, we now consider a family of operators
$A_1,\ldots,A_r,B_1,\ldots,B_r$, satisfying the partial commutation
relations (expressed in terms of the Lie bracket $[ {}\cdot , \cdot {}]$):
\begin{equation}\label{multi0}
[A_j,B_j]=1, ~~ (j=1,\ldots r);
\quad
[A_j,A_k]=[A_j,B_k]=[B_j,B_k]=0 ~~(j\not=k);
\end{equation}
A faithful model is  that of (multivariate) differential algebra where
we interpret the operators as acting on functions $f(x_1,\ldots,x_r)$,
take   $A_j$   to     be    the $j$th     partial    derivative,  $A_j
:=\frac{\partial}{\partial x_j}$,  and   $B_j$  to  be   multiplication
by~$x_j$,  that is,  $B_j f:=x_j  f$. We  shall adopt  this suggestive
interpretation   and write  $D_j$   instead  of~$A_j$ and~$X_j$ instead
of~$B_j$. (In concrete examples, we may also name variables and use,
for instance, in the case of variables $\{x,y\}$,
the notations $\partial_x,\partial_y,X,Y$.) Obviously, any polynomial in the operators
$X_j,D_j$
has a normal form in which all the $X$s precede all the $D$s.
Furthermore, we may freely adopt the additional convention
that variables obey a standard ordering $x_1\prec x_2\prec x_3\prec \cdots$,
so that $X_j$ will be systematically written before $X_k$ if $j<k$, 
and similarly for $D_j$ and $D_k$.

The idea is now simply to construct \emph{decorated gates}, which are gates 
as considered before, with the additional characteristics
that each incoming and each outgoing vertex
is tagged with a variable name (or its index). We shall also agree that the tags,
in left to right order,  follow the
standard ordering of variables.
Thus for instance, the gate associated with $X_1X_2^4D_1^2D_2$
has one inner node connected to three ingoing edges and five outgoing
edges:
\[
X_1X_2^4D_1^2D_2~:\qquad\quad
\hbox{\Img{3.3}{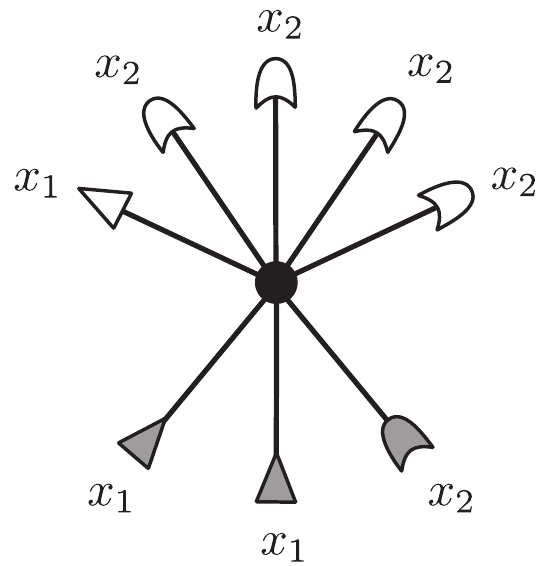}}
\]
%
(A tag $x_j$ on an input signifies that derivation is to be effected with
respect to the corresponding variable, i.e., $\partial_{x_j}$, also abbreviated as
$D_j$.)
To define diagrams, the rule is now the following:
\emph{Outputs of a gate can be connected to inputs of another gate
if and only if they have identical variable tags}. 
Equivalently, the edges of  a graph bear various colours (i.e., tags),
and when composing an existing diagram with a  gate, colours of connecting links must match.

We then have an easy generalization of Theorem~\ref{eqp-thm}:
\begin{proposition}\label{multi-prop}
The coefficient of 
\[
X_1^{a_1}\cdots X_r^{a_r}\, D_1^{b_1}\cdots D_r^{b_r}
\]
in the normal form $\nor(\frak{h}^n)$ is the number (total weight) of
(tagged, coloured) graphs built out of gates associated with
the monomials of~$\frak{h}$, that are comprised of~$n$ gates and have 
$a_j$ outputs of type~$x_j$ and $b_j$ inputs of type $x_j$, for 
all $j=1,\ldots,r$.
\end{proposition}

In the physics literature, multivariate normal form corresponding to~\eqref{multi0}
are often referred to as ``multimode''; see, e.g., \cite{AgMe77,Mattuck92,Weinberg95a}. 
Explicit forms are likely
to be quite rare, due to the additional complexity introduced by the 
need to match colours. We content ourselves with a few examples 
that are of (some) combinatorial significance.

\begin{note} \label{ehr-note} \emph{The combinatorics of $x\partial_y+y\partial_x$ and 
the Ehrenfest model.}
Consider the operator
\[
\Gamma:=x\frac{\partial}{\partial y}+y\frac{\partial}{\partial x},
\]
which serves as a simple illustration of the combinatorial approach to partial 
differential operators via diagrams. Here the gates associated with the operator $\Gamma$ are of two sorts:
\[
\hbox{\Img{6.7}{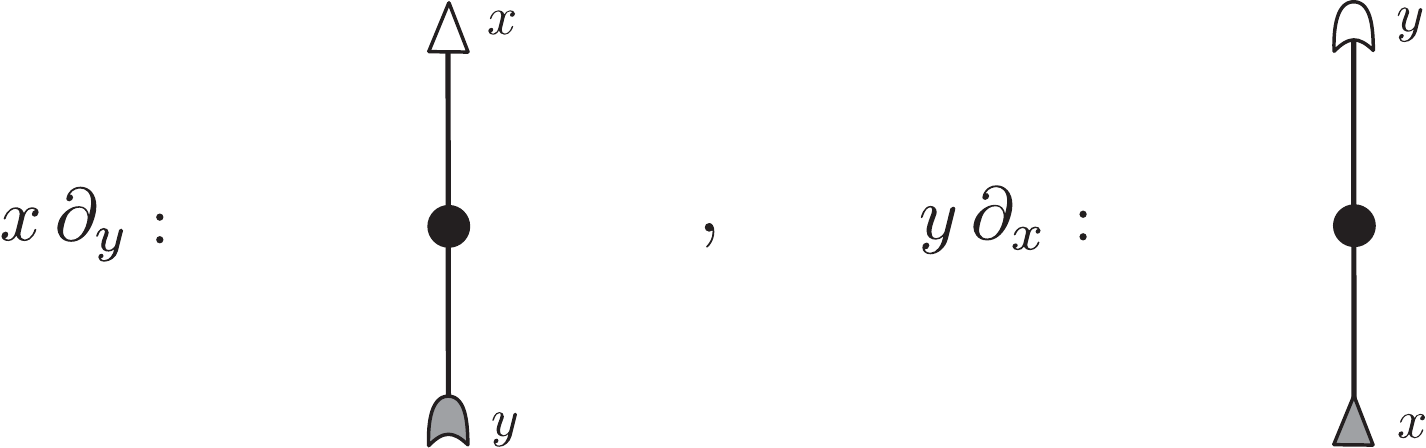}}.
\]
We first state:
\begin{proposition} 
The operator $\Gamma$ satisfies the identity
\begin{equation}\label{ehr0}
e^{z\Gamma}\, f(x,y)=f(x\cosh z+ y\sinh z, x\sinh z+y\cosh z).
\end{equation}
\end{proposition}
\begin{proof}
A single application of $\Gamma$ to $f(x,y)$ picks up
the occurrence of a variable, $x$ or $y$, and either
replaces the~$x$ with a~$y$, or the~$y$ with an $x$.
The diagrams are of the kind constructed in Subsection~\ref{xd-subsec} and relative to
the normal form of~$(XD)$: they consist of line graphs,
which are sequences of edges (drawn vertically,
as in Figure~\ref{xd-fig}). However, edges are now of one of two colours ($x$ or $y$),
with the additional constraints that the colours along each linear path
alternate, as in~$xyxy\ldots$ or $yxyx\ldots$-- see illustration in Figure~\ref{ehrenfest-diag-fig}.
\begin{figure}
\begin{center}
\Img{10}{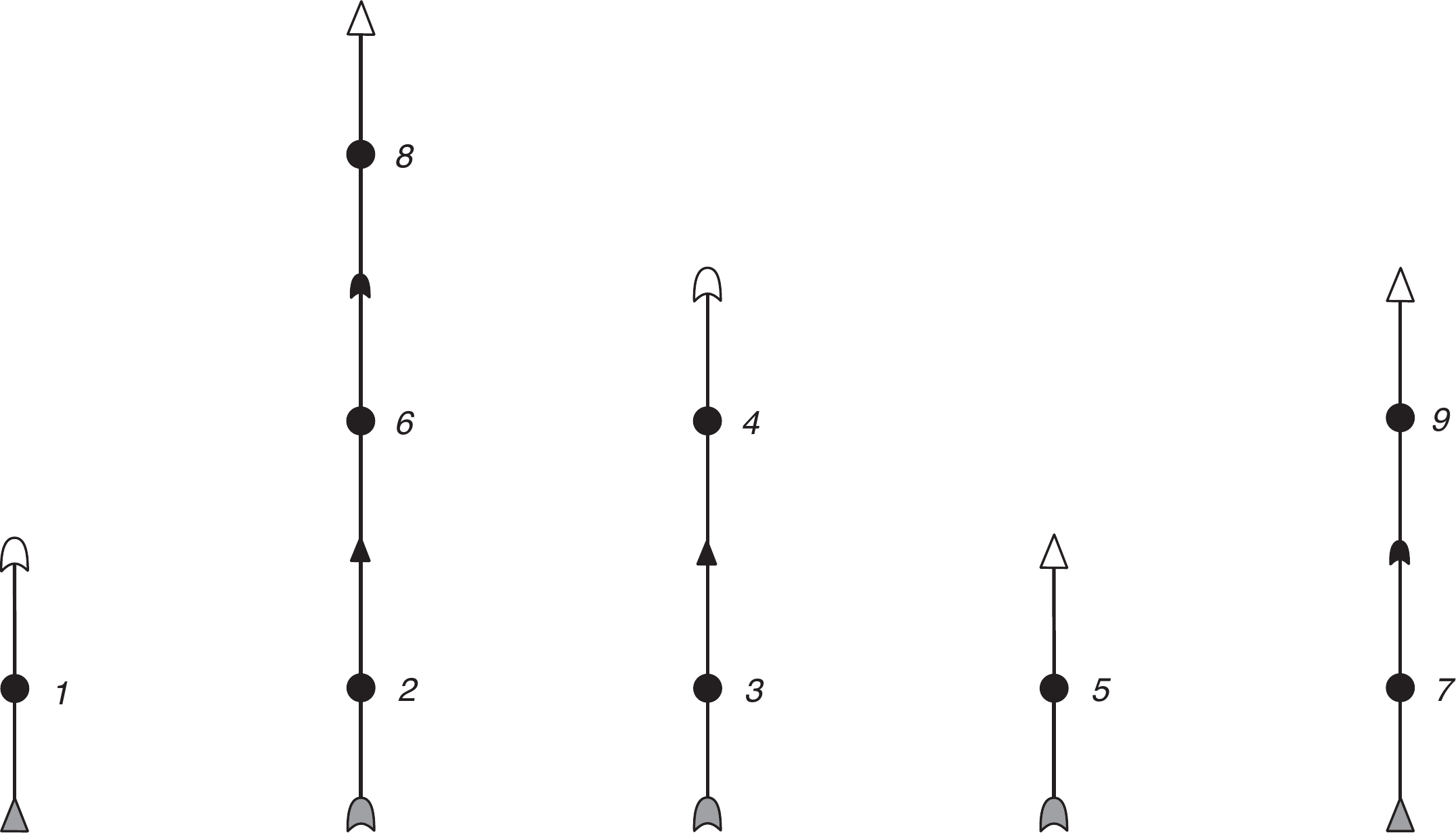}
\end{center}

\caption{\label{ehrenfest-diag-fig}\small
A particular diagram associated with $(x{\partial_y}+y{\partial_x})$.}
\end{figure}

Now comes the combinatorial reasoning that justifies~\eqref{ehr0}.
By linearity, it suffices to consider the action 
of $e^{z\Gamma}$ on a generic monomial.
First, we claim the identities
\[
e^{z\Gamma} \, x= \left(x\cosh z+y \sinh z\right),\qquad
e^{z\Gamma} \, y= \left(y\cosh z+x \sinh z\right).
\]
This corresponds to the fact\footnote{
Notice the classical expansions
\[
\cosh(z)=\sum_{n\equiv 0 \bmod 2} \frac{z^n}{n!}, \qquad
\sinh(z)=\sum_{n\equiv 1 \bmod 2} \frac{z^n}{n!}.
\]
}    that     a connected    diagram    is  a    line   graph   (as in
Subsection~\ref{xd-subsec}), whose input  and output  are tagged by
the same letter  (either $xx$ or  $yy$) if the graph  has an \emph{even} size,
but by different letters (either  $xy$ or $yx$)  if the graph has  an
\emph{odd} size.  The formula
\begin{equation}\label{ehr1}
e^{z\Gamma} \, (x^{a_0}y^{b_0})=\left(x\cosh z+y \sinh z\right)^{a_0}
\left(y\cosh z+x \sinh z\right)^{b_0}
\end{equation}
then results from the general property
that the product of EGFs enumerates all possible distributions of
labels (marked by~$z$) among components.
\end{proof}

From  a combinatorial  point of view,   what has been  done amounts to
enumerating ordered partitions into $m$  (possibly empty) blocks, with
$x$ and $y$  recording the parity  of  each block.  In the  particular
case  where $b_0=0$, $a_0=m$, and $y=0$,  we obtain the univariate EGF
$\cosh^m z$, which enumerates  ordered partitions whose blocks are all
of even size---this is a well-known (and easy) result.

\begin{figure}\small
\begin{center}
\Img{9}{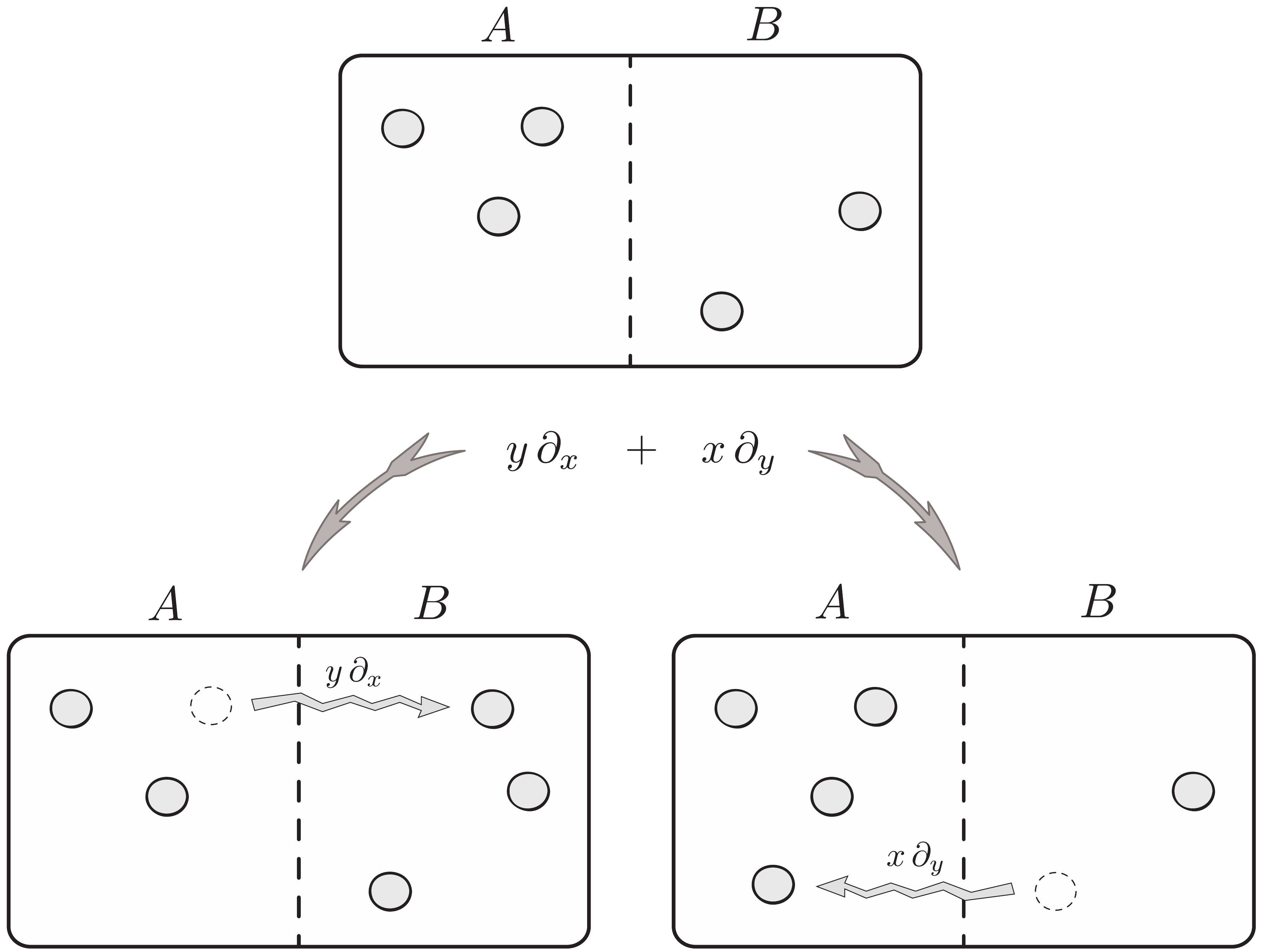}
\end{center}

\caption{\small\label{ehren-fig}
Modelling of the Ehrenfest urn by the 
partial differential operator $x\partial_y+y\partial_x$.}
\end{figure}

\medskip
The analysis above yields as a byproduct a nonstandard
analysis of the classical Ehrenfest model~\cite{EhEh07} (Figure~\ref{ehren-fig}),
which is defined as follows.

\begin{small}
\begin{itemize}
\item[] {\bf Ehrenfest Model.} There are two communicating chambers $A,B$
and~$m$ distinguishable particles (say,
numbered from~1 to~$m$). At any given instant $1,\ldots,n$, a particle
is randomly chosen to change chamber. 
\end{itemize}
\end{small}
The problem consists in determining the probability 
$P_m(n;a_0,a)$ of having, at time~$n$, $a$ particles 
in chamber~$A$ knowing that  there are~$a_0$ particles in that
chamber at time~$0$.
(The Ehrenfest Model is a simple
statistical model of the diffusion of particles or heat in a heterogenous environment.)

Kac~\cite{Kac47} provides a complete solution based on matrix algebra and 
\emph{ordinary} generating functions. Here, we may simply observe that,
by virtue of~\eqref{ehr1}, the probability is
\[
P_m(n;a_0,a)= n! m^{-n} [x^ay^bz^n] 
\left(x\cosh z+y \sinh z\right)^{a_0}
\left(y\cosh z+x \sinh z\right)^{b_0},
\]
where $b_0=m-a_0$ and $b=m-a$. (The factor $m^{-n}$ transforms 
counts of sample paths into probabilities; the factor $n!$ is due
to the fact that we deal with EGFs.) 
It is then a simple matter to 
perform coefficient extraction, by successive binomial expansions.
\begin{proposition}
The transition probabilities of the Ehrenfest model are given by
\[
P_m(n;a_0,a)=\frac{1}{2^m}\sum_{j=0}^m \lambda_{j}^{(a_0,b_0)}\lambda_{m-a}^{(m-j,j)}\left(1-\frac{2j}{m}\right)^n,
\]
where $\lambda_{n}^{(i,j)}:=[z^n](1+z)^i(1-z)^j$.
\end{proposition}
The derivation just given of the solution to the Ehrenfest model  
seems to us as ``conceptual'' as can be. It
has strong similarities with the one given by Goulden and Jackson in~\cite{GoJa86}.
It is further developed in several other papers~\cite{EdKo94,FlDuPu06,FlHu08}
and in the book~\cite{FlSe09}, pp.~118 and~530.
\end{note}

\newcommand{\mini}[1]{\begin{minipage}{7.0truecm}\footnotesize\ \par
#1\\ \ \par\end{minipage}}
\def\p{\partial}
\def\X{{\cal X}}\def\Y{{\cal Y}}

\begin{figure}\small \renewcommand{\arraystretch}{1.6}
\begin{center}
\begin{tabular}{lcl}
\hline\hline
\emph{Name} & \multicolumn{1}{c}{\emph{operator $(\Gamma)$}} & \multicolumn{1}{c}{\emph{model}}\\
\hline
P\'olya & $\ds x^2\p_x+y^2\p_y$ &
\mini{Models propagation of epidemics or genes (P\'olya--Eggenberger). 
Admits separation of variables 
and corresponds to shuffles of $1$-dimensional histories.\\
References:~\cite{FlDuPu06,JoKo77,Mahmoud03,Mahmoud08}.}
\\
\hline
Friedman & $xy\p_x+xy\p_y$ & 
\mini{Friedman's model of ``adverse safety campaign''.
Associated with Eulerian numbers and rises in perms.\\
References:~\cite{Dumont96,FlDuPu06,Friedman49,JoKo77,Mahmoud03,Mahmoud08}.}
\\
\hline
Ehrenfest & $y\p_x+x\p_y$ & 
\mini{Ehrenfest model of the diffusion of particles. Associated with ordered 
set partitions and parity of blocks.\\
References: Note~\ref{ehr-note} and~\cite{FlDuPu06,FlSe09,JoKo77,Mahmoud08}.}
\\
\hline
Coupon collector & $y\p_x+y\p_y$ &
\mini{Models the classical coupon collector problem.
Associated with surjections, ordered set partitions,
and Sirling${}_2$ numbers.\\
References:~\cite{FlDuPu06,FlSe09,JoKo77,Mahmoud08}.}
\\
\hline
Sampling~I
& $x\p_x+y\p_y$ &
\mini{Sampling with replacement. Admits separation of variables.\\
References:~\cite{FlDuPu06,JoKo77,Mahmoud08}.}\\
\hline
Sampling II
& $\p_x+\p_y$ &
\mini{Sampling without replacement. Admits separation of variables.\\
References:~\cite{FlDuPu06,JoKo77,Mahmoud08}.}
\\
\hline
Records
& $x^2\p_x+xy\p_y$ &
\mini{Models records in permutations and is associated with Stirling${}_1$ numbers.
Describes the growth of the rightmost branch in a binary increasing tree.\\
References:~\cite{FlDuPu06}.}
\\
\hline
Bimodal-2& $y^2\p_x+x^2\p_y$ & 
\mini{Models a (binary) chain reaction with two types of particles: $\X\mapsto \Y^2, \Y\mapsto \X^2$.
Corresponds to parity of levels   in  increasing binary trees and  
models fringe-balanced trees. 
Is
solved     in     terms      of     Dixonian   elliptic      functions
$\operatorname{sm},\operatorname{cm}$.
\\
References:~\cite{CoFl06,FlDuPu06,FlGaPe05,PaPr98b}.
}
\\
\hline
Bimodal-3& $y^3\p_x+x^3\p_y$ & 
\mini{Models a (ternary) chain reaction with two types of particles: $\X\mapsto \Y^3, \Y\mapsto \X^3$.
Corresponds to parity of levels   in  increasing ternary trees and   is
solved     in     terms      of   lemniscatic  elliptic      functions
$\operatorname{sl},\operatorname{cl}$.
\\
References:~\cite{FlDuPu06,FlGaPe05}.
}
\\ \hline\hline
\end{tabular}
\end{center}

\caption{\label{urn2-fig}\small
A list of some first-order bivariate operators~\eqref{polya3} 
associated with solvable urn models and 
explicit normal forms, together with the main characteristics
of the models, after~\cite{FlDuPu06}.}

\end{figure}

We next discuss a class of bivariate first-order operators, 
given by Equation~\eqref{polya3} below, 
for which the normal form problem is 
solvable in finite terms (Proposition~\ref{urn2-prop}).
Following~\cite{FlDuPu06},
Figure~\ref{urn2-fig} lists some representative operators
for which the connection with urn processes detailed in the Note~\ref{urn2-note}
provides explicit forms.
There are also two interesting papers~\cite{Dumont86,Dumont96} published by Dumont in 
\emph{Seminaire Lotharingien de Combinatoire} in 1986 and 1996
that  deal with  combinatorial  aspects of   powers  of special linear
partial differential   operators: the later  one~\cite{Dumont96} is in
particular  based on  Chen   grammars that  can    be regarded  as  a
non-probabilistic version of urn processes.

\begin{note} \emph{P\'olya urn models with two colours.} \label{urn2-note}
Here  is a specification  of  this model,  also sometimes known  
as P\'olya--Eggenberger model~\cite{JoKo77,Mahmoud08}.

\begin{small}
\begin{itemize}
\item[] {\bf P\'olya Urn.}
An urn may  contain balls of different colours.  A fixed  set of replacement
rules is given (one for each colour). At any  discrete instant, a ball
is  chosen  uniformly at random,   its colour  is   inspected, and the
corresponding replacement rule is applied.
\end{itemize}
\end{small}%
The modeling capability of the P\'olya Urn process is stupendous and 
the literature on the subject, mostly probabilistic, is immense.
Here, we follow the purely  combinatorial line of
Flajolet \emph{et al.}~\cite{FlDuPu06,FlGaPe05},
and only consider models with two colours (say, $\X$ and $\Y$).
A model is then determined by a $2\times2$ matrix with integer entries:
\[
\mathbf{M}= \left(\begin{array}{cc} \alpha&\beta\\ \gamma& \delta \end{array}\right),
\qquad \alpha,\beta,\gamma,\delta\in\Z.
\]
At any  instant, if a ball  of the first colour is  drawn,  then it is
placed back into the urn together with $\alpha$ balls of the first colour and
$\beta$ balls of the  second colour; similarly,  when a ball  of the second
colour is drawn, with $\gamma$ balls of the  first colour and  $\delta$ balls of the
second colour. 
A figurative rendition of this replacement rule that we may use occasionally is
\[
\X \mapsto \X^\alpha \Y^\beta; 
\qquad
\Y\mapsto \X^\gamma \Y^\delta.
\]
(Some authors prefer an additive notation, such as $\X\mapsto \alpha\X+\beta \Y$,
which is evocative of chemical reactions.) For instance, the Ehrenfest model
of Note~\ref{ehr-note}
is rendered by the following matrix
and rule:
\[
\mathbf{E}=\left(\begin{array}{cc} -1&1\\ 1& -1\end{array}\right),
\qquad
\left(\X\mapsto \X^{-1}\Y; \quad \Y\mapsto \X\Y^{-1}\right).
\]
(Negative diagonal entries $\alpha,\delta$ mean that balls are taken out
of  the  urn, rather  than added  to  it).  We 
henceforth restrict  attention to
balanced urns, which are such that there exists a number~$s$, called the balance,
such that
\begin{equation}\label{bal}
s=\alpha+\beta=\gamma+\delta.
\end{equation}

Given an urn initialized  with $a_0$ balls of
the first colour and $b_0$  balls of the second colour,  what is sought\footnote{
	For balanced urns (only), there is complete equivalence between  probabilistic
analysis and  enumeration of histories~\cite{FlDuPu06}.
} is
the multivariate generating function $H(x, y, z)$ (of exponential type),
such that $n![z^n x^a y^b]\, H(x, y, z)$ is the  number of possible evolutions
(also known as ``histories'' or ``sample paths'')
of the urn leading  at time~$n$ to   an urn with colour  composition $(a,
b)$. 
Let $t_0=a_0+b_0$ be the initial size of the urn. 
For $s\ge1$, the total number of evolutions  is clearly 
$t_0(t_0+s)\cdots (t_0+(n-1)s)$,  so that 
$H(1, 1, z)  = (1-sz)^{-t_0}$.

Let the formal monomial $\frak{m}=x^ay^b$ represent a particular urn comprised
of $a$ balls of type $\X$ and $b$ balls of type $\Y$.
From our introductory discussion of the combinatorics of derivatives in Note~\ref{defder-note},
we can see that the effect of
all the one-step evolutions of the urn are described by the 
linear \emph{partial differential operator}
\begin{equation}\label{polya3}
\Gamma= x^{\alpha+1}y^\beta \frac{\partial}{\partial x}
+x^\gamma y^{\delta+1} \frac{\partial}{\partial y}, \qquad \alpha+\beta=\gamma+\delta.
\end{equation}
applied to~$\frak{m}$. In the same way, $\Gamma^2$ generates all two-step evolutions, and so on.
We then have an easy observation~\cite{FlGaPe05}:

\begin{proposition}
The EGF of histories of a balanced urn 
with two colours and initial composition $(a_0,b_0)$ satisfies
\[
H(x,y,z)=e^{z\Gamma} \, (x^{a_0}y^{b_0}),
\]
where $\Gamma$ is the operator specified in~\eqref{polya3}.
\end{proposition}
In other words, the solution to the urn problem is equivalent to the reduction to normal
form of the powers $\Gamma^n$ (equivalently, the exponential $e^{z\Gamma}$) of the associated
first-order partial differential operator~$\Gamma$.
It is noteworthy that, under the  conditions (two-coloured and balanced),
this normal ordering problem is solvable. The treatment consists 
in associating an \emph{ordinary} differential system,
\begin{equation}\label{XY}
\Sigma~:\qquad \left\{\frac{d}{dt}X(t)=X(t)^{\alpha+1}Y(t)^\beta, \quad
\frac{d}{dt}Y(t)=X(t)^\gamma Y(t)^{\delta+1}\right\},
\end{equation}
which admits a first integral that is a polynomial
(namely, $Y^p-X^p=1$, with $p=\gamma-\alpha$)
and which can be  explicitly related to
the multivariate EGF $H(x,y,z)$. 
We quote from~\cite[Th.~2]{FlDuPu06}.

\begin{proposition} \label{urn2-prop}
Two-coloured balanced urns and  the associated normal form problem for
the    operator  $\Gamma$   in~\eqref{polya3}    are     solvable   by
quadrature\footnote{It can be seen that the variable~$y$ is redundant,
so that the substitution $y\mapsto 1$ entails no loss of generality.}:
in the  case\footnote{    Among negative values,
only~$\alpha=-1$   is   of direct
relevance to the normal ordering problem  in its standard form, since,
in   the  present  study,   negative  powers   of~$X$ are   excluded.
The case~$\alpha\ge0$ seems to give rise to similar developments 
(B. Morcrette, work in progress, 
July 2010).}
$\alpha<0$, one has
\[
\left[ e^{z\Gamma}\, (x^{a_0}y^{b_0})\right]_{y=1}=
\Delta^{t_0}S\left(-\alpha z\Delta^s+J(x^{-\alpha}\Delta^\alpha)\right)^{-\frac{a_0}{\alpha}}
C\left(-\alpha z\Delta^{s}+J(x^{-\alpha}\Delta^\alpha)\right)^{-\frac{b_0}{\delta}}.
\]
The notations are
\[
s=\alpha+\beta, \quad p=\beta-\alpha;,\quad t_0=a_0+b_0,
\]
as well as $\Delta\equiv \Delta(x)=(1-x^p)^{1/p}$, and 
$\ds 
J(u):=\int_0^u \frac{d\zeta}{(1+\zeta^{-p/\alpha})^{p/\beta}}$.
The function $S(z)$ is defined as the inverse of $J(u)$, namely, $S(J(u))=J(S(u))=u$;
the function $C(z)$ is given by $C(z)=(1+S(z)^{-p/\alpha})^s$.
\end{proposition}

\noindent
Though the general formula looks rather formidable, great simplifications occur in many models of
practical interest and the expressions lend themselves to precise asymptotic analysis,
as detailed in~\cite{FlDuPu06}.
\end{note}

\renewcommand{\mini}[1]{\begin{minipage}{4.8truecm}\footnotesize\ \par
#1\\ \ \par\end{minipage}}

\begin{figure}\small \renewcommand{\arraystretch}{1.6}
\begin{center}
\begin{tabular}{lcl}
\hline\hline
\emph{Name} & \multicolumn{1}{c}{\emph{operator $(\Gamma)$}} & \multicolumn{1}{c}{\emph{model}}\\
\hline
P\'olya & $\ds w^2\p_w+x^2\p_x+y^2\p_y$ &
\mini{Models propagation of  genes. 
Admits separation of variables 
and corresponds to shuffles of $1$-dimensional histories.\\
References:~\cite{FlDuPu06,JoKo77,Mahmoud08}.}
\\
\hline
C-Ehrenfest & $x\p_w+y\p_x+w\p_y$ & 
\mini{Models particles in a cycle of three chambers. Associated with ordered 
set partitions and congruence properties of block sizes.\\
References:~\cite{FlDuPu06}.}
\\
\hline
S-Ehrenfest & $(x+y)\p_w+(y+w)\p_x+(w+x)\p_y$ & 
\mini{Models particles in a symmetric set of three chambers.\\
References:~\cite{FlDuPu06}.}
\\
\hline
Coupon & $x\p_w+y\p_x+y\p_y$ &
\mini{Models multiple coupon collections.
Associated with surjections, ordered set partitions,
and generalized Sirling${}_2$ numbers.\\
References:~\cite{FlDuPu06,FlSe09}.}
\\
\hline
Pelican & $xy\p_w+yw\p_x+wx\p_y$ & 
\mini{Models the pelican sacrifice model.
Is
solved     in     terms      of  Jacobian  elliptic      functions
$\operatorname{sn},\operatorname{cn}$.
\\
References:~\cite{Dumont86,FlDuPu06,Viennot80}.
}
\\ \hline\hline
\end{tabular}
\end{center}

\caption{\label{urn3-fig}\small
Some first-order trivariate operators, in variables $w,x,y$,
that are associated with solvable urn models and 
explicit normal forms, together with the main characteristics
of the models, after~\cite{FlDuPu06}.}

\end{figure}

\begin{note} \emph{First-order differential operators, multitype trees, and the method of characteristics.}\label{char-note}
The object of interest here is the operator
\def\x{{\mathbf{x}}}\def\T{\cal T}
\[
\Lambda=\sum_{j=1}^m \phi_j(X_1,\ldots,X_m)D_j+\rho(X_1,\ldots,X_m).
\]
What   we do  here  is  to  extend  the  univariate  framework  of the
semilinear case  of  Section~\ref{semilin-sec} to~$m$  variables.  The
combinatorial model will once more be expressed in terms of increasing
trees. We shall write $\x$ for the vector of variables $(x_1,\ldots,x_m)$.
We state:

\begin{proposition}\label{char0-prop}
The exponential $e^{z\Lambda}$ admits a normal form 
\[
e^{z\Lambda}=\left.e^{R(\x,z)} \exp\left(\sum_{j=1}^m  T_j(\x,z)\,v_j\right)\right|_{
\substack{x_i\mapsto X_i\\v_j\mapsto D_j}},
\]
where the functions $T_j$ and  $R$ satisfy the \emph{ordinary differential system}
\begin{equation}\label{ch0}
\left\{
\begin{array}{cccllc} 
\ds \frac{\partial}{\partial z}T_1(\x,z)&=&\phi_1(x_1+T_1(\x,z),\ldots,x_m+T_m(\x,z)),&&
T_1(\x,0)=0\\
\vdots &\vdots & \vdots && \multicolumn{1}{c}{\vdots}\\
\ds \frac{\partial}{\partial z}T_m(\x,z)&=&\phi_m(x_1+T_1(\x,z),\ldots,x_m+T_m(\x,z),)&&
T_m(\x,0)=0,
\end{array}\right.
\end{equation}
and 
\begin{equation}\label{ch1}
R(\x,z)=\int_0^z \rho(x_1+T_1(\x,w),\ldots,x_m+T_m(\x,w))\, dw.
\end{equation}
\end{proposition}
\begin{proof}
Combinatorially, we are now dealing with a collection $\T_1,\ldots,T_m$ of trees of~$m$
different types. The multivariate extension of the basic framework of gates,
Proposition~\ref{multi-prop} shows that the $\T_j$ must  represent 
a collection of types of increasing trees satisfying the equations (cf Section~\ref{genincr-subsec})
\[
\T_j={\cal Z}^{\Box}\star \phi_j(x_1+\T_1,\ldots,x_m+\T_m),
\]
for~$j=1,\ldots,m$, and $\cal R$ is a ``planted'' variety of trees corresponding to  (cf Section~\ref{plant-subsec})
\[
\cal R={\cal Z}^{\Box} \star \rho(x_1+\T_1,\ldots,x_m+\T_m).
\]
(Here, the variables $x_j$ serve to mark the types of leaves.)
This
 gives rise to a set of integral equations,
\[
 \frac{\partial}{\partial z}T_j(\x,z)=
\int_0^z \phi_1(x_1+T_1(\x,w),\ldots,x_m+T_m(\x,w))\, dw
\]
which are clearly equivalent to the differential relations~\eqref{ch0}, 
as well as to the expression of $R$ in~\eqref{ch1}.
\end{proof}

In terms of \emph{partial} differential equations, Proposition~\ref{char0-prop} 
concerns the solution~$F(\x,z)$
of
\begin{equation}\label{par0}
\frac{\partial}{\partial z} F(\x,z)=\sum_{j=1}^m \phi_j(x_1,\ldots,x_m)
\frac{\partial}{\partial x_j} F(\x,z)+\rho(x_1,\ldots,x_m) F(\x,z).
\end{equation}
What is noticeable is the fact that there is a reduction to a collection
of \emph{ordinary} differential equations, of which the Ehrenfest urn of Note~\ref{ehr-note}
provides an explicitly solvable case---the equations are in general non-linear, though, and
integrability is far from being granted. This reduction is classically
achieved by the well-known method of characteristics~\cite[\S1.15]{Taylor96}. We then have:
\emph{the method of characteristics, in the special case~\eqref{par0} at least, admits
a combinatorial model, which is that of multitype increasing trees.}

In the particular case where $\rho=0$ and the $\phi_j$ are homogeneous\footnote{In 
the inhomogeneous case, the probabilistic connection with urn
processes is lost and we are plainly enumerating combinatorial tree families,
or, what amounts to the same by Subsection~\ref{rook-subsec}, certain families of
weighted lattice paths.} polynomials
of total degree~$s$, the solution~$F(\x,z)$ provides the enumeration of 
histories of a balanced urn model with $m$ colours and balance~$s$. This generalizes the
results of Note~\ref{urn2-note} relative to two-colour models. Our equivalence with
the combinatorics of increasing trees can then be seen as a parallel to the reduction of
urn models to multitype branching processes~\cite[\S9]{AtNe72}.
\end{note}

Explicit iteration formulae (normal forms) for first-order operators in more 
than two variables are somewhat rare: see Figure~\ref{urn3-fig} for some
solvable cases associated with balanced urn models, drawn from~\cite{FlDuPu06},
and the papers by Dumont~\cite{Dumont86,Dumont96} for some additional
examples related to classical combinatorics. Even less is explicit in the case of higher order
operators. In the quantum physics literature,
the paper of Agrawal and Mehta~\cite{AgMe77} shows what can be done with the exponential
of multivariate quadratic forms in~$X_j,D_j$ and it would be of obvious interest to
elicit its combinatorial content. Similar comments apply to the works of 
Heffner and Louisell~\cite{HeLo65}, Louisell~\cite[pp.~203--207]{Louisell90},
Marburger~\cite{Marburger66,Marburger67},
and Wilcox~\cite{Wilcox67}, to name a few.

%
%
%

\section{\bf Perspectives}

As argued throughout this paper,  a large number of algebraic identities
associated  with  two  operators  satisfying the  partial  commutation
relation $AB-BA=1$ have a clear combinatorial content, from which many
natural and simple proofs result. Of course, many, if not most, of our
formulae  are   not  new---for  instance,   many  are  known   to  Lie
theorists---but  we  feel  that  the gates-and-diagram  model  has  a
definite explanatory  value  and  it can  serve  as  a  powerful
organizing  principle for families of  polynomial  identities stemming
from  the   relation  $AB-BA=1$   (or,  in  our   favorite  notations,
$DX-XD=1$).

There are  clearly a  great many related  combinatorial works  that we
could not discuss in this already  long paper. Perhaps one of the most
interesting  connections that  we  left unexplored  is  with works  of
Fomin~\cite{Fomin88} and  Stanley~\cite{Stanley88} relative to modular
(graded)  graphs  and   differential  posets---these,  in  particular,
provide  a natural setting  for the  study of  Young tableaux  and the
Robinson-Schensted correspondence. Under the Fomin--Stanley framework,
typically,  what is  given  is a  graded  graph satisfying  additional
properties,  where  two  operators  $U$  (for  ``Up'')  and  $D$  (for
``Down'')     enjoy    the     property    $DU-UD=rI$,     for    some
integer~$r\in\Z_{\ge1}$.  Enumerative  properties that lie  beyond the
purely  algebraic operator  relations studied  here can  be fruitfully
developed within the framework of~\cite{Fomin88,Stanley88}.  (See also
the many  follow-up papers, e.g.,~\cite{Fomin94,Fomin95,Stanley90}.)
In this context,  our work is to be regarded as  an elaboration of the
particular  case  $r=1$, when  the  underlying  graph  is the  set  of
nonnegative integers (cf our Section~\ref{cf-sec}).

The present  paper is one of the  many studies that  signal the
possibility  of  approaching  the  analysis of  certain  combinatorial
processes  by means of  linear operators  bound by  suitable algebraic
relations. An  early instance is  Paterson's \emph{``Cookie Monster''}
model, which is discussed  by Greene and Knuth~\cite[Ch~3]{GrKn90} and
serves to analyse coalescence  phenomena in hashing algorithms. We have
briefly   evoked  the   loosely related  case   of  P\'olya   urn   models  in
Section~\ref{mult-sec}.  Another highly  interesting  instance is  the
modelling  of  the  exclusion  process  on an  interval  (PASEP,
TASEP, \ldots)  by  means  of
dedicated   operators  $D,E$   satisfying  the   commutation  relation
$DE-ED=D+E$, as discovered by Derrida \emph{et al.}~\cite{DeEvHaPa93}.
There is currently an  extensive literature dedicated to combinatorial
aspects  of this  process;  see,  typically,  the  works of  Corteel
\emph{et al.}~\cite{BrCoEsPaRe06,CoJoPrRu09},  Viennot~\cite{Viennot09}, and references
therein to earlier works.

There  are finally a  few areas  that we  left largely  unexplored and
which would deserve further  consideration. One concerns the r\^ole of
eigenfunction   expansion,  of   which   we  offered   a  glimpse   in
Section~\ref{maindef-sec},  when briefly  discussing  the heat  kernel
through eigenfunctions of the operator $e^{tD^2}$. Another is relative
to  further  connections with  orthogonal  expansions and  orthogonal
polynomials, otherwise known to be tightly coupled with the continued
fractions of our Section~\ref{cf-sec}.

\medskip

\begin{small}
\noindent
{\bf Acknowledgements.}
Both authors express  their appreciation of the warm hospitality
and support of the Mathematisches Forschungsinstitut Oberwolfach
where this research was started in 2008.
P.B. acknowledges support from the Polish Ministry of Science and Higher Education
under Research Grant No. N202 061434.
The work of  {P.F.} was supported in part by the  {\sc Boole} and {\sc
  Magnum} Projects  of the French National Research  Agency (ANR).  We
thank  G\'erard  Duchamp,  Andrzej  Horzela, Miguel M\'endez, Karol A. Penson, and  Allan I.
Solomon for many informative  and supportive exchanges relative to our
works.   We are  also grateful  to Christian Brouder, Sergey Fomin,
Mathieu  Josuat-Verg{\`e}s, and Xavier  Viennot  for   encouraging  comments  and  relevant
bibliographic references. Thanks finally to Morteza 
Mohammad Noori for drawing Scherk's thesis  to our attention.

\par
\end{small}

\newpage

\bibliographystyle{acm}
\bibliography{algo}

%
%
%
%


\newpage

\begin{appendix}\label{scherk-apdx}

\newtheorem{aprop}{Proposition}[section]

\section{\bf Scherk's dissertation of 1823}\label{scherk-ap}

\def\ox{X}

\def\XX{{\cal X}}

In     this      appendix,     we    examine      the     ``inaugural
dissertation''~\cite{Scherk23}    which    Heinrich   Ferdinand
[Henricus Ferdinandus]  {\sc Scherk}
defended in 1823 at the  University of Berlin.
The dissertation is in Latin, and,  as this language may not be easily
understood by the younger generation, we offer here a brief account of
its contents. In modern terms,  Scherk's thesis can be seen as dealing
with  the  reduction to  normal  form of  an  expression  of the  form
$(yD)^n$,  where  various  choices   for  the  function  $y=y(x)$  are
considered. The evocative title of the dissertation is, from its front
page:
\begin{center}
DE\\ {\large\bf EVOLVENDA FUNCTIONE} \\[2truemm]
$\ds\bf  \frac{yd.~yd.~yd\ldots yd{\XX}}{dx^n}$\\[2truemm]
\em\bf\em DISQUISITIONES NONNULLAE ANALYTICAE.
\end{center}
That is, \emph{``Some analytic investigations 
relative to the expansion of the function $(yD)^n {\XX}$''}.
Throughout this section,
the calligraphic  ${\XX}={\XX}(x)$ is taken to represent an arbitrary function,
not to be confused\footnote{This is meant to preserve
the correspondence
with Scherk's original notations.}
with the operator~$X$.

The memoir is comprised of 10 sections (numbered $\bf1$ to $\bf 10$)
and the body of the text consists of 36 pages, It is supplemented by a
two page table of  values of Stirling numbers,  followed by a two page
Vita~\cite[pp.~36-37]{Scherk23}, where  we learn that  Scherk was born
on  October 27, 1797  in the (now Polish)   city of Pozna\'n. It concludes
with a  page  containing four  assertions  (called \emph{``theses''}),
apparently  of the author's own  design,  amongst which the second one
reads

\medskip

{\em\bf\em Calculi, quos dicunt superiores, in Algebra elementari ponendi sunt.}

\smallskip
\noindent
A rough rendition is: \emph{``What  is known as the higher calculus is
  to be  placed within the  framework of elementary  algebra''}.  This
excellently  summarizes the  philosophy underlying  Scherk's  work (as
well as ours!).

We are very much indebted
to Dr Morteza Mohammad Noori~\cite{Noori10}, from the University of Tehran,
for drawing our attention to Scherk's thesis. An electronic version can
be found under the rich and
highly useful {\sc gdz} site (G\"ottinger Digitalisierungszentrum),
\begin{center}
\tt http://gdz.sub.uni-goettingen.de/\,,
\end{center}
from which a large portion of the nineteenth century
 mathematical literature in German is available.
The {\sc MacTutor} reference site for mathematicians' biographies
has an extended bibliographic notice on Scherk, which can be supplemented by the
dedicated  site hosted by the University of Halle. The url's are:
\begin{quote}
{\tt http://www-history.mcs.st-and.ac.uk/}
\\ {\tt http://cantor1.mathematik.uni-halle.de/history/scherk/}.
\end{quote}
(There   we  learnt  that   Scherk's  later   discovery of  the  third
non-trivial  examples  of a minimal   surface brought him considerable
fame. He  appears to have been otherwise  interested in  Bernoulli and
secant numbers, as well  as in the  number of combinations with bounded
repetitions.)   The Mathematical   Genealogy  Project  indicates  that
Scherk had 19728 descendants (via  his famous student, Ernst  Kummer),
as of July 2010.

\bigskip
\begin{center}
$ \sim\sim\sim\sim\sim\sim\sim\sim$
\end{center}

\bigskip\noindent
{\bf\S 1.} Scherk first notes (p.~1)
that the series previously considered 
mostly involve an iterated use of arithmetic operations,
especially, multiplication and division.
In this category, we find series such as
\[
\sum f_n x^n, \qquad \sum_n f_n x(x-1)\cdots(x-n),\qquad
\sum_n f_n \frac{1}{x(x+1)\cdots(x+n)},
\]
which are of Taylor, Newton-interpolation, and factorial type, respectively---these
had  all been investigated
by his time.  Scherk then notes that little is known if
the formation law  of the general term of  a series involves arbitrary
differentiation and integration operations.  The most general problem,
he notes, exceeds by far his own strength. Accordingly, he focuses his
attention to the special case
\[
\frac{(yd)^n {\XX}}{dx^n},
\]
which,  in  our notations,  
is $(yD)^n {\XX}$, where $y=y(x)$ and ${\XX}={\XX}(x)$.  He  then goes  on to
explain (p.~2) that his purpose will be to express the general term as
a function of the derivatives~${\XX},{\XX}',\ldots$. In our terminology,
this is equivalent to seeking \emph{normal forms}, where 
all the occurrences of the derivative~$D$ appear
on the right. 

\medskip\noindent
{\bf\S 2.} The problem to be considered now is the case $y=x$.
A small table (p.~2),
\[
\begin{array}{lll}
(xD){\XX}  &=& xD{\XX} \\
(xD)^2{\XX}&=&xD{\XX}+x^2D^2 {\XX}\\
(xD)^3{\XX}&=&xD{\XX}+3x^2D^2{\XX}+x^3D^3{\XX},
\end{array}
\]
suggests an interesting law. Indeed, the operator formula
(with $D\equiv \frac{d}{dx}$ and~$\ox$ representing 
as usual multiplication by~$x$; i.e., $\ox f(x):=xf(x)$)
\begin{equation}\label{eq0}
(\ox D)^n = \sum_{j=1}^n a_j^n \ox^jD^j
\end{equation}
serves to define  a (yet unknown) array of numbers $\{ a_j^n\}$.
(Scherk writes $\ds\mathop{a}_j^n$ for what we denote by $a_j^n$)

\newcommand{\mi}[1]{\begin{minipage}{4.5truecm}\ \\\footnotesize#1
\par\ \end{minipage}}
\newcommand{\ce}[1]{\multicolumn{1}{c}{#1}}

\begin{figure}\small
\renewcommand{\arraystretch}{1.6}
\begin{center}
\begin{tabular}{llll}
\hline\hline
\ce{\em Scherk} & \ce{\em definition} & \ce{\em modern} & 
\ce{\em properties}\\
\hline 
$a_k^n$ & \mi{defined for \S2--3 as coefficients  in the normal form of  $(XD)^n$,
cf   Eq.~\eqref{eq0}}   &   $\ds{h+k   \brace  k}$   & \mi{equivalence with Stirling${}_2$  numbers
is granted via Eq.~~\eqref{alt}, assuming as known the  alternating sum expression  of
Stirling${}_2$ numbers}
\\
${C'}_k^h$ & \mi{defined, as in Eq.~\eqref{Cp} below, as elementary symmetric function of degree~$h$
in the integers $1,2,\ldots,k$; see~\cite[p.~4]{Scherk23}} & $\ds{k+1 \brack h+1}$ & \mi{equivalence 
via the (now) classical generating function of Stirling${}_1$ numbers}
\\
${'C}_k^h$ & \mi{defined, as in Eq.~\eqref{pC} below, as 
complete homogeneous symmetric function of degree~$h$
in the integers $1,2,\ldots,k$; see~\cite[p.~5]{Scherk23}} & $\ds{h+k \brace k}$ &
\mi{equivalence 
via the (now) classical generating function of Stirling${}_2$ numbers}
\\
\hline\hline
\end{tabular}
\end{center}
\caption{\small\label{corr-fig} A correspondence table for \S1--7:
Scherk's notations, 
his definitions, the corresponding modern notations,
and corresponding properties.}
\end{figure}

By considering $(xD)^n\XX=(xD)\, (xD)^{n-1}{\XX}$,
one easily obtains the recurrence (p.~3)
\begin{equation}\label{rec0}
a^n_k=a^{n-1}_{k-1}+k a^{n-1}_{k}.
\end{equation}
Scherk then derives (p.~4) the (by now classical) alternating sum formula
\begin{equation}\label{alt}
a^n_{k}=\frac{1}{(k-1)!} \sum_{h=0}^{k-1}
\binom{k-1}{h-1} (-1)^{k-h} h^{n-1},
\end{equation}
his notations being $
\Pi(k)=k!$ and $P_y^x =\binom{y}{x}=\frac{y!}{x!\, (y-x)!}$.

Next (bottom of p.~4 and top of page 5), Scherk
proceeds to introduce two types of numbers\footnote{
We continue using more modern notations and write $C_k^h$
for what Scherk denotes by $\ds \mp{C}_k^h$.}
${C'}_k^h $ and $'C_k^h$, s follows.
\begin{itemize}
\item[---]
The quantity ${C'}_k^h $ is defined as the sum of combinations \emph{without
repetitions} of $h$ of the integers $1,2\ldots,k$; in symbols,
\begin{equation}\label{Cp}
{C'}_k^h =\sum_{1\le j_1<j_2<\cdots<j_h\le k} j_1j_2\cdots j_h.
\end{equation}
\item[---]
The quantity ${'C}_k^h $ is similarly defined as a sum of combinations
\emph{with repetitions}; in symbols,
\begin{equation}\label{pC}
{'C}_k^h =\sum_{1\le j_1\le j_2\le \cdots\le j_h\le k} j_1j_2\cdots j_h.
\end{equation}
\end{itemize}

\smallskip

\noindent\begin{footnotesize}
{\bf Remarks.} The definitions of the number arrays $'C,C'$
are nowadays easily seen to be equivalent to
the generating function expressions
\[
\sum_h {C'}_n^h x^h=(x+1)\cdots (x+n),
\qquad 
\sum_h {'C}_n^h x^h=\frac{1}{(1-x)(1-2x)\cdots(1-nx)},
\]
where one can recognize variants of Stirling numbers~\cite{Comtet74}.
\end{footnotesize}

\smallskip

Scherk then  gives (p.~5) two recurrences
\begin{equation}\label{recs12}
{C'}_k^h={C'}_{k-1}^h+k{C'}_{k-1}^{h-1};
\qquad
{'C}_k^h={{'C}}_{k-1}^h+k{'C}_k^{h-1}.
\end{equation}
and he concludes (p.~6) from a comparison of the recurrences 
\eqref{rec0} and~\eqref{recs12} that
\[
a_{n,k}={'C}_k^{n-k},
\]
which provides the normal form of $(\ox D)^n$  as
(his notations, p.~6):
\begin{equation}\label{nor1}
\frac{(xd)^n {\XX}}{dx^n}=\sum_{k=1}^n {'C}_k^{n-k}
x^k \frac{d^k {\XX}}{dx^k}.
\end{equation}
We state\footnote{As it was customary in his time, Scherk does not
provide typographically marked statements, such as theorems
or propositions, but rather plainly quotes
his main results \emph{en passant}.}, with modern notations:
\begin{aprop}[{Scherk~\cite[p.~6]{Scherk23}}]
The normal form of $(\ox D)^n$ is given by
\[
(\ox D)^n=\sum_{k=1}^n {n \brace k} \ox^kD^k,\qquad n\ge1.
\]
\end{aprop}

\medskip\noindent
{\bf\S3.} Scherk continues (pp.~7--8)
with various considerations related to setting ${\XX}=-x(1-x)^{k-1}$
followed by the substitution $x=1$. His goal is to obtain a new proof
of the normal form result~\eqref{nor1}. He also rederives in this framework
the relation~\eqref{alt} directly from the definition of the $a_k^n$.

\medskip\noindent
{\bf\S4.} We now come to another problem, namely, that
of the reduction of
\[
(e^xD)^n {\XX}.
\]
The coefficients\footnote{
In fact, Scherk also denotes the new coefficients by $a_{n,k}$;
we adopt a different convention for clarity.} $b_k^n$ are determined by
the normal form (p. 9):
\begin{equation}\label{inv0}
(e^xD)^{n}{\XX}=e^{nx}\sum_{k=1}^n b_k^n D^k {\XX}
\end{equation}
Scherk then derives a recurrence to be compared to our Equation~\eqref{recs12},
from which there results (p.~10) that
\[
b_k^n={C'}_{n-1}^{n-k}.
\]
Thus the $b_k^n$ are now recognizable as Stirling numbers of the first kind.
We state:
\begin{aprop} \label{rectree-aprop}
The normal form of $(e^xD)^n$ is given by
\[
(e^x D)^n = e^{nx}\sum_{k=1}^n {n\brack k} D^k.
\]
\end{aprop}
The section concludes with representations of Stirling${}_1$ numbers in
terms of (sort of) generalized harmonic numbers.

\medskip\noindent
{\bf \S5.} This section digresses from the main course of the dissertation
(p.~14) and is admittedly only loosely related to the other sections.
The problem consists in finding  the laws 
of the derivatives of a composite function.
The formula, commonly attributed to Fa\`a di Bruno (born in 1825(!), publication
dated 1850),
figures explicitly (p.~14);
see~\cite{Comtet74} for a statement. Scherk 
says:

\smallskip
\begin{center}
\begin{small}
\begin{tabular}{ccc}
\begin{minipage}{5.5truecm}
{expressionem quamdam invenimus, qualem frustra in multis libris jam seapius
quaesiveramus, quam itaque paucis verbis hic tangere liceat.}
\end{minipage}
& \qquad &
\begin{minipage}{5.5truecm}
{\em we have discovered a certain formula, which we have very often  tried in vain to 
find in books,
so that we feel allowed to devote a few words to its description.}
\end{minipage}
\end{tabular}
\end{small}
\end{center}
\smallskip

\noindent
We refer to the scholarly study of Warren Johnson~\cite{Johnson02}, regarding the history of this formula,
which was once considered a basic component of combinatorial analysis. Nowadays,
via generating functions,
the computation of the $n$th derivative
of a composite function $f(g(z))$ amounts to nothing 
but the determination of a coefficient,
\[
n!\, [z^n] \sum_{j=0}^n \frac{f_j}{j!}\left(\sum_{k=1}^n g_k \frac{z^k}{k!}\right)^j;
\]
i.e., it is a simple  avatar of the multinomial formula.

\medskip\noindent
{\bf \S6.} Scherk now considers (p. 15) the \emph{inverse problem}:
how to express the derivatives $D^nX$ in terms of the quantities $(xD)^nX$ and 
$(e^xD)^n{\XX}$. His main formula is (p.~18)
\[
(X^nD^n) {\XX} =\sum_{h=0}^{n-1} (-1)^h {C'}_{n-1}^h (xD)^{n-h} {\XX}.
\]
and a parallel one serves to express $e^{nx}D^n$ in terms of the powers $(e^x D)^j$,
this time by means of the Stirling${}_2$ numbers (his $'C$).
In retrospect, these two formulae are manifestations of the fact that the matrices
of Stirling${}_1$ and Stirling${}_2$ are inverse of each other.
Scherk can then conclude that the coefficients $C'$ serve to make explicit
the law of the successive
derivatives of $1/(x\log x)$ (p.~18) and $\log\log x$ (p.~19, top):
\[
\frac{d^n}{dx^n}\log\log x =
\frac{(-1)^{n-1}}{x^n} \left(
\sum_{k=0} ^{n-1} {n \brack k+1} \frac{k\,!}{(\log x)^{k+1}}\right).
\]

\medskip\noindent
{\bf\S7.} This section gives two further consequences (p.~20)
of previous developments including a formula which expresses the expansion of a function 
$g(z)$ in terms of $\log z$, when $z$ is near~1.

\def\uu{\upsilon}

\medskip\noindent
{\bf\S8.} The purpose is now to  come back to the  general problem (p.~20) of the
reduction of $(yD)^n$.  Naturally, large symbolic  expressions
are available,  for   small~$n$,  involving  $y,y',y'',\ldots$   (The
notation  used  in  the   memoir   is  $\ds\mathop{y}_i=y^{(i)}$.) 
A typical formula (p.~20) reads:
\[
\frac{(yd)^4{\XX}}{dx^4}=\left(y{y'}^3+4y^2y'y''+y^3y'''\right) {\XX}'+
(7y^2y'^2+4y^3y''){\XX}''+6y^3y'{\XX}'''+y^4{\XX}''''.
\]
A
discussion  of the  combinatorics  of the  coefficients occupies the
next few pages (pp.~21--27). Naturally, the previously considered
cases $y=x$ and $y=e^x$ provide various checks.
The corresponding coefficients are denoted by $\uu_n^k$
(written as usual as $\ds\mathop{\uu}_n^k$ in the memoir). Scherk deduces various \emph{qualitative}
results on the shape of a general expansion: 
he obtains exact values for the first few coefficients as well as some
infinite classes of particular coefficients. He then concludes
regarding a possible discovery of the general law:
. 
\smallskip
\begin{center}
\begin{small}
\begin{tabular}{ccc}
\begin{minipage}{5.5truecm}
{Solutio itaque ideo tam difficilis facta est, quod disquisitionem coefficientium
numericorum ab inventione singulorum terminorum ipsorum segregare non potuirimus.}
\end{minipage}
& \qquad &
\begin{minipage}{5.5truecm}
{\em Then, the process has become  so untractable that we could not
succeed with an  investigation of  all  the numerical  coefficients,
based on the sole discovery of some individual terms.}
\end{minipage}
\end{tabular}
\end{small}
\end{center}
\smallskip

\begin{figure}\small

\begin{center}
\fbox{\Img{7.5}{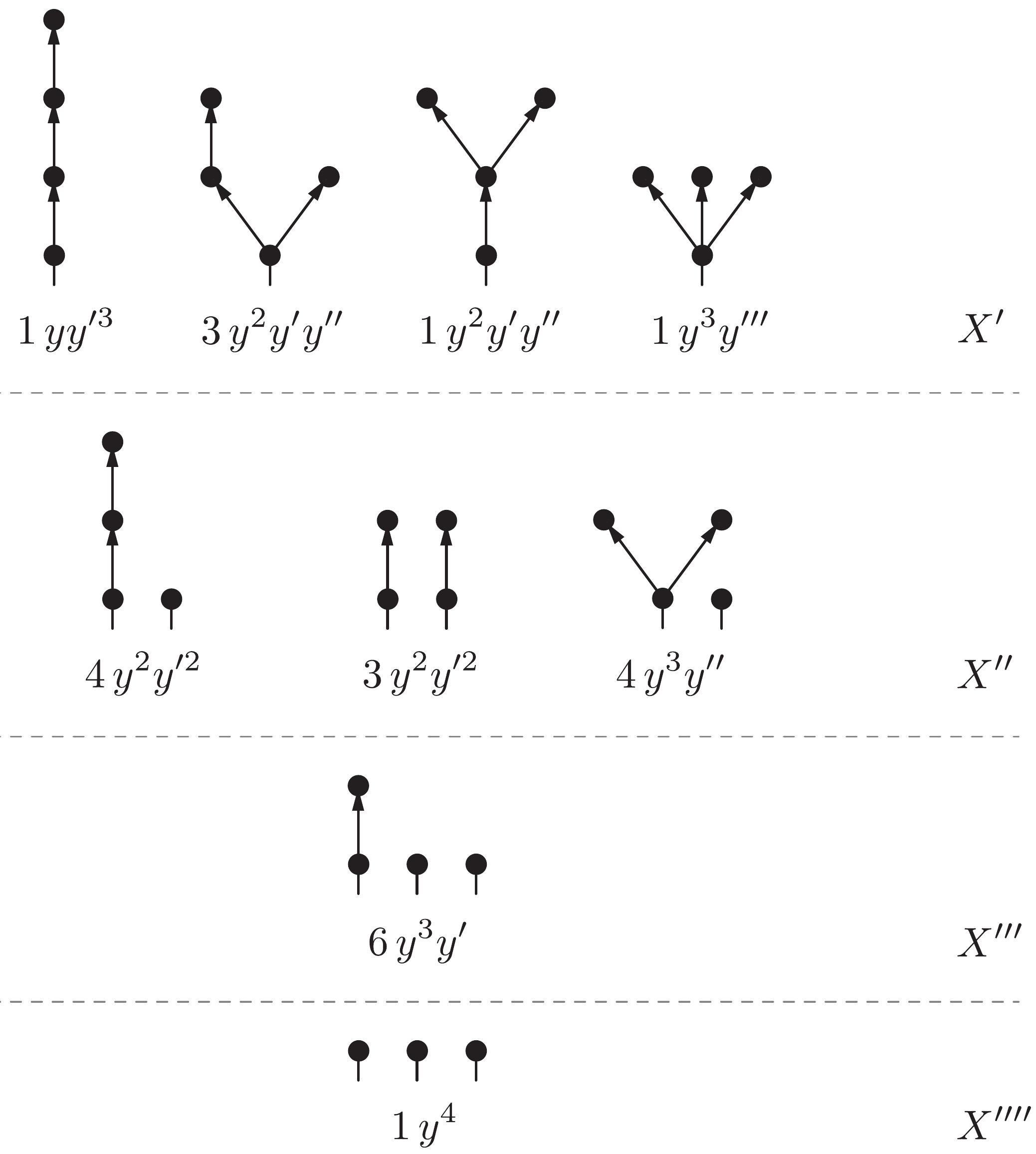}}\qquad
\Img{2.2}{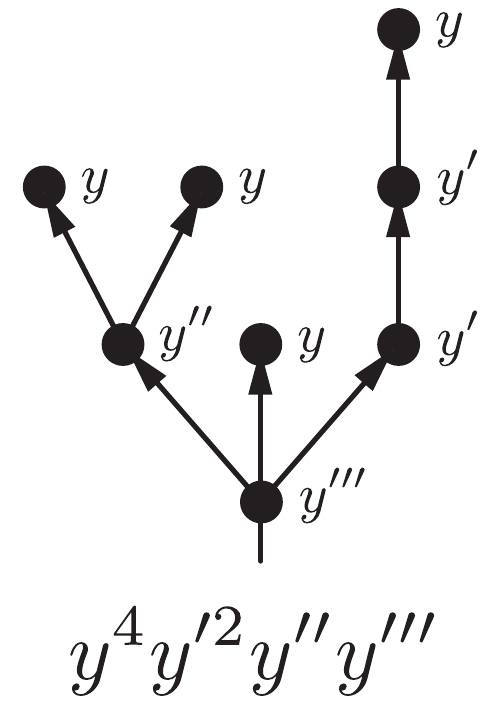}
\end{center}

\caption{\label{scherktree-fig}\small
\emph{Left}: An illustration of the correspondence between Scherk's expansion
of $(yD)^n{\XX}$ and forests of trees (case $n=4$).
\emph{Right}: the decoration of a particular tree by the derivatives of~$y$}
\end{figure}

\medskip

\noindent
{\footnotesize
{\bf Remarks.} \label{scherk8-pg}
The symbolic problem is indeed \emph{difficult}. From the solution described earlier, in 
our Section~\ref{semilin-sec}, we now know the following.
\begin{itemize} 
\item[$(i)$] The coefficient of ${\XX}^{(r)}$ in~$(yD)^n{\XX}$ is obtained from
the set of unordered $r$-forests of (rooted, non plane)
increasing trees having a total size of~$n$,
by assigning to a vertex of outdegree~$j$ a weight equal to~$y^{(j)}$; see Figure~\ref{scherktree-fig}.
(Trees are taken ``unordered'' (i.e., non-plane)
owing to the fact that $y(x)$ is here an exponential generating function.)
\item[$(ii)$] Symbolically, one should proceed as follows, for an expansion of $(yD)^n{\XX}$:
\begin{itemize}
\item[---] Start from $y(x)=\sum_j y_j x^j/j!$, so that $y_j$ plays the r\^ole of $y^{(j)}$.
\item[---] Compute $1/y(x)$, which involves a series of multinomial expansions.
\item[---] Integrate the previous expansion, which only involves a substitution $x^n\mapsto x^n/n$ for each~$n$: we get in this way $\Phi:=\int y^{-1}$.
\item[---] Take the compositional inverse of the preceding expansion~$\Phi$---this, by the Lagrange
inversion theorem only involves a further sequence of multinomial expansions. (We 
determine here the generating function~$T(x)$ of increasing trees.)
\item[---] Raise $T(x)$ to the $r$th power (for the coefficient of ${\XX}^{(r)}$),
extract  the  coefficient of  $x^n$
in $\frac{1}{r!}T^r$, and finally   multiply this coefficient by~$n!$
since we are dealing with exponential generating functions.
\end{itemize}
\end{itemize}

\noindent
Barely half a dozen instruction suffice to implement the general procedure
in a symbolic manipulation system such as {\sc Maple}. We can in this way verify, in a 
fraction of a second, the correctness of the expansion relative to $(yD)^5{\XX}$, 
in the form given
by Scherk (p.~20). For the benefit of the curious reader, here  is the outcome of 
our program in the case of $(yD)^6{\XX}$:
\[\begin{array}{l}
\quad\,\, {\XX}_{{1}} \left( 32\,y_{{3}}{y_{{0}}}^{3}{y_{{1}}}^{2}+34\,{y_{{2}}}^{2}
{y_{{0}}}^{3}y_{{1}}+y_{{0}}{y_{{1}}}^{5}+26\,y_{{2}}{y_{{0}}}^{2}{y_{
{1}}}^{3}+11\,y_{{1}}y_{{4}}{y_{{0}}}^{4}+15\,y_{{3}}{y_{{0}}}^{4}y_{{
2}}+y_{{5}}{y_{{0}}}^{5} \right) 
\\{} +{\XX}_{{2}} \left( 34\,{y_{{0}}}^{4}{y_{
{2}}}^{2}+31\,{y_{{0}}}^{2}{y_{{1}}}^{4}+146\,{y_{{0}}}^{3}y_{{2}}{y_{
{1}}}^{2}+57\,{y_{{0}}}^{4}y_{{1}}y_{{3}}+6\,{y_{{0}}}^{5}y_{{4}}
 \right) 
\\{} +{\XX}_{{3}} \left( 90\,{y_{{0}}}^{3}{y_{{1}}}^{3}+120\,{y_{{0}}}
^{4}y_{{1}}y_{{2}}+15\,{y_{{0}}}^{5}y_{{3}} \right) 
\\{} +{\XX}_{{4}} \left( 65
\,{y_{{0}}}^{4}{y_{{1}}}^{2}+20\,{y_{{0}}}^{5}y_{{2}} \right) ~+~ 15\,{\XX}_{
{5}}{y_{{0}}}^{5}y_{{1}} ~+~{\XX}_{{6}}{y_{{0}}}^{6}.
\end{array}
\]
}

\medskip
{\bf\S9.} This section starting p.~27 contains a further combinatorial
investigation of the general case, in the light of the special 
cases considered earlier in~\S2--4. 
Near the end, the author offers a brief discussion (pp.~29--30)
of the particular reduction of $(x^pD)^n$. What he obtains is,
in present day notations (we write $c_{n,k}$ for coefficients,
which Scherk once more denotes by $a_n^k$):

\begin{aprop}[{Scherk~\cite[pp.~30--31]{Scherk23}}] \label{scherkXrD}
The normal form of $(X^pD)^n$ is given by
\[
(X^pD)^n=X^{n(p-1)}\sum_k c_n^k  X^kD^k,
\]
where the coefficients~$c_n^k$ satisfy
\begin{equation*} 
c_n^k =\!\!\!
\sum_{1\le j_1\le j_2\le \cdots\le j_{n-k}\le k}  \!\!\! \!\!\!
[j_1p]\cdot [(j_2+1)p-1]\cdot [(j_3+2)p-2]\cdots
[(j_{n-k}+n-k-1)p-(n-k-1)].
\end{equation*}
\end{aprop}

\medskip
{\bf\S10.} This section (pp.~31--36) contains formulae that the author obtained,
not being at the time cognizant of similar works by the Bernoullis and Laplace.
They are relative to Bernoulli numbers, hence they will not be discussed 
further here.

\end{appendix}

\end{document}